\documentclass[10pt,a4paper, envcountsame]{amsart}

\usepackage[margin=1in]{geometry}

\usepackage{multirow}
\usepackage{amsmath}
\usepackage{amssymb}
\usepackage{amsthm}
\usepackage{enumitem}
\usepackage{url}
\usepackage{MnSymbol}

\usepackage{graphicx}
\usepackage{changepage}

\usepackage{hyperref}

\usepackage[all,cmtip]{xy}
\usepackage{tikz}
\usetikzlibrary{arrows,snakes,positioning,backgrounds,shadows}
\usepackage{pgfplots}

\newtheorem{theorem}{Theorem}[section]
\newtheorem{lemma}[theorem]{Lemma}

\newtheorem{corollary}[theorem]{Corollary}
\newtheorem{proposition}[theorem]{Proposition}
\newtheorem{claim}[theorem]{Claim}
\newtheorem{fact}[theorem]{Fact}

\theoremstyle{definition}
\newtheorem{definition}[theorem]{Definition}
\newtheorem{example}[theorem]{Example}

\newtheorem*{question*}{Question}

\theoremstyle{remark}
\newtheorem{remark}[theorem]{Remark}

\newcommand{\mc}[1]{\mathcal{#1}}

\newcommand{\res}{\upharpoonright}

\newcommand{\Dy}{\mathtt{Dy}}
\newcommand{\Shat}{\mathtt{\hat{S}}}
\newcommand{\outl}{\mathtt{out_L}}
\newcommand{\outr}{\mathtt{out_R}}

\newcommand{\st}{\;|\;}

\setenumerate[1]{label={(\arabic*)}}

\usepackage{xfrac}

\usepackage{xcolor}

\makeatletter
\newcommand{\dbloverline}[1]{\overline{\dbl@overline{#1}}}
\newcommand{\dbl@overline}[1]{\mathpalette\dbl@@overline{#1}}
\newcommand{\dbl@@overline}[2]{%
  \begingroup
  \sbox\z@{$\m@th#1\overline{#2}$}%
  \ht\z@=\dimexpr\ht\z@-2\dbl@adjust{#1}\relax
  \box\z@
  \ifx#1\scriptstyle\kern-\scriptspace\else
  \ifx#1\scriptscriptstyle\kern-\scriptspace\fi\fi
  \endgroup
}
\newcommand{\dbl@adjust}[1]{%
  \fontdimen8
  \ifx#1\displaystyle\textfont\else
  \ifx#1\textstyle\textfont\else
  \ifx#1\scriptstyle\scriptfont\else
  \scriptscriptfont\fi\fi\fi 3
}
\makeatother

\newcommand{\sgnd}[1]{\dbloverline{#1}}

\begin{document}
\sloppy

\title[Computable classifications of continuous, transducer, and regular functions]{Computable classifications of continuous,\\ transducer, and regular functions}

\author[J.\ Franklin]{Johanna N.Y.\ Franklin}
\address{Hofstra University}
\email{johanna.n.franklin@hofstra.edu}

\author[R.\ H\"olzl]{Rupert H\"olzl}
\address{Universität der Bundeswehr M\"unchen}
\email{r@hoelzl.fr}

\author[A.\ Melnikov]{Alexander Melnikov}
\address{Massey University}
\email{alexander.g.melnikov@gmail.com, a.melnikov@massey.ac.nz}

\author[K.M.\ Ng]{\\Keng Meng Ng}
\address{Nanyang Technological University}
\email{kmng@ntu.edu.sg}

\author[D.\ Turetsky]{Daniel Turetsky}
\address{Victoria University of Wellington}
\email{dan.turetsky@vuw.ac.nz}

\thanks{We thank Bakhadyr Khoussainov for his valuable help with the literature about automata, and specifically for bringing \cite{vardi} to our attention.\\
Franklin was supported in part by Simons Foundation Collaboration Grant \#420806, and Ng was partially supported by the MOE grant RG23/19.}

\maketitle

\begin{abstract} We develop a systematic algorithmic framework that unites global and local classification problems using index sets. We prove that the classification problem for continuous (binary) regular functions among almost everywhere linear, pointwise linear-time Lipschitz functions is $\Sigma^0_2$-complete. (Every regular function is pointwise linear-time Lipschitz.) We show that a function $f\colon [0,1] \rightarrow \mathbb{R}$ is (binary) transducer if and only if it is continuous regular.   As one of many consequences, our $\Sigma^0_2$-completeness result covers the class of transducer functions as well. Finally, we show that the Banach space $C[0,1]$ of real-valued continuous functions admits an arithmetical classification among separable Banach spaces. Our proofs combine methods of abstract computability theory, automata theory, and functional analysis.

\end{abstract}

\tableofcontents

\section{Introduction}

One of the primary desiderata for a classification of a collection of mathematical objects is that it make the objects in question easier to understand or manipulate algorithmically. 
Here, we use tools from computability theory and automata theory to measure the complexity of three classification problems, all of which are formulated for the Banach space of continuous functions on the unit interval, $C[0,1]$:

 \begin{enumerate}
  \item[(Q.1)] Is there a simple characterization of the class of continuous regular functions?

 \item[(Q.2)] Which  $f \in C[0,1]$ are transducer?

 \item[(Q.3)] Given a Banach space $\mc B$, how hard is it to decide whether  $\mc B \cong C[0,1]$?
 \end{enumerate}
 We will clarify all the terms used in the questions in due course. 
Chaudhuri, Sankaranarayanan, and Vardi~\cite{vardi} were the first to raise and attack the first question, and Block Gorman et al~\cite{reeeg} have recently studied the question in great depth.   According to \cite{MullerTransducer}, the study  of transducer functions  can be traced back to Ercegovac and
Trivedi~\cite{triverdi},  who initiated the systematic theoretical study of on-line arithmetic.
Partial solutions to the second question can be found in, e.g.,~\cite{MullerTransducer, Lisovik1}.
As far as we know,  Brown~\cite{BrownThesis} was the first to attack the third question.  

All our results are concerned with computability-theoretic aspects of $C[0,1]$.
Note that the first two of these questions are local in the sense that they are concerned with elements of $C[0,1]$ while the last is global in the sense that the space itself is the object of study.
Perhaps unexpectedly, both local and global problems of this kind admit a unified algorithmic approach via \emph{index sets}; before we formally state our results we will discuss this approach in each circumstance.  Our proofs exploit a  wide variety of algorithmic approximation techniques.  We will 
approximate continuous functions using an automaton, a Turing machine, and an oracle Turing machine to attack (Q.1), (Q.2), and (Q.3), respectively. 
 These approximation techniques are relatively combinatorially involved but  our arguments are self-contained, so  not much specific background is required to understand the proofs.

\smallskip

The reader may find some aspects of our proofs and results unexpected. For instance,
our attack on (Q.1) resulted in discovering that a continuous function is regular if and only if it is transducer via a purely automata-theoretic/analytic argument. Our initial expectation was that these classes should be different as they are distinguished in the literature, and that this difference should be (for instance) reflected in their index sets. It is not completely uncommon in the study of index sets to discover new \emph{positive} classification-type results, but such results are rare; see, e.g, \cite{DoMel1,leb}.
 Further, our index set result for (Q.2) (and thus for (Q.1)) entails that none of the many currently known properties of continuous regular (transducer) functions can possibly help in their classification from the algorithmic standpoint. 
 Finally,  we prove that (Q.3) is \emph{arithmetical}, which seems to contradict one's expectation: recall that $C[0,1]$ is a universal separable space with fairly intricate properties. Even though the latter is a global classification-type result, it is also an index set result about $C[0,1]$. Furthermore, its proof is in fact highly ``local'' in the sense that we use oracle computation to approximate a (Schauder) basis of $C[0,1]$ in stages based on a special arithmetical notion of independence for finite sets of functions.

\medskip

To formally state our results, we need to define several computability notions for continuous functions.

  \medskip

\subsection{Computable continuous functions}\label{subs:comp}
Computability in the space $C[0,1]$ plays an important role in computable analysis since it serves as an excellent playground for testing various theories and methods. Some such results have already been mentioned above. For classical results using abstract Turing computability, see \cite{PRich,Myhill} and the book~\cite{Wei00}, and for some recent results relating algorithmic randomness and differentiability, see ~\cite{Nies:2010,BHMN,BMN}. 
Our article makes contributions to this classical line of study of algorithms in $C[0,1]$.

 We will take the dense subspace of polynomials over $\mathbb{Q}$ to be a computable structure on $C[0,1]$.
According to the classical definition, a function $f\colon [0,1] \rightarrow \mathbb{R}$ is computable if there is an effective procedure  which, on input $s$, outputs a tuple of rationals~$\langle q_0, \ldots, q_n \rangle$ such that $\sup_{x \in [0,1]} \{| f(x) - \sum_{i=0}^n q_i x^i |\} < 2^{-s}$.
Note that computability implies continuity.


 An adaptation of polynomial-time computability to continuous functions has previously been studied; the foundations can be found in \cite{Ko}.  However, it is not necessarily clear that polynomial-time computability is the right notion of efficiency when it comes to infinite objects; see \cite{BDKM:19, Roughgarden:19} for a discussion. It is thus not surprising that other models of efficient computability for members of~$C[0,1]$ have been tested as well; we will be focused on two such notions, both of which use finite state automata. There are two distinct types of automata: The first kind is finite-state automata (FSA) in which the machine either accepts or rejects a given (finite length) input, and the second and more general kind is finite-state transducers (FST), which produce an output string based on the given input. These can be adapted to give rise to the following two notions of automatic real analysis.

\smallskip

\noindent {\bf Transducer computability.}  Since real numbers are necessarily infinite objects, we cannot store one as an argument in finite time or space. The traditional solution to this problem is to approximate each real number by a sequence of finite objects and only have a full description of the given real number in the limit. This gives rise to different possible representations of the real numbers, which will be discussed later. In order to define transducer computability for real functions, we will have to consider finite-state transducers that work with an infinite input and output.

Informally, a function $f$ is transducer if there exists a finite state memoryless machine (an automaton) which reads
 finite chunks of the binary representation of $x$ and gradually outputs more and more of the binary representation of $f(x).$ Of course, this definition can be extended and modified to any representation and to nondeterministic automata (etc.)~\cite{Lisovik3,Lisovik2} as well as to higher dimensions of the domain~\cite{Kon}. We will adopt a transducer model that allows for a fixed finite delay, i.e., one in which the transducer is only required to begin writing its output after scanning the first $D$ many input bits, where $D$ is independent of the input.

 It is not hard to see that transducer  functions are continuous; we present a proof in Proposition ~\ref{cont}.
The problem of characterizing transducer functions, stated in (Q.2) above, has been central to this topic since the beginning.
 Muller \cite{MullerTransducer} showed that if $f$ is transducer and $f'$ is piecewise continuous, then $f$ must be piecewise linear with rational parameters.  Lisovik and Shkaravskaya \cite{Lisovik1} extended this result to the case in which $f'$ exists but is not necessarily piecewise continuous and, more recently, Konecny \cite{Kon} extended these results to functions in higher dimensions. Nonetheless, it is not hard to produce complex examples of transducer functions, including  nowhere differentiable ones; see, e.g., Theorem 4 of \cite{Lisovik1} and the example in \cite{vardi}. As it stands, the classification problem (Q.2)
for transducer functions remains unresolved. Formal definitions and further discussion will appear in Section \ref{sec:transducer}.

\smallskip

\noindent {\bf Regular functions.} Chaudhuri, Sankaranarayanan, and Vardi \cite{vardi} initiated the notion of regular real analysis using an adaptation of the FSA. As mentioned above, we have to consider a version of the FSA which works with infinite words. They suggested using nondeterminisitic B\"{u}chi automata to define the related notion of a regular function.

Informally, a function $f\colon [0,1]\rightarrow \mathbb{R}$ is regular if there exists a B\"{u}chi automaton on two tapes which accepts the graph of $f$; that is, there is a finite state memoryless machine (an automaton) which simultaneously reads two (say, binary) representations of two reals $x$ and $y$; we have $f(x) =y$ if and only if the automaton visits the accepting state infinitely often. Although regularity does not even imply continuity in general (see \cite{vardi}), one can show, using results from \cite{vardi}, that every regular \emph{continuous} function is in fact computable. In recent work, Block Gorman~et al. \cite{reeeg} proved that a continuous regular function must be locally linear (in fact, locally $\mathbb{Q}$-affine) outside a measure zero nowhere dense set.
But even for continuous regular functions, the  classification problem for regular functions raised in \cite{vardi} and stated in (Q.1) above remains unresolved. See Section \ref{reg:prelim} for formal definitions and further discussions.  

\smallskip

We are ready to state the first main result of the present article. Although our results are not really sensitive to the choice of base, for simplicity we will fix the standard binary representation of reals. Our first result shows that the classification problems for continuous regular and transducer functions (that is, (Q.1) and (Q.2)) are equivalent.

\begin{theorem}\label{thm:kuku}
Suppose $f\colon [0,1]\rightarrow [0,1]$. The following are equivalent with respect to the standard binary representation:
\begin{enumerate}
\item $f$ is continuous regular.
\item $f$ can be computed by a nondeterministic transducer.
\end{enumerate}
\end{theorem}
This result contrasts with the results in \cite{BKN:08} which show that there exist automatic algebraic structures which do not have transducer presentations;
for a further discussion, see Remark~\ref{rem:makingtherefhappy}.
Furthermore, in Theorem~\ref{thm:kuku1} we will show that the answers to (Q.1) and (Q.2) are equivalent to saying that $f$ can be computed by a deterministic transducer for the \emph{signed binary representation} (to be defined in Def.\ \ref{def:sbr}). In other words, the \emph{only} difficulty in determinization is that we cannot predict whether a long sequence of 1s will ever terminate in 0 or not. Signed binary representations can be computed from binary representations by a transducer, but not vice versa. We also show that this difficulty is essential by producing an example of a nondeterministic transducer function which cannot be computed by a deterministic transducer; see Prop.~\ref{determcounterexample}.

Theorem~\ref{thm:kuku} explains why some proofs for regular functions that can be found in \cite{vardi,reeeg} are quite similar  to proofs of the analoguous results for transducer functions in~\cite{Lisovik1,MullerTransducer}. For instance,  the following theorem follows from the aforementioned main result in \cite{Lisovik1} and Theorem \ref{thm:kuku}.

\begin{theorem}[Block Gorman et al.~\cite{reeeg}]
 For a function $f\colon  [0,1] \rightarrow \mathbb{R}$, the following are equivalent:
\begin{enumerate}
\item   $f$ is regular and differentiable;

\item there exist $r, q \in
\mathbb{Q}$ such that $f(x) = rx +q$.
\end{enumerate}
\end{theorem}

It follows that (Q.1) is completely reduced to (Q.2).
To attack (Q.2) using computability-theoretic tools we need some background on index sets.

\subsection{Classification via index sets} We begin with local effective classifications; let us first consider the classification problem for real numbers with property $P$. We may initially analyze the subproblem for computable real numbers: a real $x$ is computable if there exists an algorithm which, on input $n$, outputs a rational number $r$ such that $|x -r|< 2^{-n}$ \cite{Turing:36,Turing:37}. We can therefore list all computable reals $x_0,x_1,\ldots$ effectively. More precisely, we can effectively list all (possibly partial) algorithms computing approximations to computable reals. 
This allows us to define the complexity of $P$ within the arithmetic hierarchy by restricting its domain to computable reals and considering the \emph{index set} of $P$:
$$I(P)=\{ i \st \mbox{$x_i$ satisfies $P$}\}.$$
For instance, if $P(x_i)$ holds if and only if $(\exists n_1)(\forall n_2)R(i,n_1,n_2)$ for some computable relation $R$, then we can say that $P$ is $\Sigma^0_2$, and if $I(P)$ is $\Sigma^0_2$-complete, then we can say that our description of $P$ is optimal. See~\cite{Soa} for more on the classes $\Sigma^0_n$ and $\Pi^0_n$, which together form the arithmetical hierarchy.
Of course, if we want to consider an analytic $P$ for noncomputable reals as well, we will relativize our restatement of $P$ to an arbitrary oracle $X$; consider the case in which $P(\xi)$ holds if and only if $(\exists n_1)(\forall n_2)\widehat{R}(\xi;n_1,n_2)$, where $\widehat{R}$ is a computable predicate with the real parameter $\xi$. In this example, $P$ is $\mathbf{\Sigma^0_2}$; see~\cite{GaoBook} for the boldface hierarchy and its applications to algebraic structures. We note here that usually establishing a ``lightface'' bound requires a more detailed and constructive analysis than providing a ``boldface'' bound. This sort of work has been done by, for instance, Becher, Heiber, and Slaman, who showed that the index set of all computable real numbers normal to base 2 is $\Pi^0_3$-complete \cite{norm1}. Their proof  is relativizable and implies the earlier result of Ki and Linton, who showed that the set of reals normal  to base 2 is $\bf \Pi^0_3$-complete \cite{ki}. 
 For more examples see, e.g., \cite{norm2,CRindex}.

 The index set approach can be naturally extended to elements of an arbitrary computable Banach space $\mathcal{B}$ as follows. A computable presentation (of a computable structure)~\cite{PourElRich, Brattka.Hertling.ea:08}  of a Banach space $\mathcal{B}$ is a linearly dense sequence $\{v_n\}_{n \in \mathbb{N}}$ of vectors such that $||v_i||$ is a computable real uniformly in $i$, and we say that a space is computable if it admits a computable presentation. 
 A point $\xi \in \mathcal{B}$ is computable with respect to the given computable presentation if there exists an algorithm which, on input $n$, outputs an $i$ such that $d(\xi, v_i)< 2^{-n}$.  As before, we can list all computable points in the space effectively, which enables us to apply the index set technique for elements of the space. In this article, we will use the local algorithmic approach to study special classes of functions in $C[0,1]$. As far as we know, Westrick \cite{Westr} was the first to apply this approach to study classes of functions in $C[0,1]$, but index sets have never been used to study feasible computability of continuous functions.

\medskip

We are now ready to state the second main result of the present article.
It is not difficult to see that transducer functions $[0,1] \rightarrow \mathbb{R}$ map rationals to rationals in linear time (see Fact~\ref{linearfact}); we call such functions \emph{pointwise linear time}.
By combining the strongest known results about transducer and continuous regular functions in the literature, we can conclude that such functions are Lipschitz~\cite{vardi} and are locally linear outside of  a measure zero nowhere dense set~\cite{reeeg}; we call the latter property \emph{almost linearity}. While these properties are surely strong, it is not obvious that they help to  classify regular functions among pointwise linear time functions.
We prove the following:


\begin{theorem}\label{thm:mainreg}
Given a pointwise linear-time computable, almost linear Lipschitz $f\colon  [0,1] \rightarrow \mathbb{R}$ with $f(\mathbb{Q} \cap [0,1]) \subseteq \mathbb{Q}$,
checking whether $f$ is transducer (equivalently, regular) is a $\Sigma^0_2$-complete problem.
\end{theorem}

We explain how the framework of index sets is used to formally clarify the statement of  Theorem~\ref{thm:mainreg}.  Let $P$ be the property of being transducer or regular continuous in binary, and fix a uniform enumeration of all linear time computable functions $(l_e)_{e \in \omega}$ of the space.
Let $Q$ denote the conjunction of ``being Lipschitz'', ``mapping $\mathbb{Q}$ to $\mathbb{Q}$'', and ``being almost linear''.
We fix a $\Sigma^0_2$-complete set $S$ and  prove $\Sigma^0_2$-completeness of $I(P) = \{e\st  l_e \mbox{ satisfies } P \}$ via
$$e \in S \iff  g(e) \in I(P), $$
where $g(e)$ \emph{always} has property $Q$ regardless of whether $ P(l_e)$ or $\neg P(l_e)$ holds.  Thus, $\Sigma^0_2$-completeness measures the complexity of the classification problem for $P$ relative to/given $Q$.

How good is this upper bound $\Sigma^0_2$ in Theorem~\ref{thm:mainreg}? It is not difficult to see that, for an everywhere defined computable $f\colon [0,1] \rightarrow \mathbb{R}$,  ``$f$ is regular'' is a $\Sigma^0_2$-statement (to be explained in more detail). It follows that even if an interesting analytic characterisation of regular functions exists, it will not allow us to simplify the determination of whether a given (total) computable function is regular. Indeed, from the algorithmic point of view, any such characterisation will not be simpler than the brute force definition, which  is not even decidable relative to the halting problem. Furthermore, even assuming the strong properties---almost linear, Lipschitz,  $f(\mathbb{Q} \cap [0,1]) \subseteq \mathbb{Q}$, and linear time---we cannot reduce the complexity.  According to the index set approach, none of these properties  help to reduce the complexity of deciding whether a function is regular.
We will further discuss what this tells us about (Q.1) and (Q.2)  in the brief conclusion.

While the first main result (Theorem~\ref{thm:kuku}) is essentially purely automata-theoretic-analytic, Theorem~\ref{thm:mainreg} blends automata-theoretic results with techniques from computable analysis. Using the machinery developed for (and in) the proof of 
Theorem~\ref{thm:kuku} to do most of the heavy lifting, we will produce a rather compact and hopefully transparent proof of Theorem~\ref{thm:mainreg} that is essentially computability-theoretic rather than automata-theoretic.



\subsection{Classifying $C[0,1]$ among separable Banach spaces} Finally, we reach our result on the global classification problem for $C[0,1]$. For this result, we abandon automata theory and instead will use local definability techniques and oracle computability combined with \emph{global} index sets.

Now we discuss global classifications: whereas we consider the individual points in the space in a local classification problem, here we consider the space as a whole and attempt to characterize the space (or even a class of spaces) within some larger class.
Fix a uniformly computable list of all linear spaces over $\mathbb{Q}$ and all potential partial computable norms on these spaces. This gives us a uniformly computable enumeration of (partial) computable normed linear spaces $B_0, B_1, \ldots$. 
Then the characterization problem for a property $P$ of a separable Banach space 
is the (complexity of) the index set
$$I(P) = \{i \st \overline{B_i} \mbox{ has property }P\}.$$
The basic idea comes from computable structure theory, where the method of index sets has become standard; see, e.g., \cite{GK:02, DoMo,KnMcind}. A similar approach had been successfully applied in the classes of Polish spaces and Polish groups previously (see \cite{CompComp,NSlocal,Pontr}), but 
this method is still new for Banach spaces and needs to be investigated.

For example, suppose we define $P$ to be ``being a Hilbert space". Determining whether $\overline{B_i}$ admits an inner product is naively $\Sigma^1_1$, however, it is well known that $\overline{B_i}$ is a Hilbert space if and only if it obeys the parallelogram law, which only has to be checked for points in the dense linear subspace. Therefore, its index set is merely $\Pi^0_1$.  The situation is much more complicated if we consider Lebesgue spaces instead of Hilbert spaces. It seems that checking whether $ \overline{B_i} \cong L^p(\Omega)$ requires searching for a real $p$, a measure space $\Omega$, and an isomorphism to $L^p(\Omega)$.
Nonetheless, Brown, McNicholl and Melnikov~\cite{leb} have recently discovered that Lebesgue spaces admit a local description in terms of a certain new notion of independence. This analysis allows us to reduce the complexity of the index set from  $\Sigma^1_1$ down to $\Pi^0_3$, a considerable difference. 
As usual, this result can be relativized to any oracle, and therefore they are not restricted to computable Lebesgue spaces.

The main problem considered in Theorem  \ref{thm:3} is:
 How hard is it to tell that a given Banach space is (linearly) isometrically isomorphic to $C[0,1]$? More formally, what is the complexity of the set $\{i\st  \overline{B_i} \cong C[0,1]\}$ (Q.3)? The crude upper bound is $\Sigma^1_1$. Our challenge is to present a better bound.
Melnikov and Ng~\cite{MelNg} showed that  $C[0,1]$ admits computable presentations which are not computably isomorphic, and that therefore we should not expect an elementary solution. This question (in its  equivalent form) has recently been attacked by Brown \cite{BrownThesis}. Brown found an arithmetical upper bound, but he had to extend the signature of Banach spaces to get it. The last main result of the present article gives a surprisingly low upper bound:


 \begin{theorem}\label{thm:3}
The index set of $(C[0,1], || \cdot ||_{sup}, +, (r\cdot)_{r \in \mathbb{Q}})$ among all computable Banach spaces  is arithmetical.
\end{theorem}

More specifically, we show that the upper bound can be improved from $\Sigma^1_1$ to arithmetical, namely $\Sigma^0_7$. We conjecture that with some extra work, this upper bound can be improved to $\Pi^0_6$, but we do not know whether this is sharp. We leave these as open problems. The key ingredients of our proof are the notion of independence introduced by Brown \cite{BrownThesis} combined with a \emph{new definability technique}.
 In particular, the initial results on the computational strength of the auxiliary functions required appear in \cite{BrownThesis}, though we provide different proofs here. 
We use this analysis
to dynamically approximate (an isomorphic image of) the standard basis of the space consisting of tooth functions with rational breaking points. Another key idea in the proof is that one can  approximate a strictly monotonic function in $C[0,1]$ using only a few Turing jumps\footnote{We thank Alec Fox who pointed out that the original write-up of our proof was incomplete.}.
 
 In contrast with 
our previous theorem about transducer (regular) functions,
Theorem~\ref{thm:3} is a positive result. 
It follows that the global characterization problem for $C[0,1]$ admits a first-order characterisation based on local analysis of some individual elements of the space. This characterization is similar to that of Hilbert spaces and Lebesgue spaces; this is a peculiar property which may prove useful in the future.
See the brief conclusion (Section~\ref{concl}) for further discussion and explanation.

\medskip

The remainder of the article is structured as follows: We begin with a summary of the necessary results on regular functions in Section \ref{sec:regfns}; these results are often incompletely given in the literature, so we present full proofs here both for the sake of completeness and because we will use the notation and methods later on. (We also suspect that most potential readers of this article will be experts on computable analysis rather than automata theorists.)
 Then we devote a section to each of our three main results in turn: Section~\ref{sec:kuku} contains a proof of Theorem \ref{thm:kuku}, and  Section \ref{sec:mainreg} contains a proof of Theorem \ref{thm:mainreg}, which blends the automata-theoretic and analytic methods developed in the previous sections with elements of computable analysis. 
A bit more computable analysis is needed for the proof of  Theorem \ref{thm:3}; it is contained in Section \ref{sec:thm:3}. This section also contains the very little extra background that is needed for the proof. Finally, Section~\ref{concl} is a short conclusion that further clarifies 
what these index set results tell us about (Q.1)--(Q.3), and it also contains several questions that we leave open.

\section{Background on regular functions}\label{sec:regfns}

In this section we state and prove several known results which are scattered throughout the literature. Some of these results were only published in conference proceedings and thus there are no complete detailed proofs in the literature, which we provide for the sake of completeness. Some of the notation and methods introduced in these proofs will be useful in the proofs of our main results.

\subsection{Representing the real numbers}\label{sec:representation}
We now address the elephant in the room - the issue of representing real numbers. Our task is to study the automatic real analysis of continuous real-valued functions.
In order to avoid having to represent the sign and the integer part of a real number we will focus only on functions of the form $f\colon  [0,1]\rightarrow [0,1]$. 
Since we only work with a finite alphabet it is natural to represent real numbers by considering surjective functions $j: k^\omega\rightarrow [0,1]$, where $k^\omega$ is a compact subset of the Baire space.

Since we will frequently be working with strings (both finite and infinite), we write $\alpha\subset\beta$ to mean that the string $\alpha$ is a strict prefix of the string $\beta$. If $\alpha$ and $\beta$ are two infinite strings, then $\alpha\oplus\beta$ is the standard interweaving of the two strings,  defined by $(\alpha\oplus\beta)(2n)=\alpha(n)$ and $(\alpha\oplus\beta)(2n+1)=\beta(n)$ for all $n$. If $\alpha$ is a (finite or infinite) string, and $m$ is an integer smaller than $|\alpha|$, we denote $\alpha\upharpoonright m$ to be its initial segment of length $m$, $\alpha(0)\alpha(1)\cdots\alpha(m-1)$.

For simplicity, we restrict ourselves to the different binary representations of reals. We note, though, that this is just a convenient simplification for the sake of exposition and our results are not really restricted to binary representations.
\begin{definition}[Binary representations of real numbers]\label{def:sbr}
For each finite binary string $\sigma$, let $\overline{\sigma}$ denote the dyadic rational represented by $\sigma$, i.e., $\overline{\sigma}=\sum_{s<|\sigma|} \sigma(s)2^{-s-1}$. For an infinite binary string $\beta$, let $\overline{\beta}$ denote the real number represented by $\beta$. This is the \emph{standard binary representation} of $[0,1]$.

If $\eta\in \{-1,0,1\}^{<\omega}$, let $\sgnd{\eta}=\sum_{s<|\eta|} \eta(s)2^{-s-1}$; similarly we define $\sgnd{\beta}$ for $\beta\in \{-1,0,1\}^{\omega}$. This is the \emph{signed binary representation} of $[0,1]$.
\end{definition}

These are two commonly used representations for the real numbers; the latter is well known to be a universal (admissible) representation. Both representations are clearly not injective; for instance, each dyadic rational has two representations with respect to the standard binary representation and three with respect to the signed binary representation. Chaudhuri, Sankaranarayanan, and Vardi~\cite{vardi} overcome this by disallowing binary representations that end with $1111\ldots$. This allows for a binary representation of $[0,1]$ which is injective (but not total), and therefore circumvents the issue of having to identify different representations when considering the mathematical object being computed by an automaton.

We find this convention somewhat unsatisfactory and therefore propose to consider the more (arguably) natural standard and signed binary representations. These representations are also well-studied in computable analysis, and reflects practical problems in online computing more closely.

 The following observation is trivial but useful: Let $\sigma$ and $\tau$ be two strings in $2^{n}$ such that $|\overline{\sigma}-\overline{\tau}|>2^{-n}$. Then for every $\alpha\supset \sigma$ and every $\beta\supset\tau$ such that $\alpha,\beta\in 2^\omega$, we have~$\overline{\alpha}\neq\overline{\beta}$. Furthermore, if $\eta$ and $\nu$ are two strings in $\{-1,0,1\}^n$ such that $|\sgnd{\eta}-\sgnd{\nu}|>2^{-n+1}$, then for every~$\alpha\supset \eta$ and every $\beta\supset\nu$ such that $\alpha,\beta\in \{-1,0,1\}^\omega$, we have $\sgnd{\alpha}\neq \sgnd{\beta}$.

\subsection{Regular functions}\label{reg:prelim} 

Recall that a \emph{B\"{u}chi automaton} is a tuple $\left(Q,\Sigma,\delta,q_0,F\right)$ such that $\Sigma$ is a finite non-empty set of symbols representing the set of input alphabets, $Q$ is a finite non-empty set of states, $q_0\in Q$ is the initial state, $\delta:Q\times\Sigma\rightarrow Q$ is the state-transition function, and $F\subseteq Q$ is the set of accepting states. A run is accepting if some state of $F$ appears infinitely often in the run. A nondeterministic B\"{u}chi automaton is one where the state-transition function $\delta$ is replaced by a multi-function. An input word is accepted by the B\"{u}chi automaton if there is an accepting run associated with the input.

\begin{definition}[Modified from \cite{vardi}]\label{def:regularbase2}
A function $f\colon  [0,1] \rightarrow [0,1]$ is \emph{regular to base 2} if there is a nondeterministic B\"{u}chi automaton  $A$ which accepts
the graph of $f$ with respect to the standard binary expansion. That is, $A$ accepts $\alpha\oplus\beta$ if and only if $f(\overline{\alpha})=\overline{\beta}$.\end{definition}
Note that our definition differs slightly from \cite{vardi} as they use the modified standard binary representation for both the input and the output (see the discussion in Section \ref{sec:representation}). In their definition, each real number in $[0,1]$ has a unique binary representation, whereas in our definition, the automaton must accept any (binary) representation for $x$ and for $y$.

To use only one tape we alternate binary bits  of $x$ and $y$:
$$ x_0y_0x_1y_1x_2y_2\ldots,$$
and then $f(x) = y$ if and only if the automaton visits its acceptance state(s) infinitely often while scanning this input.
%
Chaudhuri, Sankaranarayanan, and Vardi~\cite{vardi} studied regular functions with respect to bases other than binary. In the rest of this article, we will use ``regular" to refer to functions that are regular to base 2.

We point out some different terminologies used in the literature. In \cite{reeeg} the authors refer to the set of (infinite) strings \emph{accepted} by a B\"{u}chi automaton, and call a subset $A\subseteq\mathbb{R}^n$ \emph{recognized} by a B\"{u}chi automaton (with respect to base $r$) if the set of all words representing an element of $A$ to base $r$ is accepted by the B\"{u}chi automaton. In \cite{vardi} however the authors used \emph{accepted} for both words and sets of reals and functions, although they introduced the notion of being accepted by a function automaton.

In this article, we prefer to simplify the notation and will only refer to words being accepted by a B\"{u}chi automaton. Whenever we refer to the \emph{graph} of a function $f$, we always mean the set of all strings $\alpha\oplus\beta$ such that $f(\overline{\alpha})=\overline{\beta}$.

It may not be immediately apparent, but not every regular function is continuous:

\begin{example}\label{easy:exampledisc}
The function
\[f(x)=\begin{cases}
0&\text{if $0\leq x<\frac{1}{2}$},\\
1 &\text{if $\frac{1}{2}\leq x\leq 1$}
\end{cases}\]
is regular. To see this, consider the following B\"{u}chi automaton  $A$. We interpret the input stream as $\alpha\oplus\beta$ where $\alpha$ and $\beta$ are infinite binary strings encoding the binary expansion of the real numbers $x$ and $ y$ respectively. $A$ will accept iff $\alpha(0)=1,\beta=1^\omega$, or $\alpha=01^\omega,\beta=1^\omega$, or $\alpha(0)=0,\beta=0^\omega$ and $\alpha$ contains at least two $0$ bits.


In fact, any step function with finitely many breakpoints, all of which are dyadic rationals, and a dyadic range is regular. This is because an automaton can first scan $D$ many bits of $\beta$ (where $D$ does not depend on the input) to decide which interval $x$ must necessarily belong to. It then scans $x$ and accepts iff $x$ lies in the relevant dyadic interval and $y$ is in the range of the step function.
\end{example}

It turns out that \emph{continuous} regular functions possess several nice properties. The rest of this subsection will be devoted to understanding continuous regular functions. First, we begin with a lemma; for the sake of completeness we include its proof as well.

\begin{lemma}[Block Gorman et al~\cite{reeeg}]\label{lem:deterministic}
Suppose $f\colon  [0,1] \rightarrow  [0,1]$ is continuous and regular. Then $Graph(f)$ is accepted by a deterministic  B\"{u}chi automaton.
\end{lemma}
\begin{proof}


Suppose that $Graph(f)$ is accepted by a B\"{u}chi automaton $A_0$ with states $Q_0$ and transition relation $\Delta_0$. Recall that a \emph{run}
is a (finite or infinite) sequence $(s_0,n_0,s_1,n_1,\ldots)$ such that $s_0$ is a starting state and for every $i$,  $n_i=0,1$ and $(s_i,n_i,s_{i+1})\in\Delta_0$. $A_0$ accepts $\alpha\oplus\beta$ if and only if there is an infinite run $(s_0,n_0,s_1,n_1,\ldots)$ such that $(\alpha\oplus\beta)(i)=n_i$ for every $i$ and there are infinitely many $i$ such that $s_i$ is an accepting state. It is easy to see that there is another B\"{u}chi automaton $A_1$ such that $A_0$ and $A_1$ accept the same set of strings and for every finite $A_1$-run $(t_0,n_0,\ldots,t_k,n_k,t_{k+1})$, we have that $t_{k+1}$ is $A_1$-accepting if and only if there is some $A_0$-run $(s_0,n_0,\ldots,s_k,n_k,s_{k+1})$ and some $i\leq k$ such that $s_i=s_{k+1}$ is $A_0$-accepting. We wish to consider $A_1$ instead of $A_0$ because we want to be able to apply an argument similar to the pumping lemma to show that each accepting state must be periodically visited.

Now we define the deterministic  B\"{u}chi automaton $A_2$ by the usual subset construction on $A_1$. More specifically, $Q_2=\mathcal{P}(Q_1)$, and we put $(X,n,Y)\in \Delta_2$ if and only if $Y=\{ y\mid (x,n,y)\in \Delta_1$ and $x\in X\}$. A~state $X$ is $A_2$-accepting if and only if $X$ contains an $A_1$-accepting state.

We claim that $A_2$ accepts $\alpha\oplus\beta$ if and only if $A_0$ accepts $\alpha\oplus \beta$. The subset construction always produces an automaton that accepts every string accepted by $A_1$;  however, the converse is not always true. In particular, there are $\omega$-regular languages that are not accepted by any deterministic B\"{u}chi automata. So suppose that $A_2$ accepts $\alpha\oplus\beta$. Let $x$ be the real encoded by $\alpha$ and $y$ be the real encoded by $\beta$; we argue that $y=f(x)$. Let $\left( X_0,n_0,X_1,n_1,\ldots\right)$ be an $A_2$-accepting run for $\alpha\oplus\beta$, 
 and let $g$ be such that $X_{g(k)}$ contains an $A_1$-accepting state for every $k$. For each $k$, the $A_1$-accepting state in $X_{g(k)}$ corresponds to a finite $A_1$-run $\left(t_0,n_0,t_1,n_1,\ldots,t_{g(k)}\right)$ where $t_{g(k)}$ is $A_1$-accepting. This means that there is an $A_0$-run $\left(s_0,n_0,s_1,n_1,\ldots,s_{g(k)}\right)$ and some $i(k)<g(k)$ such that $s_{i(k)}=s_{g(k)}$ is $A_0$-accepting; without loss of generality, assume that both $i(k)$ and $g(k)$ are odd. Let $x_k$ denote the rational encoded by $(n_0n_2\ldots n_{i(k)-1})*\left(n_{i(k)+1}n_{i(k)+3}\ldots n_{g(k)-1}\right)^\omega$ and $y_k$ be the rational encoded by $(n_1n_3\ldots n_{i(k)})*\left(n_{i(k)+2}n_{i(k)+4}\ldots n_{g(k)}\right)^\omega$. Obviously, for each $k$, $A_0$ accepts (the pair of strings encoding) $x_k$ and $y_k$, and thus $f(x_k)=y_k$. As $\lim_{k\rightarrow\infty}g(k)=\infty$, we have $\lim_{k\rightarrow\infty}x_k=x$ and $\lim_{k\rightarrow\infty}y_k=y$, thus, by the continuity of $f$, we have $f(x)=y$. Thus, $A_0$ must also accept $\alpha\oplus \beta$.
\end{proof}

We will also consider the fact that every continuous regular function is in fact Lipschitz continuous.
Recall $f$ is Lipschitz continuous if for every $x_1, x_2$, it satisfies ${\displaystyle |f(x_{1})-f(x_{2})|\leq K|x_{1}-x_{2}|}$ for some constant $K$.
 The implication is proved in Chaudhuri, Sankaranarayanan and Vardi~\cite{vardi}. We give an alternate proof of this fact. The reason for this is that our analysis here will be critical for the rest of this section, particularly in our  proof of Proposition \ref{reg->tr}, which is in turn used to prove Proposition \ref{lemma:kukudeterminsticnondeterministic}.

We also mention that even though we prove Lemma \ref{lemlip} for $2$-regular functions, it also holds for $k$\nobreakdash-regular functions for all $k\geq 2$ with essentially the same proof. However, as shown in \cite{reeeg} 
 if the input and output of a regular function are in different bases, then the function is not necessarily Lipschitz continuous.

\begin{lemma}[Chaudhuri, Sankaranarayanan, and Vardi~{\cite[Theorem 10]{vardi}}]\label{lemlip}
Suppose $f\colon  [0,1] \rightarrow \mathbb{R}$ is continuous and regular. Then $f$ is Lipschitz.
\end{lemma}

\begin{proof}
Suppose that $f$ is continuous and regular as witnessed by a  B\"{u}chi automaton $A$. By Lemma \ref{lem:deterministic}, we may assume that
$A$ is deterministic. We further assume that $A$ has $N$ states and represent the graph of $f$ by the collection of all strings $\alpha\oplus\beta$ where $\alpha$ is a binary expansion of some real number $x\in[0,1]$ and $\beta$ is a binary expansion of $f(x)$.

Instead of representing $A$ by a directed graph, it will be more convenient to represent the space of all possible configurations of $A$ by a labeled subtree of the full binary tree. We will denote the projection of this tree to be $T_\alpha$ when we fix the input to be $\alpha\in 2^\omega$: First, let $\ell(\sigma,\tau)$ be the state reached by $A$ immediately after scanning $(\sigma\upharpoonright m)\oplus (\tau\upharpoonright m)$, where $m=\min\{|\sigma|,|\tau|\}$. Now for each finite string $\sigma$, denote the finite tree $T_\sigma$ to be the set of all strings $\tau$ such that $|\tau|\leq |\sigma|$ and for all $i\leq |\tau|$, $\ell(\sigma,\tau\upharpoonright i)$ is unique among the collection of $\ell(\sigma,\eta)$s for all $\eta$ of length $i$: in other words, if $\ell(\sigma,\tau)=\ell(\sigma,\eta)$ for some~$\eta\neq\tau$ where $|\eta|=|\tau|$, we remove both $\eta$ and $\tau$ from the tree $T_\sigma$. This is because no infinite extension of $\tau$ (or of $\eta$) can be accepted by $A$ with any input extending $\sigma$. (Otherwise, use $\ell(\sigma,\tau)=\ell(\sigma,\eta)$ to output two different \emph{real numbers} for the same input $\sigma$ by extending $\tau$ and $\eta$ in the same way
). Such (pairs of) strings cannot represent correct outputs and we do not need to consider them anyway.


 Thus, $T_\sigma$ will be pruned so that it only contains those finite strings which still have a chance of being extended to a valid output.
A node $\tau$ in $T_\sigma$ is said to be extendible if there is some $\tau'\supseteq\tau$ such that $|\tau'|=|\sigma|$ and $\tau'\in T_\sigma$. Clearly, the maximum number of pairwise incomparable extendible nodes of $T_\sigma$ is $N$. For each string $\sigma$ and integers $i<j\leq|\sigma|$, it will be convenient to write $\sigma\res [i,j)$ instead of $\sigma(i)\sigma(i+1)\ldots\sigma(j-1)$. If $T$ is a tree we write $[T]$ to be the set of all infinite paths of $T$, i.e., $[T]$ is the set of all infinite strings $\alpha$ such that $\alpha\upharpoonright m\in T$ for every $m$.

\begin{claim}\label{claim:kuku00}
For any $\alpha\oplus\beta$ accepted by $A$, we have $\beta\in [T_{\alpha\upharpoonright m}]$ for every $m$.
\end{claim}
\begin{proof}[Proof of claim]
Suppose not. Then there is some $i\leq m$ and some $\eta$ of length $i$ such that $\eta\neq\beta\upharpoonright i$ and $\ell(\alpha,\beta\upharpoonright i)=\ell(\alpha,\eta)$. But this means that $\alpha\oplus \left(\eta*\beta\res [i,\infty)\right)$ is also accepted by $A$. However, $\eta*\beta\res [i,\infty)$ clearly represents a different real number than $\beta$, giving us a contradiction.
\end{proof}

\begin{claim}\label{fact:onemore}\,
Given any $\alpha\oplus\beta$ accepted by $A$, and $n<m\leq p$ such that $\ell\left(\alpha\upharpoonright p,\beta\upharpoonright n\right)=\ell\left(\alpha\upharpoonright p,\beta\upharpoonright m\right)$ is a non-accepting state, there must be some $q$ such that $n<q<m$ where $\ell\left(\alpha\upharpoonright p,\beta\upharpoonright q\right)$ is accepting. 
\end{claim}
\begin{proof}[Proof of claim]
Suppose not. 
Then consider, for each $s$, $\alpha_s$ to be the infinite string $\left(\alpha\upharpoonright n\right)*\left(\alpha\upharpoonright [n,m)\right)^s*\alpha\upharpoonright [m,\infty)$, and $\beta_s$ defined similarly. Then clearly for each $s$, $\alpha_s\oplus\beta_s$ is accepted by $A$. Since $\alpha_s$ converges in value to the value of $\alpha_\infty=\left(\alpha\upharpoonright n\right)*\left(\alpha\upharpoonright [n,m)\right)^\omega$ and $\beta_s$ converges in value to the value of $\beta_\infty=\left(\beta\upharpoonright n\right)*\left(\beta\upharpoonright [n,m)\right)^\omega$, by continuity, 
$A$ should also accept $\alpha_\infty\oplus\beta_\infty$. But this is not true as none of the states $\ell\left(\alpha\upharpoonright p,\beta\upharpoonright q\right)$ is accepting.
\end{proof}

\begin{claim}\label{claim:kukumain}
There is a constant $D$ (which depends only on $N$) such that given any $i$ and any $\sigma$ of length $i+D$, there are $\eta_0\subset\eta_1\subset\ldots\subset\eta_{2N}$ such that $i+3<|\eta_0|<|\eta_1|<\ldots<|\eta_{2N}|< i+D$ and for all $j,j'\leq 2N$, $\eta_j\in T_\sigma$ and $\ell(\sigma,\eta_j)=\ell(\sigma,\eta_{j'})$ is an accepting state.
(That is, $\ell(\sigma,\eta_j)$ is constant amongst the $j$s.) 
\end{claim}
\begin{proof}[Proof of claim]
Given $i$, consider $D$ large enough. There must some infinite strings $\alpha\supset \sigma$ and $\beta$ such that $\alpha\oplus \beta$ is accepted by $A$, since the function is total. By Claim \ref{claim:kuku00}, all finite initial segments of $\beta$ are in $T_\sigma$. By Claim \ref{fact:onemore}, for every $N+1$ consecutive values of $n$, $\ell(\sigma,\beta\upharpoonright n)$ must be an accepting state for some $n$. $D$ can be calculated accordingly and depends only on $N$.
\end{proof}

We now return to the proof of Lemma \ref{lemlip}. Fix $\sigma,i,j<j'$ as in Claim \ref{claim:kukumain}. Note that $(\sigma\upharpoonright |\eta_j|)*\left(\sigma\res \left[|\eta_j|,|\eta_{j'}|\right)\right)^\omega$ is accepted by $A$ together with $\eta_j*\left(\eta_{j'}\res[|\eta_j|,|\eta_{j'}|)\right)^\omega$ (i.e., as a pair).

To show that $f$ is Lipschitz, we wish to argue that $|f(x)-f(x')|\leq 2^C|x-x'|$ for any $x,x'$ for some constant $C$: we will see that $C=D+1$ will do. Note that $C$ depends only on $N$.
It is enough to show the following statement: For each $i$ and each dyadic rational $x'=p2^{-i-C}$ in $[0,1]$
  and each real number $x \in [0,1]$, if $|x-x'|<2^{-i-C}$ then $|f(x)-f(x')|<2^{-i}$. We fix such $i$, $x'$, and $x$ as above so that $|x-x'|<2^{-i-C}$. As $x'$ is dyadic, we can represent $x'$ as a finite binary string $\sigma$, i.e., $x'=\overline{\sigma}$ where $|\sigma|=i+C$. Apply Claim \ref{claim:kukumain} to obtain the sequence $\eta_0\subset\eta_1\subset\ldots\subset\eta_{2N}$ for $\sigma$, $i$, and $C$. For the remainder of this proof, we fix $i,x',x,\sigma$ and the sequence
$\eta_0\subset\eta_1\subset\ldots\subset\eta_{2N}$. We start by proving the following claim:

\begin{claim}\label{claim:kuku1}
Given any infinite strings $\alpha$ and $\beta$ such that $\alpha\oplus\beta$ is accepted by $A$, let $z$ be the real number represented by $\alpha$. Then if $|z-x'|<2^{-i-C}$, we have that $\left|\overline{\beta\upharpoonright (i+4)}-\overline{\eta_0\upharpoonright (i+4)}\right|\leq 2^{-i-4}$.
\end{claim}
\begin{proof}[Proof of claim]
There are three possibilities for $\alpha$: it may extend $\sigma\upharpoonright |\eta_N|$, be lexicographically to its left, or be lexicographically to its right. We first assume that $\alpha\supset\sigma\upharpoonright |\eta_N|$. Suppose that \[{\left|\overline{\beta\upharpoonright (i+4)}-\overline{\eta_0\upharpoonright (i+4)}\right|> 2^{-i-4}}.\] Fix some $j<j'<N$ such that $\ell\left(\alpha,\beta\upharpoonright |\eta_j|\right)=\ell\left(\alpha,\beta\upharpoonright |\eta_{j'}|\right)$. Now, by Claim \ref{claim:kukumain}, $$\left[(\sigma\upharpoonright |\eta_j|)*\left(\sigma\res \left[|\eta_j|,|\eta_{j'}|\right)\right)^\omega\right]\oplus
\left[\eta_j*\left(\eta_{j'}\res[|\eta_j|,|\eta_{j'}|)\right)^\omega\right]$$ is accepted by $A$. On the other hand, for each $k \in \omega$, 
we also have that
$$\left[(\sigma\upharpoonright |\eta_j|)*\left(\sigma\res \left[|\eta_j|,|\eta_{j'}|\right)\right)^k*\alpha\upharpoonright[|\eta_{j'}|,\infty)\right]\oplus
\left[\left(\beta\upharpoonright|\eta_j|\right)*\left(\beta\upharpoonright[|\eta_j|,|\eta_{j'}|)\right)^k*\left(\beta\upharpoonright[|\eta_{j'}|,\infty)\right)\right]$$
is accepted by $A$. Obviously, the sequence of real numbers represented by
$$(\sigma\upharpoonright |\eta_j|)*\left(\sigma\res \left[|\eta_j|,|\eta_{j'}|\right)\right)^k*(\alpha\upharpoonright[|\eta_{j'}|,\infty))$$
converges to the real number represented by $\left[(\sigma\upharpoonright |\eta_j|)*\left(\sigma\res \left[|\eta_j|,|\eta_{j'}|\right)\right)^\omega\right]$ as $k\rightarrow\infty$. However, since we assumed that $\left|\overline{\beta\upharpoonright (i+4)}-\overline{\eta_0\upharpoonright (i+4)}\right|> 2^{-i-4}$, we conclude that $\beta$ and $\left[\eta_j*\left(\eta_{j'}\res[|\eta_j|,|\eta_{j'}|)\right)^\omega\right]$ represent different real numbers, a contradiction to the continuity of $f$.

Now assume that $\alpha$ is lexicograpically to the left of $\sigma\upharpoonright |\eta_{N}|$. (The case where $\alpha$ is lexicographically to the right of $\sigma\upharpoonright |\eta_{N}|$ is similar.) Again, suppose that $\left|\overline{\beta\upharpoonright (i+4)}-\overline{\eta_0\upharpoonright (i+4)}\right|> 2^{-i-4}$. Let $d<|\eta_N|$ be such that $\alpha\upharpoonright d=\sigma\upharpoonright d$, $\alpha(d)=0$ and $\sigma(d)=1$. Since $|z-x'|<2^{-i-C}$, 
it follows that for every $d<d'<|\eta_{2N}|$ we have $\alpha(d')=1$ and $\sigma(d')=0$. Now fix some $N<j<j'<2N$ such that $\ell\left(\alpha,\beta\upharpoonright |\eta_j|\right)=\ell\left(\alpha,\beta\upharpoonright |\eta_{j'}|\right)$. By the same argument as the first part of the proof of this claim, we have that
$$\left[(\sigma\upharpoonright |\eta_j|)*\left(\sigma\res \left[|\eta_j|,|\eta_{j'}|\right)\right)^\omega\right]\oplus
\left[\eta_j*\left(\eta_{j'}\res[|\eta_j|,|\eta_{j'}|)\right)^\omega\right]$$
as well as
$$\left[(\alpha\upharpoonright |\eta_j|)*\left(\alpha\res \left[|\eta_j|,|\eta_{j'}|\right)\right)^k*(\alpha\upharpoonright[|\eta_{j'}|,\infty))\right]\oplus
\left[\left(\beta\upharpoonright|\eta_j|\right)*\left(\beta\upharpoonright[|\eta_j|,|\eta_{j'}|)\right)^k*\left(\beta\upharpoonright[|\eta_{j'}|,\infty)\right)\right]$$
are accepted by $A$ for all $k$. However, since $(\sigma\upharpoonright |\eta_j|)*\left(\sigma\res \left[|\eta_j|,|\eta_{j'}|\right)\right)^\omega=\left(\sigma\upharpoonright d\right)*1*0^\omega$ and ${{(\alpha\upharpoonright |\eta_j|)}*\left(\alpha\res \left[|\eta_j|,|\eta_{j'}|\right)\right)^k*(\alpha\upharpoonright[|\eta_{j'}|,\infty))}$ converges to
$\left(\sigma\upharpoonright d\right)*0*1^\omega$ as $k\rightarrow\infty$, this means that if $f$ is continuous, then $\eta_j*\left(\eta_{j'}\res[|\eta_j|,|\eta_{j'}|)\right)^\omega$ and  $\left(\beta\upharpoonright|\eta_j|\right)*\left(\beta\upharpoonright[|\eta_j|,|\eta_{j'}|)\right)^\omega$ represent the same real number. This, of course, contradicts our assumption that $\left|\overline{\beta\upharpoonright (i+4)}-\overline{\eta_0\upharpoonright (i+4)}\right|> 2^{-i-4}$.
\end{proof}

Now we fix $\alpha_0$ to be a binary representation of $x$ and $\beta_0$ to be a binary representation of $f(x)$. Let $\beta_1$ be a binary representation of $f(x')$. We wish to argue that $|f(x)-f(x')|<2^{-i}$. Applying Claim \ref{claim:kuku1} twice, first to $\alpha_0\oplus\beta_0$ and then to $\sigma*0^\omega\oplus\beta_1$, we see that
\[|f(x)-f(x')|\leq \left|f(x)-\overline{\eta_0\upharpoonright(i+4)}\right|+\left|f(x')-\overline{\eta_0\upharpoonright(i+4)}\right|\leq 2^{i-2}+2^{i-2}<2^{-i}.\qedhere\]
\end{proof}

\begin{remark}
We remark that regular (real) functions possess additional convenient decidability-theoretic properties. For example, there is an \( O(n^2) \)-time procedure for deciding  
the continuity of regular functions that can be represented by deterministic automata \cite[Theorem 7]{vardi}, and the set  
of minima of a regular function is also regular \cite[Theorem 14]{vardi}. As one of the referees pointed out, the equality of two regular real functions is also a decidable property. To see why, fix two automata representing~\(g\) and~\(f\), and use \cite[Theorem 4]{vardi} to produce an automaton representing \( h = f - g \). Then \( \forall x \, h(x) = 0 \) is automatic; this follows, for example, from the regularity of projection and the decidability of the universality problem. (Consider all inputs $(\xi, 0^{\omega})$ accepted by the automaton representing~$g$ and check if such $\xi$ make up the entire $\Sigma^\omega$.)
We will not need these properties in the present article.  
\end{remark}

\subsection{Transducer functions}\label{sec:transducer}
Formally, a \emph{finite-state transducer} is a tuple $\left(\Sigma,\Gamma,Q,s_0,\delta,\omega\right)$ where $\Sigma,\Gamma$ are finite non-empty sets of symbols representing the set of input and output alphabets respectively, $Q$ is a finite non-empty set of states, $s_0\in Q$ is the initial state, $\delta:Q\times\Sigma\rightarrow \Gamma\times Q$ is the state-transition function. A nondeterministic FST is one where the state-transition function is a multifunction.

A deterministic transducer computes a function $f\colon [0,1]\rightarrow\mathbb{R}$ if for every $x\in[0,1]$ and for every representation $\alpha$ of $x$, the transducer writes a representation for $f(x)$  on the output tape. Again we will use either the standard or signed binary representations, but this definition is not restricted to these types of representations.

A transducer may also be nondeterministic, and in this case, each finite run may be extendible to a number of different states and write different output bits upon reading the next input bit, or it may ``die,'' i.e., enter a state without outgoing edges.

 \begin{definition}
 A nondeterministic transducer computes a function $f\colon [0,1]\rightarrow[0,1]$ if for every~${x\in[0,1]}$, and every representation $\alpha$ of $x$, the following hold:
 \begin{itemize}
 \item There is an infinite run of the transducer while reading $\alpha$, and
 \item for every infinite run of the transducer while reading $\alpha$, it writes a representation for $f(x)$ on the output tape.
 \end{itemize}
 \end{definition}
 The nondeterministic nature is obviously captured by the ability of the transducer to discover an undesirable finite run and kill that particular run. However all runs which are allowed to go on indefinitely must produce a representation of the correct output at the end. Note that a nondeterministic transducer is allowed to produce different representations of $f(x)$ on different runs even for the same input representation of $x$.

If $\nu$ is a representation of $[0,1]$, then a transducer which works with input and output represented by~$\nu$ is called a \emph{transducer in $\nu$}.

The transducer model (both deterministic and nondeterministic) we will use is one which allows for a fixed delay $D$ (see \cite{MullerTransducer}). This model is identical to the one defined formally above, except that it has an additional parameter $D$. The transducer with delay $D$ is allowed to scan the first $D$~many digits of the input without writing any output, and may possibly go through different state transitions during this time (which allows the transducer to remember the first $D$ many digits read). At step $s>D$, the transducer will scan the $s$-th digit of the input, make a transition from state~$e_{s}$ to~$e_{s+1}$, and write digit $y_{s-D}$ on the output tape.
(See~\cite{srimani}). The transducer model is based on the idea of \emph{online} computation, therefore it is compelled to produce one output bit for every input bit read. A non-deterministic transducer works similarly; each run is allowed to die after finitely many steps (if that run enters a rejecting state), but every infinite run must produce an infinite length output.

\begin{remark}\label{rem:makingtherefhappy}
We acknowledge that there is a wide array of definitions in automata theory, and we expect that our results will be quite sensitive to the notions used. Here we use the definition of Muller~\cite{MullerTransducer}, who was perhaps one of the first to use transducers in real analysis; this notion essentially coincides with the well-known Mealy machines~\cite{Mealy}.  Thus, in the simplest case, the machines  
read exactly one symbol and write exactly one symbol during each cycle. We also require accepting all computations that produce an infinite output word. As one of the referees noted, this acceptance condition is stricter for transducers than Büchi's, which consequently permits a more diverse range of machines.

We mention that one might consider even further additions to the machine model, such as including rejecting states or more relaxed rules for reading and writing. As one of the referees pointed out, our approach likely admits equivalent formulations—and indeed, perhaps many equivalent formulations. For example, a similar but apparently more general approach was used in a series of articles by Lisovik and Shkaravskaya~\cite{Lisovik2,Lisovik3,Lisovik1}, who develop the use of pushdown $R$-transducers in real analysis.  
However, it turns out that this extra degree of generality does not seem to extend the class of continuous functions from $[0,1]$ to $\mathbb{R}$ that can be computed by a transducer. The key difference from classical automata theory is that we view automata in the real-analytic context, which inevitably imposes extra restrictions and requires additional care.  
For more results about Mealy machines, refer to~\cite{MR2264276,MR2194959}.  

We also point out that Brough, Khoussainov, and Nelson \cite{BKN:08} studied variations of Mealy machines, which they called input-output automata, in the algebraic context. In their article they introduced a hierarchy of notions for presenting algebraic structures based on the Mealy machine model. However, the notion in that article which most closely resembles our own is much more restrictive and implies finiteness in many classes, such as groups, rings, and Boolean algebras \cite[Proposition 1]{BKN:08}. In particular, it is not implied by regularity in the sense of Khoussainov and Nerode~\cite{KhoussainovN94}.

For alternate approaches to automatic structures, we refer the reader to Hodgson~\cite{hod1,hod2} and Blumensath and Gr{\"a}del~\cite{blumensath2000}. For a comprehensive survey of recent results on automatic structures, see  Gr{\"a}del~\cite{gradel2020}. We also mention the work of Khoussainov and Nerode~\cite{KhoussainovN08} in this context. We are unaware of similar survey work covering automatic presentations in real analysis.


\end{remark}

An important observation is that any $f$ computed by a deterministic transducer is continuous because it is computable.
It is a bit less obvious why functions computed by nondeterministic transducers have to be continuous as well; we provide a proof here.

\begin{proposition}\label{cont}
Suppose $f\colon [0,1] \rightarrow [0,1]$ is computed by a nondeterministic transducer in binary. Then $f$ is continuous.
\end{proposition}

\begin{proof}
Suppose that $f$ is computed by a nondeterministic transducer $M$ for the standard binary representation and suppose the states of $M$ are $s_0,\ldots,s_n$ and the transition relation is $\Delta_M$.  We can assume that $M$ has no delay: if it did, we could simply consider the function $2^{-D}f(x)$ instead.

We show that for any $\alpha$ and any $j_0$, there is some $j_1$ such that for every $\alpha'$ with $\alpha'\upharpoonright j_1=\alpha\upharpoonright j_1$, $|f\left(\overline{\alpha'}\right)-f(\overline{\alpha})|< 2^{-j_0}$. We fix $\alpha$ and $j_0$ as above and assume that no such $j_1$ exists. Fix an infinite run produced by $M(\alpha)$ and assume that this run writes the output $\beta$; of course, we have $\overline{\beta}=f(\overline{\alpha})$.

Now consider the set $U$ of all finite strings $s_{i(0)},s_{i(1)},\ldots,s_{i(k)}$ such that $s_{i(0)}$ is an initial $M$-state and, for every $k'<k$, $\left(s_{i(k')},\alpha(k'),\tau(k'),s_{i(k'+1)}\right)\in\Delta_M$ for some finite string $\tau$ of length $k$ such that $\left|\overline{\tau}-\overline{\beta}\right|\geq 2^{-j_0-1}$. That is, $U$ is the set of all finite runs of $M$ while scanning $\alpha$ which write an output that is relatively far away from $\beta$. The downwards closure of $U$ is a subtree $T$ of $\{s_0,s_1,\ldots,s_n\}^{<\omega}$. We first claim that $T$ is infinite: If $j_1>j_0+1$, then there is some $\alpha'\supset\alpha\upharpoonright j_1$ and some infinite run of $M(\alpha')$ that produces an output $\beta'$ such that $\left|\overline{\beta'}-\overline{\beta}\right|\geq 2^{-j_0}$. However, the finite initial segment of this run of length $j_1$ must be in $U$ since $\left|\overline{\beta'}-\overline{\beta}\right|\leq \left|\overline{\beta'\upharpoonright j_1}-\overline{\beta}\right|+\left|\overline{\beta'\upharpoonright j_1}-\overline{\beta'}\right|\leq\left|\overline{\beta'\upharpoonright j_1}-\overline{\beta}\right|+2^{-{j_0-1}}$. Therefore, by compactness, $T$ contains an infinite path $\delta$. Obviously, $\delta$ is a infinite run of $M(\alpha)$; assume it writes an output $\widetilde{\beta}$. If~$j>j_0+1$, then, as $\delta\upharpoonright j$ has an extension in $U$, we can conclude that $\left|\overline{\widetilde{\beta}\upharpoonright j}-\overline{\beta}\right|\geq 2^{-j_0-1}-2^{-j}\geq 2^{-j_0-2}$. Therefore, $\widetilde{\beta}$ and $\beta$ represent different real numbers, which is impossible as they are each an output of an infinite run of $M(\alpha)$. Thus, $f$ is continuous.
\end{proof}

We note that the class of functions that can be computed by a deterministic transducer for the standard binary representation is a strictly smaller class than the class of functions that can be computed by a nondeterministic transducer for the standard binary representation.  The following proposition provides us with a witness to this fact.

\begin{proposition}\label{determcounterexample}There is a continuous function $f\colon [0,1] \rightarrow [0,1]$ that can be computed by a nondeterministic transducer for the standard binary representation but not by any deterministic transducer for the standard binary representation.
\end{proposition}
	
\begin{proof}
We define a function $f$ as follows, see Figure~\ref{asdalkjdfgnlkasln}:
\[f(x)=\begin{cases}
2^{-1} &\text{if $x=0$},\\
x+2^{-1}-2^{-n-1}, &\text{if $x\in \left(2^{-n-1},2^{-n-1}+2^{-n-3}\right]$},\\
-x+2^{-1}+2^{-n-1}+2^{-n-2} &\text{if $x\in \left(2^{-n-1}+2^{-n-3},2^{-n-1}+3\cdot 2^{-n-3}\right]$},\\
x+2^{-1}-2^{-n}, &\text{if $x\in \left(2^{-n-1}+3\cdot 2^{-n-3},2^{-n}\right]$}.
\end{cases}\]
It is easy to see that $f$ can be computed by a nondeterministic transducer for the standard binary representation: Given an input stream $\alpha$, we output both $1000\ldots$ and $0111\ldots$ until we find the first bit $n$ such that $\alpha(n)=1$. Then the value of $\alpha(n)\alpha(n+1)\alpha(n+2)$ will determine whether $x$ lies in the second, third or fourth case in the definition of $f(x)$. If we need to output a number larger than $\frac{1}{2}$, we kill the run that outputs $011\ldots$, and we kill the other run otherwise.

We claim that $f$ cannot be computed by a deterministic transducer for the standard binary representation. If, on the contrary, $M$ does this, then $M$ must output the first digit of the output after scanning~$0^n$ for some $n$. Since $f\left(2^{-n-1}+2^{-n-3}\right)>\frac{1}{2}$ and $f\left(2^{-n-1}+3\cdot 2^{-n-3}\right)<\frac{1}{2}$, we see that $M$ cannot possibly compute $f$.
\end{proof}

	
\begin{figure}[t]
	\label{asdalkjdfgnlkasln}
	\begin{tikzpicture}%
		[declare function={%
			f(\t) =%
			%
			(\t > 2^(-0-1)) * (\t <= 2^(-0-1)+2^(-0-3)) * (\t+2^(-1)-2^(-0-1))+%
			(\t > 2^(-0-1)+2^(-0-3)) * (\t <= 2^(-0-1)+3* 2^(-0-3)) * (-\t+2^(-1)+2^(-0-1)+2^(-0-2))+%
			(\t > 2^(-0-1)+3* 2^(-0-3)) * (\t <= 2^(-0)) * (\t+2^(-1)-2^(-0))+%
			(\t > 2^(-1-1)) * (\t <= 2^(-1-1)+2^(-1-3)) * (\t+2^(-1)-2^(-1-1))+%
			(\t > 2^(-1-1)+2^(-1-3)) * (\t <= 2^(-1-1)+3* 2^(-1-3)) * (-\t+2^(-1)+2^(-1-1)+2^(-1-2))+%
			(\t > 2^(-1-1)+3* 2^(-1-3)) * (\t <= 2^(-1)) * (\t+2^(-1)-2^(-1))+%
			(\t > 2^(-2-1)) * (\t <= 2^(-2-1)+2^(-2-3)) * (\t+2^(-1)-2^(-2-1))+%
			(\t > 2^(-2-1)+2^(-2-3)) * (\t <= 2^(-2-1)+3* 2^(-2-3)) * (-\t+2^(-1)+2^(-2-1)+2^(-2-2))+%
			(\t > 2^(-2-1)+3* 2^(-2-3)) * (\t <= 2^(-2)) * (\t+2^(-1)-2^(-2))+%
			(\t > 2^(-3-1)) * (\t <= 2^(-3-1)+2^(-3-3)) * (\t+2^(-1)-2^(-3-1))+%
			(\t > 2^(-3-1)+2^(-3-3)) * (\t <= 2^(-3-1)+3* 2^(-3-3)) * (-\t+2^(-1)+2^(-3-1)+2^(-3-2))+%
			(\t > 2^(-3-1)+3* 2^(-3-3)) * (\t <= 2^(-3)) * (\t+2^(-1)-2^(-3))+%
			(\t > 2^(-4-1)) * (\t <= 2^(-4-1)+2^(-4-3)) * (\t+2^(-1)-2^(-4-1))+%
			(\t > 2^(-4-1)+2^(-4-3)) * (\t <= 2^(-4-1)+3* 2^(-4-3)) * (-\t+2^(-1)+2^(-4-1)+2^(-4-2))+%
			(\t > 2^(-4-1)+3* 2^(-4-3)) * (\t <= 2^(-4)) * (\t+2^(-1)-2^(-4))+%
			(\t > 2^(-5-1)) * (\t <= 2^(-5-1)+2^(-5-3)) * (\t+2^(-1)-2^(-5-1))+%
			(\t > 2^(-5-1)+2^(-5-3)) * (\t <= 2^(-5-1)+3* 2^(-5-3)) * (-\t+2^(-1)+2^(-5-1)+2^(-5-2))+%
			(\t > 2^(-5-1)+3* 2^(-5-3)) * (\t <= 2^(-5)) * (\t+2^(-1)-2^(-5))+%
			(\t <= 2^(-6-1)+2^(-6-3)) * (\t+2^(-1)-2^(-6-1))+%
			(\t > 2^(-6-1)+2^(-6-3)) * (\t <= 2^(-6-1)+3* 2^(-6-3)) * (-\t+2^(-1)+2^(-6-1)+2^(-6-2))+%
			(\t > 2^(-6-1)+3* 2^(-6-3)) * (\t <= 2^(-6)) * (\t+2^(-1)-2^(-6));%
		}]%
		\pgfplotsset{every tick/.style={black,}}
		\begin{axis}[
			axis x line=bottom, 
			axis y line=left, 
			ymin=0, 
			ymax=1, 
			xmin=0, 
			xmax=1,
			enlarge x limits={value=0.05,upper},
			enlarge y limits={value=0.05,upper}, 
			samples=509, 
			xtick={0,1}, 
			ytick={0,1}, 
			width=10cm, 
			height=10cm,
			xlabel style={at={(rel axis cs:1,0)},anchor=west,yshift=0.52cm},
			xlabel=$x$,
			ylabel style={at={(rel axis cs:0,1)},rotate=-90,anchor=south,xshift=1.2cm},
			ylabel=$f(x)$,
			]%
			\addplot[domain=0.0078125:1] {f(x)}; 
			\node[scale=0.15] at (axis cs:0.0041,0.5) {$\cdots$};
		\end{axis}
	\end{tikzpicture}
	\caption{The counterexample given in Proposition~\ref{determcounterexample}.}
\end{figure}
	
Finally, we can establish the following two results:
\begin{proposition}\label{Q->Q}
Let $f\colon [0,1] \rightarrow [0,1]$ be regular. Then $f$ maps rationals to rationals.
\end{proposition}
\begin{proof}
On input $(c_1c_2 \ldots c_k (u_1 \ldots u_k)^{\omega})\oplus(d_0d_1\cdots)$, 
the automaton must visit the same accepting state $q$ infinitely often. This means that there is some $i$ such that the automaton is in state $1$ after scanning infinitely many prefixes of the input that ends with $u_i$ (or $u_id_n$ for some $n$). Let $\sigma\oplus\tau$ and $(\sigma*\eta)\oplus(\tau*\nu)$ be two such prefixes. Then the automaton must also accept the input $(\sigma*\eta^\omega)\oplus(\tau*\nu^\infty)$. Obviously $\tau*\nu^\infty$ is the representation of a rational number.
\end{proof}

It follows from the proposition that if $f(x) = ax+b$ is regular, then $a, b \in \mathbb{Q}$.

\begin{lemma}\label{cor:piecewiselinearregular}
Suppose $f\colon  [0,1] \rightarrow \mathbb{R}$ is piecewise linear such that the breakpoints have dyadic coordinates. Then
$f$ is regular.
\end{lemma}
\begin{proof}
This simple fact seems to be folklore among experts in the subject; see, e.g., \cite{Kon, Lisovik1, reeeg} for similar results, and see Theorems 4 and 5 in \cite{vardi} for more general results which imply the proposition. Thus,
we  give only a sketch.
First, observe that a linear function of the form  $x + d$, where $d$ is dyadic identified with a finite string representing it, is regular. This is because the result of  adding $d$ to any binary sequence depends only on the first $length(d)+1$ many bits of $x$, and the rest is handled using a suitable carry. This process is regular.
For a function of the form $\frac{m}{2^k}(x + d)$, use that $x + d$ is regular, and also that the operation $\frac{1}{2^k} \cdot$ is just a shift of the representation by $k$ positions to the left, while multiplication by $m$ involves only looking at most $m$ digits ahead. \end{proof}

\subsection{Linear time computability}

Let $\mathbb{D}_2$ be the set of all dyadic rationals between $0$ and $1$, i.e.~all rationals of the form $\{\frac{n}{2^m}\mid n\leq 2^m$ and $m\in\omega\}$. Each dyadic rational $r$ can be represented by a finite binary string $\sigma$ such that $\overline{\sigma}=r$. This is extended in the obvious way to base $p$ for any $p>2$; i.e., each element of $\mathbb{D}_p=\{\frac{n}{p^m}\mid n\leq p^m$ and $m\in\omega\}$ can be represented by some $\sigma\in p^{<\omega}$. Each rational $r\in [0,1]$ can be represented in the binary almost periodic representation $r=\overline{\sigma*\tau^\omega}$ for some $\sigma,\tau\in 2^{<\omega}$.

We adopt the approach from \cite{Ko} and define complexity classes for continuous functions via bounding resources used by Turing functionals which represent the functions. We present a much more natural equivalent reformulation of this notion which works in the Lipschitz case (see  \cite[Corollary 2.21]{Ko}).

\begin{definition}
A Lipschitz function $f\colon [0,1] \rightarrow \mathbb{R}$ is said to be \emph{polynomial-time computable} if and only if there is a polynomial $p\colon \omega\rightarrow\omega$ and a function $\psi\colon 2^{<\omega}\times\omega\rightarrow 2^{<\omega}$ such that for every $\sigma$ and $n$,
\[\left|\overline{\psi(\sigma,n)}-f(\overline{\sigma})\right|<2^{-n}\]
and $\psi(\sigma,n)$ can be computed by a multi-tape deterministic Turing machine in $p\left(|\sigma|+n\right)$ many steps.
The Lipschitz function $f$ is said to be \emph{linear-time computable} if $p$ can be taken to be $O(|\sigma|+n)$. Clearly, each polynomial-time computable Lipschitz function is computable.
\end{definition}

We note that Kawamura, Steinberg, and Thies~\cite{kawamurafeeb} have  defined   linear-time computability for functions that are not necessarily Lipschitz. The restriction of their definition to the Lipschitz case is equivalent to our notion.


\begin{definition}
Let $f: [0,1] \rightarrow \mathbb{R}$ be Lipschitz such that $f\left[\mathbb{D}_2\cap[0,1]\right]\subseteq \mathbb{Q}$. We say that $f$ is \emph{pointwise linear-time computable} if, given any dyadic representation of $d$, we can compute a binary almost periodic representation of the output $f(d)$ in linear time.
  More formally,
there is a function $\psi: 2^{<\omega}\rightarrow 2^{<\omega}\times 2^{<\omega}$
such that for all $\sigma\in 2^{<\omega}$, $\overline{\tau*\eta^\omega}=f\left(\overline{\sigma}\right)$, 
$\psi(\sigma)=(\tau,\eta)$, and $\psi(\sigma)$ can be computed by a multi-tape deterministic Turing machine in $O(|\sigma|)$ many steps. \end{definition}


In  Proposition \ref{lemma:kukudeterminsticnondeterministic} we will show that every regular continuous $f$ can be computed by a deterministic transducer in signed binary representation. A simple consequence of this fact  is the following:

\begin{fact}\label{linearfact}
Every regular $f\in C[0,1]$ is pointwise linear-time computable.
\end{fact}

\begin{proof}
Each regular $f\in C[0,1]$ is Lipschitz (Lemma~\ref{lemlip}) and maps each dyadic rational to a rational number (Proposition~\ref{Q->Q}). By Proposition \ref{lemma:kukudeterminsticnondeterministic}, each regular $f\in C[0,1]$ can be computed by a deterministic transducer $M$ for the signed binary representation.  

Let $D$ be the delay of $M$ and $m$ be the number of states in $M$.  Let $q_0, q_1, \dots, q_m$ be states of $M$ such that when $M$ is run on $\sigma^\omega$, $M$ takes state $q_i$ at step $D + 1 + i|\sigma|$.  By the pigeonhole principle, there is $0 \le i < j \le m$ with $q_i = q_j$.  Let $\eta$ be the string written by $M$ running on $\sigma^\omega$ at steps $D+1$ through $D+1+i|\sigma|$, and let $\tau$ be the string written by $M$ at steps $D+1+i|\sigma|+1$ through $D+1+j|\sigma|$.  Then~${M(\sigma^\omega) = \eta\tau^\omega}$.  Furthermore, $\eta$ and $\tau$ can be computed from $\sigma$ by running $M(\sigma^\omega)$ for ${D+1+m|\sigma|}$~steps and recording the relevant information in a table of size $m$ and then searching that table for a repeat.  As $m$ is constant, this runs in $O(|\sigma|)$ time.
%
%
\end{proof}

Clearly, every pointwise linear-time computable function is linear-time computable; our next goal is to explain why not every linear-time computable function  $f\left[ \mathbb{D}_2\cap[0,1]\right]\subseteq \mathbb{Q}$ is pointwise linear-time computable.
As mentioned above, Kawamura, Steinberg, and Thies~\cite{kawamurafeeb} have defined linear time computability for not necessarily Lipschitz functions.
The idea is that the computation of $f(x)$ with precision $2^{-n}$ uses only $O(n)$ bits of $x$ and takes only $O(n)$ steps. This idea is formalised using second-order polynomials; we omit the formal definition (see \cite{kawamurafeeb}). The crucial difference with pointwise
linear time computability is that $f(d)$ only has to be computed with a good enough precision, even for a dyadic rational $d$.
So, for instance, if $f(0) = 0.000\ldots 001$, then we might be unable to output the exact answer very quickly because it would take too much space, but we can always output $0$ quickly and say that it is the best approximation we can get within the provided time limits. This simple idea is exploited in the proof of the proposition below.

\begin{proposition}\label{prop:counterexlinear} There is a function $f\colon [0,1]\rightarrow [0,1]$ such that $f$ is Lipschitz, $f\left[ \mathbb{Q}\cap[0,1]\right]\subseteq \mathbb{Q}$, and linear-time computable but not pointwise linear-time computable.
\end{proposition}
\begin{proof}
Let $(\tilde{J}_n)_{n\geq 3}$ be the sequence of intervals defined by $\tilde{J}_n=\left(2^{-n}-2^{-n^2},2^{-n}+2^{-n^2}\right)$. Define
\[\tilde{\delta}(\sigma)=\begin{cases}
1^n&\text{if $\overline{\sigma}\in \tilde{J}_n$ for some $n>2$},\\
0 &\text{if $\displaystyle\overline{\sigma}\not\in\bigcup_{n=3}^\infty \tilde{J}_n$.}
\end{cases}\]
Then we claim that $\tilde{\delta}(\sigma)$ can be computed in $O(|\sigma|)$ many steps. To see this, note that once we have $1^n$~on a tape, we can generate $1^{n^2}$ in $O(n^2)$ many steps, so we can determine whether $\sigma(j)=0$ for every~${n<j<\min\{|\sigma|,n^2+1\}}$ in~$O(|\sigma|)$ many steps.

Now define $\tilde{f}$ to be the function
\[\tilde{f}(x)=\begin{cases}
x+\left(2^{-n^2}-2^{-n} \right) &\text{if $\displaystyle x\in \left(2^{-n}-2^{-n^2},2^{-n}\right)$ for some $n\geq 3$,}\\
-x+\left( 2^{-n^2}+2^{-n}\right) &\text{if $\displaystyle x\in \left[2^{-n}, 2^{-n}+2^{-n^2}\right)$ for some $n\geq 3$,}\\
0 &\text{if $\displaystyle x\not\in\bigcup_{n=3}^\infty \tilde{J}_n$.}
\end{cases}
\]
This $\tilde{f}$ is clearly Lipschitz and maps rationals to rationals. Since $\tilde{f}(2^{-n})=2^{-n^2}$, $\tilde{f}$ clearly cannot be pointwise linear-time computable. We now argue that $\tilde{f}$ is linear-time computable. First observe that if $\overline{\sigma}\in \tilde{J}_n$, then either $\overline{\sigma}=2^{-n}$ or $|\sigma|> n^2$. Our goal is to show that given any $\sigma\in 2^{<\omega}$ and any $m$, we need to produce $\psi(\sigma,m)$ in $O(|\sigma|+m)$ steps.

It takes $O(|\sigma|)$ steps to first evaluate $\tilde{\delta}(\sigma)$ and to check if $\overline{\sigma}\in \tilde{J}_n$ for some $n$. It then takes $O(|\sigma|+n^2+m)$ steps to carry out the operations required to compute some value for $\psi(\sigma,m)$ close enough to $\tilde{f}(\overline{\sigma})$. This is obviously fine if $\overline{\sigma}\in\tilde{J}_n-\{2^{-n}\}$. 

In the case $\overline{\sigma}=2^{-n}$, $\tilde{f}(\overline{\sigma})=2^{-n^2}$. For each $m>n$, we have to generate $0^m$ if $m<n^2$ and $0^{n^2-1}*1$ if $m\geq n^2$ in $O(m)$ steps. But given both $1^n$ and $1^m$, it takes $O(m)$ many steps to compare~$m$ with~$n^2$ and to generate the desired output.
\end{proof}

\section{Proof of Theorem~\ref{thm:kuku}}\label{sec:kuku}

We actually prove a stronger result than Theorem \ref{thm:kuku} here: Theorem \ref{thm:kuku} only contains the equivalence of statements (i) and (ii) in Theorem \ref{thm:kuku1}. The further equivalence of statement (iii) of this theorem lets us see that nondeterminism is nothing but a pathology of the standard binary representation.

\begin{theorem}\label{thm:kuku1}
Suppose $f\colon [0,1]\rightarrow [0,1]$. The following are equivalent.
\begin{itemize}
\item[(i)] $f$ is continuous regular.
\item[(ii)] $f$ can be computed by a nondeterministic transducer for the standard binary representation.
\item[(iii)] $f$ can be computed by a deterministic transducer for the signed binary representation.
\end{itemize}
\end{theorem}
\begin{proof}

(i) $\Rightarrow$ (iii): We delay the proof of this direction to Proposition \ref{lemma:kukudeterminsticnondeterministic} in the next subsection; it is highly technical and involves a great deal of machinery different from that involved in the rest of the proof of this theorem.

(iii) $\Rightarrow$ (ii): It is easy to design a nondeterministic automaton $M$ that reads an input stream~${\alpha\in\{-1,0,1\}^\omega}$ and which writes an output stream $\beta\in 2^\omega$ such that $\overline{\beta}=\sgnd{\alpha}$  along every infinite run. We can use $M$ to convert the output written by a deterministic transducer for the signed binary representation.

(ii) $\Rightarrow$ (i): We know that $f$ is continuous by Proposition~\ref{cont}. We need to show that $f$ is regular.

Informally, given $\alpha$ and $\beta$, we will imitate the run of $M$ on $\alpha$ and see if it agrees with $\beta$. If more agreement is observed then we allow our new automaton $A_0$ to visit an acceptance state once again. Of course, $M$ could be nondeterministic, but recall that $A_0$ can be nondeterministic as well. We will therefore be able to track all possible partial runs of $M$ and code them into $A_0$ so that no possibility is missing.  The only dangerous potential problem is related to the usual pathology of the binary representation, i.e., if we encounter a potentially infinite tail of 1s in $\beta$ but then discover very late that the actual $f(\alpha)$ is only very slightly larger than $\beta$. (In Proposition~\ref{reg->tr}, we will essentially show that this pathology is the only difficulty which makes determinization impossible for transducers in general.)

Formally, suppose that $f$ is computed by a nondeterministic transducer $M$ for the standard binary representation. Suppose the states of $M$ are $s_0,\ldots,s_n$ and the transition relation is $\Delta_M$.  If $M$ has delay~$D$, we will consider the function $2^{-D}f(x)$ instead, so we can assume without loss of generality that $M$~has no delay.

We define a (nondeterministic) B\"{u}chi automaton $A_0$ accepting the graph of $f$. The states of $A_0$ are of the form $(s,X)$ where $s$ is an $M$-state and $X$ is one of the following three options: $L$ (representing ``left''), $R$ (representing ``right''), or $E$ (representing ``exact''). These states are all accepting. We also include a single nonaccepting state $t$ for $A_0$, and we denote the input stream of $A_0$ as $\alpha\oplus \beta$.

Assume that the transitions in $\Delta_M$ are of the form $(s,a,b,s')$, where $s$ is the current $M$-state, $a$ is the next input bit scanned, $b$ is the  output bit written by $M$ after scanning $a$, and $s'$ is the new $M$-state. $A_0$ will scan two binary bits $ab$ at a time (the next bit $a$ of $\alpha$ and the the next bit $b$ of $\beta$). The initial states of $A_0$ are $(s,E)$ where $s$ is an initial state of $M$, and the transitions of $A_0$ are the following:
\begin{itemize}
\item $(t, ab,t)$ for any $ab\in\{00,01,10,11\}$.
\item $\left((s,L),ab,(s',L)\right)$ if $(s,a,0,s')\in \Delta_M$ and $b=1$.
\item $\left((s,L),ab,t\right)$ if $(s,a,1,s')\in \Delta_M$ or $b=0$.
\item $\left((s,R),ab,(s',R)\right)$ if $(s,a,1,s')\in \Delta_M$ and $b=0$.
\item $\left((s,R),ab,t\right)$ if $(s,a,0,s')\in \Delta_M$ or $b=1$.
\item $\left((s,E),ab,(s',E)\right)$ if $(s,a,b,s')\in \Delta_M$.
\item $\left((s,E),ab,(s',L)\right)$ if $(s,a,1,s')\in \Delta_M$ and $b=0$.
\item $\left((s,E),ab,(s',R)\right)$ if $(s,a,0,s')\in \Delta_M$ and $b=1$.
\end{itemize}
Now we want to check that $A_0$ accepts $\alpha\oplus\beta$ if and only if $\overline{\beta}=f(\overline{\alpha})$. Suppose that
$\overline{\beta}=f(\overline{\alpha})$. Then there must be some infinite run of $M$, $s_{i(0)},s_{i(1)},\ldots$ such that $s_{i(0)}$ is an initial $M$-state and for every $k$, $\left(s_{i(k)},\alpha(k),\beta'(k),s_{i(k+1)}\right)\in\Delta_M$ for some $\beta'$ such that $\overline{\beta'}=\overline{\beta}$. Then, as
\[(s_{i(0)},E), \alpha(0)\beta'(0),(s_{i(1)},E), \alpha(1)\beta'(1),\ldots\] is an infinite run of $A_0$, we conclude that $A_0$ accepts $\alpha\oplus \beta'$. Hence if $\beta=\beta'$ then we are done, so suppose that $\beta\neq\beta'$. In that case, there is some $k$ such that either $\beta=\beta\upharpoonright k*10^\omega$ and ${\beta'=\beta\upharpoonright k*01^\omega}$ or vice versa. Assume we have $\beta=\beta\upharpoonright k*10^\omega$. Then it is straightforward to verify that
\begin{multline*}
(s_{i(0)},E), \alpha(0)\beta(0),(s_{i(1)},E), \alpha(1)\beta(1),\allowbreak  \ldots,\\
(s_{i(k)},E), \alpha(k)1, (s_{i(k+1)},R), \alpha(k+1)0, (s_{i(k+2)},R), \alpha(k+2)0,\ldots
\end{multline*}
is an infinite run of $A_0$. Hence, $A_0$ accepts $\alpha\oplus \beta$.

Now suppose that $(s_{i(0)},X_{i(0)}), \alpha(0)\beta(0),(s_{i(1)},X_{i(1)}), \alpha(1)\beta(1),\ldots$ is an infinite accepting run of an input $\alpha\oplus\beta$. Note that no infinite accepting run can encounter the state $t$, so the accepting run must be of the form above. Suppose that $X_{i(k)}=E$ for all $k$. Then $\left(s_{i(k)},\alpha(k),\beta(k),s_{i(k+1)}\right)\in\Delta_M$ for all $k$, and so $\overline{\beta}=f\left(\overline{\alpha}\right)$. Otherwise, there is some $j$ such that $X_{i(0)}=\ldots=X_{i(j)}=E$ and ${X_{i(j+1)}=X_{i(j+2)}=\ldots=L}$ or $X_{i(j+1)}=X_{i(j+2)}=\ldots=R$. Assume, without loss of generality, that $X_{i(j+1)}=L$. Then $\left(s_{i(k)},\alpha(k),\beta(k),s_{i(k+1)}\right)\in\Delta_M$ for all $k<j$. We have $(s_{i(j)},\alpha(j),1,s_{i(j+1)})\in\Delta_M$ and $\beta(j)=0$. We must also have $\left(s_{i(k)},\alpha(k),0,s_{i(k+1)}\right)\in\Delta_M$ and $\beta(k)=1$ for all $k>j$. Hence
there is an infinite run of $M$ which writes the output $(\beta\upharpoonright j)*10^\omega$ on input $\alpha$, while $\beta=(\beta\upharpoonright j)*01^\omega$. Hence, $\overline{\beta}=\overline{(\beta\upharpoonright j)*10^\omega}=f\left(\overline{\alpha}\right)$.\end{proof}
		
\subsection{The technical propositions}	

Proposition \ref{reg->tr}  and its proof  will be
used heavily in the proof of the main technical  Proposition \ref{lemma:kukudeterminsticnondeterministic}. 	
Even though this proposition gives us (i) $\Rightarrow$ (ii) in Theorem \ref{thm:kuku1}, we do not need it to establish this implication of the theorem since we have (iii) $\Rightarrow$ (ii) and (ii) $\Rightarrow$ (i) already, and Proposition \ref{lemma:kukudeterminsticnondeterministic} will give us (i) $\Rightarrow$ (iii). It may be instructive to compare it with Lemma~\ref{lem:deterministic}.

\begin{proposition}\label{reg->tr}
Each continuous and regular $f\colon  [0,1] \rightarrow [0,1]$ can be computed by a nondeterministic transducer for the standard binary representation.
\end{proposition}
\begin{proof}[Proof of Proposition~\ref{reg->tr}] We will describe a nondeterministic transducer computing $f$. We require that for each input $x$ and each infinite run corresponding to $x$, the output from this run is a binary representation for $f(x)$.

The \emph{informal idea} of this proof is as follows. By Lemma~\ref{lem:deterministic}, we can assume the automaton witnessing regularity is deterministic.
 By  Lemma \ref{lemlip}, $f$ has to be Lipschitz. Thus,
the naive hope is that it should be possible to decide the next bit of $f(x)$, say the $i^{\text{th}}$ bit, by looking ahead at the next $n$~bits of~$x$, where $n$ does not depend on $i$.   We scan all possible runs of the automaton on strings of length~$Cn$, where~$C>1$ will depend on the size of the automaton in the spirit of the proof of Lemma \ref{lemlip},  and see which one hits an accepting state twice. This should allow us to decide what $f(x)$ is with good precision. Of course, the problem is that after several iterations of this strategy, we may see that both $0.011111$ and $0.100000$ could serve as $f(x)$ because we have not yet scanned $x$ far enough.
Then, at a very late stage we may discover that the true output has the form $0.1000\ldots01$. Thus, if we decide to stick with $0.01111\ldots1$, we will be in trouble. Similarly, if we stick with $0.1000\ldots0$ and the actual output is
$0.01111\ldots10$, we will not be able to go back. Thus, we will have to keep both possibilities open, and this is why the resulting transducer  will be nondeterministic at the end.

To make these ideas formal,
we will apply the proof for Lemma \ref{lemlip} here, using the ``memoryless" property of an automaton. We will need to rephrase the argument used in Lemma \ref{lemlip} appropriately; the notation we use here is discussed in the proof of that lemma.

For each $d$, consider the tree with height $d+1$, $N$ many branches at the root, and binary branching thereafter. It is easy to see that given any $n$, $d$, and $\sigma_0,\sigma_1$ of lengths $n$ and $d$ respectively, there is a subtree $\hat{T}_{\sigma_1}$ of the abovementioned tree such that each string $k*\tau$ on~$\hat{T}_{\sigma_1}$ corresponds to a string $\eta$ on~$T_{\sigma_0*\sigma_1}$ such that $\ell\left(\sigma_0,\eta\upharpoonright n\right)=k$ and $\ell\left(\sigma_0*\sigma_1,\eta\right)=\ell\left(\sigma_1,k*\tau\right)$. (Here, we assume that $\{1,2,\ldots,N\}$ are the states of $A$ and that $\ell\left(\sigma_1,k*\tau\right)$ for $\hat{T}_{\sigma_1}$ is defined in the obvious way, with $k$ representing the starting state used to scan $\sigma_1*\tau$.) This correspondence is one-to-one. Therefore, if we are only concerned with the strings of $T_{\sigma_0*\sigma_1}$ of length greater than $n$, we will only need the parameters $d$, $\sigma_1$ and the finite set $X=\{\ell(\sigma_0,\eta)\mid\eta\in T_{\sigma_0}$ has length $n_0\}$; the string $\sigma_0$ is irrelevant, as the strings extending the leaves of $T_{\sigma_0}$ depend only on $X$ and $\sigma_1$.

With this setup, Claim \ref{claim:kukumain} may now be rephrased as follows:

\begin{quote}
Given any string $\hat{\sigma}$ of length $D$ and any subset $X$ of $\{1,2,\ldots,N\}$ of ``reachable" states, there are strings $\hat{\eta}_0\subset\hat{\eta}_1\subset\ldots\subset \hat{\eta}_{2N}$ such that $3<|\hat{\eta}_0|<\ldots<|\hat{\eta}_{2N}|<D$ and some $k\in X$ such that for all $j,j'\leq 2N$, $k*\hat{\eta}_j\in\hat{T}_{\hat{\sigma}}$ and $\ell\left(\hat{\sigma},k*\hat{\eta}_j\right)=\ell\left(\hat{\sigma},k*\hat{\eta}_{j'}\right)$ is accepting.
\end{quote}

Now, given the parameters $\hat{\sigma}$ and $X$, let $\nu(\hat{\sigma},X)$ encode the first two digits of $\hat{\eta}_0$ as well as all the states on the second level of the tree $\hat{T}_{\hat{\sigma}}$. The function $\nu$ contains only finitely much information and may therefore be used as memory in the transducer $\hat{A}$ we are about to define.

The transducer $\hat{A}$ works as follows. The states of $\hat{A}$ encode a set $X\subseteq\{1,\ldots,N\}$, a binary string $\gamma$ of length $2$, and a symbol $S$ from the set $\{L,R\}$. The meanings of the parameters are as follows. $X$ stands for the set of states on the leaves of the tree $T_{\sigma}$, where $\sigma$ is the input read so far by $\hat{A}$ ($\hat{A}$ of course does not remember all of $\sigma$, only the last $D$ bits it has most recently read). $\gamma$ stands for the first two bits of $\eta_0$ produced at the previous step, and the symbols $L$ and $R$ mean that the state is recording the left path of a potential output $\beta*0*111\ldots$ or the right path $\beta*1*000\ldots$, respectively. The nondeterminism of $\hat{A}$ is used here to allow us to keep both binary representations of a dyadic rational alive as potential outputs.

The initial state of $\hat{A}$ is the state encoding $X=\{1\},\gamma=00$, and $S=R$, where $1$ is the initial state of $A$.
From this initial state, it scans the first $D+1$ bits of the input $\hat{\sigma}$. Let $X'$ and $\gamma'$ be the output of $\nu(\hat{\sigma},\{1\})$. Without loss of generality, we assume that $Rng(f)\subseteq \left(\frac{1}{4},\frac{1}{2}\right)$ and that $\gamma'=01$ (we will later verify that this choice of $\gamma'$ works). 
We allow the initial state $\left(\{1\},00,R\right)$ to transit (nondeterministically) to $\left(X',01,L\right)$ and write the output $0$ and to $\left(X',01,R\right)$ and write the output $1$.

Now we describe the transitions from a noninitial state. Consider the state $\left(X,\gamma,S\right)$ and assume that $A$ has just read the next bit of the input and that $\hat{\sigma}$ is the last $D$ bits of the input read by $A$. Let $X'$ and $\gamma'$ be the output of $\nu(\hat{\sigma},X)$. Our transitions are listed in the table below. The second, third, and fourth columns represent the values of $S,\gamma$ and $\gamma'$ respectively. The last column represents the state (or states) that we transit to and the corresponding output bit to be written. For instance, the third row represents the situation where we have $S=L$, $\gamma$ is either $01$ or $11$, and $\gamma'$ is $01$ or $10$. In this case, we allow nondeterministic transitions to the state $\left(X',\gamma',L\right)$ where we write output bit $0$ as well as to $\left(X',\gamma',R\right)$ where we write output bit $1$.


\begin{center}
\begin{tabular}{|c|c|l|l|l|}
\hline
{\hspace*{-0.17cm}\small Row} & $S$ & \multicolumn{1}{c|}{$\gamma$} & \multicolumn{1}{c|}{$\gamma'$} & New state(s) and output bit(s) \\\hline
1 & $L$ &  $00,01,10,11$ & $00,11$ & $\left(X',\gamma',L;1\right)$ \\
2 & $L$ &  $00,10$ & $01,10$ & run dies \\
3 & $L$ &  $01,11$ & $01,10$ & $\left(X',\gamma',L;0\right),\left(X',\gamma',R;1\right)$ \\\hline
4 & $R$ &  $00,01,10,11$ & $00,11$ & $\left(X',\gamma',R;0\right)$ \\
5 & $R$ &  $00,10$ & $01,10$ & $\left(X',\gamma',L;0\right),\left(X',\gamma',R;1\right)$ \\
6 & $R$ &  $01,11$ & $01,10$ & run dies \\\hline
\end{tabular}
\end{center}
Given any finite run $\left(X_0,\gamma_0,S_0\right),\ldots,\left(X_{n-1},\gamma_{n-1},S_{n-1}\right)$ of length $n\geq 2$, there is a maximal $j<n-1$ such that the transition from $\left(X_j,\gamma_j,S_j\right)$ to $\left(X_{j+1},\gamma_{j+1},S_{j+1}\right)$ is nondeterministic (corresponding to row~$3$ or~$5$). In this case, we say that the run is \emph{$j$-splitting}. Clearly, we must have $S_{j+1}=\ldots=S_{n-1}$, and all transitions
from $\left(X_{i},\gamma_i,S_i\right)$ to $\left(X_{i+1},\gamma_{i+1},S_{i+1}\right)$ for $j<i<n-1$ are of type $1$ or $4$. It is clear that both
\[
\begin{aligned}
\left(X_0,\gamma_0,S_0\right),\ldots,\left(X_{j},\gamma_j,S_j\right),\left(X_{j+1},\gamma_{j+1},L\right),\left(X_{j+2},\gamma_{j+2},L\right),\ldots,
\left(X_{n-1},\gamma_{n-1},L\right)\rlap{\text{ and}}\\
\left(X_0,\gamma_0,S_0\right),\ldots,\left(X_{j},\gamma_j,S_j\right),\left(X_{j+1},\gamma_{j+1},R\right),\left(X_{j+2},\gamma_{j+2},R\right),\ldots,
\left(X_{n-1},\gamma_{n-1},R\right)
\end{aligned}
\]
are runs of length $n$. We call pairs of runs such as these \emph{twin runs}. Furthermore, note that if $j<n-2$, then $\gamma_{i}\in\{00,11\}$ for all $j+2\leq i\leq n-1$.

We now prove that $\hat{A}$ computes the function $f$. For the rest of this proof, fix a real number $x$ as input, a binary representation $\alpha$ of $x$, and a binary representation $\beta$ of $f(x)$. We first show that there is an infinite run corresponding to the input $\alpha$. By compactness, it suffices to show that for every $n\geq 2$, there is a run of $\hat{A}$ of length $n$ beginning with the initial state $\left(\{1\},00,R\right)$ and where the first $n+D-1$ many bits of $\alpha$ is read during the run (recall that the length of a run is always one more than the number of output bits written by the run, and that $\hat{A}$ has a delay of $D$). 
 Suppose that $\left(X_0,\gamma_0,S_0\right),\ldots,\left(X_{n-1},\gamma_{n-1},S_{n-1}\right)$ is a run of length $n$. Every combination of $S_{n-1},\gamma_{n-1},\nu(\alpha,X_{n-1})$ listed in the table will allow the current run to be extended except for rows $2$ and $6$. However, in those two cases, the corresponding twin run will be extendible, and therefore there must be some run of length $n+1$.

	
Next we have to show that any infinite run produced by $\hat{A}$ while scanning input $\alpha$ must write an output representing $f(x)$. To do this, we prove the following claim:

\begin{claim}\label{claim:kukutransducer}
Suppose that  $\left(X_0,\gamma_0,S_0\right),\ldots,\left(X_{n-1},\gamma_{n-1},S_{n-1}\right)$ is a run of length $n\geq 2$ produced by $\hat{A}$ while scanning $\alpha$, where $X_0=\{1\},\gamma_0=00$ and $S_0=R$, and $\tau$ is the output written by this run. Suppose this run is $j$-splitting for some $j<n-1$. Then
\begin{itemize}
\item[(i)] If $S_{n-1}=L$ then $\tau=(\tau\upharpoonright j)*01^{n-j-2}$, and if $S_{n-1}=R$ then $\tau=(\tau\upharpoonright j)*10^{n-j-2}$.

\item[(ii)] $X_{n-1}$ is the set $\{\ell(\alpha,\rho)\mid \rho\in T_{\alpha\upharpoonright (n-1)}$ and $|\rho|=n-1\}$.

\item[(iii)] If $\gamma_{n-1}=00$ or $10$ then
\[\overline{(\tau\upharpoonright j)*10^{n-j-1}}-2^{-n-2}\leq \overline{\beta\upharpoonright (n+2)}\leq \overline{(\tau\upharpoonright j)*10^{n-j-1}}+2^{-n},\] and if $\gamma_{n-1}=01$ or $11$ then
\[\overline{(\tau\upharpoonright j)*01^{n-j-1}}-2^{-n-2}\leq \overline{\beta\upharpoonright (n+2)}\leq \overline{(\tau\upharpoonright j)*01^{n-j-1}}+2^{-n}.\]
\end{itemize}
\end{claim}
\begin{proof}(i) and (ii) can be verified easily by definition chasing. We prove (iii) by induction over $n$. The basis is $n=2$, where we have the run $(\{1\},00,R),$ $(X_1,01,S_1)$, which is of course $0$-splitting. We have to verify that $\overline{01}-2^{-4}\leq \overline{\beta\upharpoonright 4}\leq \overline{01}+2^{-2}$, which is the same as $\frac{3}{16}\leq\beta\upharpoonright 4\leq \frac{1}{2}$, which is true since we assumed that $f(x)\in\left(\frac{1}{4},\frac{1}{2}\right)$.

Now suppose that (iii) holds for $n\geq 2$. We fix a run $\left(X_0,\gamma_0,S_0\right),\ldots,\left(X_{n},\gamma_{n},S_{n}\right)$ of length $n+1$. First assume that $\gamma_{n}\in \{00, 10\}$ and that the run is $j$-splitting for some $j<n-1$. This means that the run $\left(X_0,\gamma_0,S_0\right),\ldots,\left(X_{n-1},\gamma_{n-1},S_{n-1}\right)$ is also $j$-splitting. Since $\gamma_n$ is either $00$ or $10$ and $j<n-1$, by examining the table, we see that $\gamma_n$ must in fact be $00$, and the very last transition $\left(X_{n-1},\gamma_{n-1},S_{n-1}\right)$ to $\left(X_{n},\gamma_{n},S_{n}\right)$ corresponds to row $1$ or $4$. This means that $\gamma_n=00$ is the output of $\nu(\hat{\sigma},X_{n-1})$, and by~(ii), there is some string $\tilde{\gamma}$ of length $n-1$ such that $\tilde{\gamma}*\gamma_n$ extends to a chain of $2N+1$ many nodes with accepting states on $T_{\alpha\upharpoonright (n-1+D)}$. Now apply the inductive hypothesis on $\gamma_{n-1}$; there are two cases. First, suppose that $\gamma_{n-1}\in\{00,10\}$, then, we have $\overline{(\tau\upharpoonright j)*10^{n-j-1}}-2^{-n-2}\leq \overline{\beta\upharpoonright (n+2)}\leq \overline{(\tau\upharpoonright j)*10^{n-j-1}+2^{-n}}$. Applying Claim \ref{claim:kuku1} to $\tilde{\gamma}*\gamma_n$ limits the possible values of $\overline{\beta\upharpoonright (n+2)}$ relative to $\overline{\tilde{\gamma}*\gamma_n}$, and using the abovementioned bounds on $\beta\upharpoonright (n+2)$, we conclude that $\tilde{\gamma}=(\tau\upharpoonright j)*10^{n-j-2}$. If $\gamma_{n-1}\in\{01,11\}$, then we can also conclude by a similar argument that $\tilde{\gamma}=(\tau\upharpoonright j)*10^{n-j-2}$. Since $\gamma_n=00$, applying Claim \ref{claim:kuku1} to $\tilde{\gamma}*\gamma_n$ gives us the bounds $\overline{(\tau\upharpoonright j)*10^{n-j}}-2^{-n-3}\leq \overline{\beta\upharpoonright (n+3)}\leq \overline{(\tau\upharpoonright j)*10^{n-j}}+2^{-n-1}$, which is what we need. If $\gamma_n\in\{01,11\}$, then a symmetric argument holds, thus, the inductive step can be verified in the case where $j<n-1$.

Now we assume that $\gamma_{n}\in\{00,10\}$ and that the run is $(n-1)$-splitting. As above, fix $\tilde{\gamma}$ of length~${n-1}$ such that $\tilde{\gamma}*\gamma_n$ extends to a chain of $2N+1$ many nodes with accepting states on $T_{\alpha\upharpoonright (n-1+D)}$. Fix~${j'<n-1}$ such that $\left(X_0,\gamma_0,S_0\right),\ldots,\left(X_{n-1},\gamma_{n-1},S_{n-1}\right)$ is $j'$-splitting. Applying the inductive hypothesis to $\gamma_{n-1}$, again there are two cases, depending on the value of $\gamma_{n-1}$. First suppose that $\gamma_{n-1}$ is either $00$ or $10$. Then since the last transition is a nondeterministic transition, this means that $\gamma_{n}$ must be $10$ and~$S_{n-1}=R$. By a similar argument as above, we can conclude that $\tilde{\gamma}=(\tau\upharpoonright j')*10^{n-j'-2}$. Since $S_{n-1}=R$, $\tau$ must extend $(\tau\upharpoonright j')*10^{n-j'-2}$ (by (i)). Together with the fact that $\gamma_n=10$, applying Claim~\ref{claim:kuku1} to~$\tilde{\gamma}*\gamma_n$ gives us the bounds $\overline{(\tau\upharpoonright n-1)*10}-2^{-n-3}\leq \overline{\beta\upharpoonright (n+3)}\leq \overline{(\tau\upharpoonright n-1)*10}+2^{-n-1}$, which is what we need. The other cases are symmetric. This concludes the induction, and hence (iii) holds.
\end{proof}
Now consider an infinite run produced by $\hat{A}$ while scanning input $\alpha$ and let $\tilde{\beta}$ be the output of this run. If there are infinitely many initial segments of this run which split at different positions, then  we claim that $\beta\supset(\tilde{\beta}\upharpoonright j)$ at each of these splitting positions, and therefore $\beta=\tilde{\beta}$. To see this, fix a splitting position $j$ along the infinite run. If there exists a least number $j'<j$ such that $\beta(j')\neq \tilde{\beta}(j')$, consider the following two cases. First if $\beta(j')=0$ then $\overline{\beta\upharpoonright j}\leq \overline{(\beta\upharpoonright j')*10^\omega}$, and since $\tilde{\beta}(j')=1$ we have $\overline{\beta\upharpoonright j}\leq \overline{(\tilde{\beta}\upharpoonright j)*0^\omega}$. On the other hand, if $\beta(j')=1$ then $\overline{\beta\upharpoonright j}\geq \overline{(\beta\upharpoonright j')*01^\omega}$, and since $\tilde{\beta}(j')=0$ we have $\overline{\beta\upharpoonright j}\geq \overline{(\tilde{\beta}\upharpoonright j)*1^\omega}$. Now applying Claim \ref{claim:kukutransducer}(iii) on a suitable $n>j$ tells us that either $\overline{(\tilde{\beta}\upharpoonright j)*10^{n-j-1}}-2^{-n-2}\leq \overline{\beta\upharpoonright (n+2)}\leq \overline{(\tilde{\beta}\upharpoonright j)*10^{n-j-1}}+2^{-n}$ or $\overline{(\tilde{\beta}\upharpoonright j)*01^{n-j-1}}-2^{-n-2}\leq \overline{\beta\upharpoonright (n+2)}\leq \overline{(\tilde{\beta}\upharpoonright j)*01^{n-j-1}}+2^{-n}$ holds. However note that this is a contradiction since the left side of both sets of inequalities is strictly greater than $\overline{(\tilde{\beta}\upharpoonright j)*0^\omega}$, the the right side of both sets of inequalities is strictly smaller than $ \overline{(\tilde{\beta}\upharpoonright j)*1^\omega}$. Therefore we conclude that $\beta=\tilde{\beta}$.

Now suppose there are only finitely many splitting positions, and let $j$ be the final splitting position along the infinite run. By Claim \ref{claim:kukutransducer}(i), $\tilde{\beta}=(\tilde{\beta}\upharpoonright j)*01^\omega$ or $\tilde{\beta}=(\tilde{\beta}\upharpoonright j)*10^\omega$, and by Claim \ref{claim:kukutransducer}(iii) (as $n\rightarrow\infty$), we conclude that $\beta$ and $\tilde{\beta}$ represent the same real (in fact, rational) number.
%
%

This concludes the proof of Proposition  \ref{reg->tr}.
\end{proof}

\begin{proposition}\label{lemma:kukudeterminsticnondeterministic}
Let $f\in C[0,1]$ and is regular. Then $f$ can be computed by a deterministic transducer for the signed binary representation.
\end{proposition}

\begin{proof}
Let $f\in C[0,1]$ be regular. By Proposition \ref{reg->tr} we fix the nondeterministic transducer $\hat{A}$ for the standard binary representation computing $f$.  Our goal is to produce a deterministic transducer $A_d$ for the signed binary representation which computes $f$.

The \emph{informal idea} behind the proof below is based on the following intuition. Recall that the trouble in the proof of Proposition \ref{reg->tr} was that we could not quickly decide between the outputs with long tails of $1$s vs.~the outputs with long tails  of $0$s. The intuition is that the signed binary representation will allow us correct the errors online: if we have to go back from, say, $0.100\ldots00$ to
$0.111\ldots10$ we can instantly correct the error by outputting $0.100\ldots00(-1)$. Thus, if we have to decide between two such potential outputs we stick with $0.100\ldots00$.
To make sure we are dealing with a transducer which does not have any other issues unrelated to the pathology described above, we will be using the same automaton $\hat{A}$ as in the proof of Proposition \ref{reg->tr}. Of course, the input will also have to be signed binary. We will discuss this and further tensions in detail later.

We will make use of the properties of $\hat{A}$ proved in Proposition~\ref{reg->tr}; as in  Proposition~\ref{reg->tr} we assume that $Rng(f)\subseteq\left(\frac{1}{4},\frac{1}{2}\right)$.
Since we have to convert between the standard and signed binary representations of $[0,1]$, the following definition will be convenient: Given any $\eta\in\{-1,0,1\}^{<\omega}$, let $\Dy(\eta)$ be the unique $\sigma\in 2^{|\eta|}$ such that $\overline{\sigma}=\max\left\{\sgnd{\eta},0\right\}$. Note that this cannot be done in an online way.

\subsubsection*{Recalling $\hat{A}$} Since the definition of $A_d$ will rely heavily on $\hat{A}$, we will first recall some properties of $\hat{A}$ and prove several additional facts about $\hat{A}$ which will be used in the definition of $A_d$.

Recall that $\hat{A}$ has a delay of $D$ (for some constant $D$), i.e., for each $n$, $\hat{A}$ produces runs of length $n+1$ and writes an output of length $n$ after scanning $D+n$ many bits of the input. In fact, by Claim \ref{claim:kukutransducer}, for every $\sigma\in 2^{<\omega}$, there are exactly two runs
\[\begin{aligned}\left(X_0,\gamma_0,S_0\right),\ldots,\left(X_{j},\gamma_j,S_j\right), \left(X_{j+1},\gamma_{j+1},L\right),\left(X_{j+2},\gamma_{j+2},L\right),\ldots,
\left(X_{n},\gamma_{n},L\right)\rlap{\text{ and}}\\
\left(X_0,\gamma_0,S_0\right),\ldots,\left(X_{j},\gamma_j,S_j\right),\left(X_{j+1},\gamma_{j+1},R\right),\left(X_{j+2},\gamma_{j+2},R\right),\ldots,
\left(X_{n},\gamma_{n},R\right)\end{aligned}\] of length $n+1$, where $|\sigma|=n+D$. Let $\mathtt{out_L}(\sigma)$ and $\mathtt{out_R}(\sigma)$ be the output produced by these two runs respectively. Claim  \ref{claim:kukutransducer} says that if these runs are $j$-splitting, where $j<n$, then $\mathtt{out_L}(\sigma)=\tau*01^{n-j-1}$ and $\mathtt{out_R}(\sigma)=\tau*10^{n-j-1}$, where $|\tau|=j$ and $\tau=\mathtt{out_L}\left(\sigma\upharpoonright D+j\right)$ or $\mathtt{out_R}\left(\sigma\upharpoonright D+j\right)$. Henceforth we will say that \emph{$\sigma$ is $j$-splitting}.

Recall that the states of $\hat{A}$ are of the form $(X,\gamma,S)$; for brevity we write $\mathtt{\hat{S}}$ instead of $(X,\gamma,S)$. At any point there are exactly two runs of $\hat{A}$ which are alive;  the ``current state'' of $\hat{A}$ is not unique since $\hat{A}$ is a nondeterministic automaton. However, we can denote the ``current state'' (noninitial) of $\hat{A}$ by the pair of states $\Shat_L,\Shat_R$ corresponding to the current state of each of the living runs. Given any $b\in\{0,1\}$ and any $\sigma\in 2^{<\omega}$ with $|\sigma|>D$, let $\Shat_L,\Shat_R$ be the current state of $\hat{A}$ after reading $\sigma$. Suppose that $\mathtt{out_L}(\sigma)=\zeta*01^k$ and $\mathtt{out_R}(\sigma)=\zeta*10^k$ for some common prefix $\zeta$ and some $k\geq 0$. Recall that there are three possibilities for the current state of $\hat{A}$ after reading the next symbol $b$:
\begin{itemize}
\item[T(i)] Both runs make deterministic transitions. Then the left run transits from $\Shat_L$ to $\Shat_L^*$, and the right run transits from $\Shat_R$ to $\Shat_R^*$. In this case,  $\mathtt{out_L}(\sigma*b)=\mathtt{out_L}(\sigma)*1$ and $\mathtt{out_R}(\sigma*b)=\mathtt{out_R}(\sigma)*0$. More specifically, we will have $\mathtt{out_L}(\sigma*b)=\zeta*01^{k+1}$ and $\mathtt{out_R}(\sigma*b)=\zeta*10^{k+1}$.

\item[T(ii)] The left run makes a nondeterministic transition and the right run dies. Then the left run transits from $\Shat_L$ to $\Shat_L^*$ and $\Shat_R^*$. In this case,  $\mathtt{out_L}(\sigma*b)=\mathtt{out_L}(\sigma)*0$ and $\mathtt{out_R}(\sigma*b)=\mathtt{out_L}(\sigma)*1$. More specifically, we will have $\mathtt{out_L}(\sigma*b)=\zeta*01^{k}*0$ and $\mathtt{out_R}(\sigma*b)=\zeta*01^{k+1}$. 
\item[T(iii)] The right run makes a nondeterministic transition and the left run dies, symmetric to case T(ii). In this case,  $\mathtt{out_L}(\sigma*b)=\mathtt{out_R}(\sigma)*0$ and $\mathtt{out_R}(\sigma*b)=\mathtt{out_R}(\sigma)*1$. More specifically, we will have $\mathtt{out_L}(\sigma*b)=\zeta*10^{k+1}$ and $\mathtt{out_R}(\sigma*b)=\zeta*10^{k}*1$. 
\end{itemize}
We make a technical remark here: Since $\hat{A}$ has delay $D$, each state $\Shat$ of $\hat{A}$ also has to record the history of the most recent $D$ bits of the input scanned by $\hat{A}$; we assume that this is the case although this is not explicitly encoded in the triple  $(X,\gamma,S)$.

\subsubsection*{A discussion of problems faced in simulating $\hat{A}$}

The rough idea is for $A_d$ to read the next bit of the input $\eta$ and output the next bit of $\mathtt{out_R}(\Dy(\eta))$ by simulating the action of $\hat{A}$ on $\Dy(\eta)$. In other words, we wish for $A_d$ to write an output string $\nu$ such that $\sgnd{\nu}=\mathtt{out_R}(\Dy(\eta))$. The reason for this is that if the next bit of the input read is $b=0$ or $b=1$, then $\Dy(\eta*b)=\Dy(\eta)*b$, which means that $\mathtt{out_R}(\Dy(\eta*b))$ will be equal to either $\mathtt{out_R}(\Dy(\eta))*0$, $\mathtt{out_R}(\Dy(\eta))*1$ or $\mathtt{out_L}(\Dy(\eta))*1$, corresponding to $\hat{A}$ having a transition of type T(i), T(iii) or T(ii) respectively upon reading symbol $b$. $A_d$ can then write $b'=0,1$ or $-1$ as the next bit of the output respectively, and we will be able to maintain  $\sgnd{\nu*b'}=\mathtt{out_R}(\Dy(\eta*b))$.

However, the main difficulty arises when the next input bit read is $b=-1$. In this case, $\Dy(\eta*b)\nsupseteq \Dy(\eta)$, and in order to simulate $\hat{A}$ on the input $\Dy(\eta*b)$, we shall need $A_d$ to record certain information about the status of $\hat{A}$ in its limited memory in order for it to simulate $\hat{A}$. The least $A_d$ needs to do is to remember the state that $\hat{A}$ is in after scanning  $\Dy(\eta*(-1))$. Since $A_d$ has limited memory, it needs to keep track of this additional information regardless of what the input bit currently being read is; when we next see an input bit $-1$, we can immediately switch to simulating $\hat{A}$ on the input $\Dy(\eta*(-1))$. The second thing we need $A_d$ to remember is the position of $\mathtt{out_R}(\Dy(\eta*(-1)))$ relative to $\mathtt{out_R}(\Dy(\eta))$; this information is obviously needed by $A_d$ when it next finds $b=-1$ because it needs to know which bit to output next in order to maintain $\sgnd{\nu*b'}=\mathtt{out_R}(\Dy(\eta*b))$.

Fortunately, by the following claim, $\mathtt{out_R}(\Dy(\eta*(-1))\upharpoonright n+D)$ is either equal to, immediately to the left, or immediately to the right of $\mathtt{out_R}(\Dy(\eta)\upharpoonright n+D)$. Thus to record the relative position of $\mathtt{out_R}(\Dy(\eta))$ and $\mathtt{out_R}(\Dy(\eta*(-1)))$ we only need a single ternary bit, irrespective of the length of $\eta$.


\begin{claim}\label{claim:kukurelativeposition} Suppose that $\sigma_0,\sigma_1\in 2^{n+D}$ such that $|\overline{\sigma_0}-\overline{\sigma_1}|\leq 2^{-n-D}$. Then $|\overline{\mathtt{out_R}(\sigma_0)}-\overline{\mathtt{out_R}(\sigma_1)}|\leq 2^{-n}$.
\end{claim}
\begin{proof} Since $|\overline{\sigma_0}-\overline{\sigma_1}|\leq 2^{-n-D}$, for each $i=0,1$ there is some $\alpha_i\supset\sigma_i$ such that $\alpha_i\in 2^\omega$, and~${\overline{\alpha_0}=\overline{\alpha_1}}$. Let~${\beta_i\in 2^\omega}$ be the output produced by $\hat{A}$ on input $\alpha_i$, for $i=0,1$. If $\overline{\mathtt{out_R}(\sigma_0)}<\overline{\mathtt{out_R}(\sigma_1)}-2^{-n}=\overline{\mathtt{out_L}(\sigma_1)}$, we can apply Claim \ref{claim:kukutransducer} to get $\overline{\beta_0\upharpoonright n+3}\leq \overline{\mathtt{out_R}(\sigma_0)}+2^{-n-1}$ and $\overline{\beta_1\upharpoonright n+3}\geq \overline{\mathtt{out_L}(\sigma_1)}$. However, as $\overline{\mathtt{out_R}(\sigma_0)}\leq \overline{\mathtt{out_L}(\sigma_1)}-2^{-n}$, we see that $\overline{\beta_0\upharpoonright n+3}<\overline{\beta_1\upharpoonright n+3}-2^{-n-3}$, and so $\overline{\beta_0}\neq\overline{\beta_1}$, which is impossible. Thus $\overline{\mathtt{out_R}(\sigma_0)}\geq\overline{\mathtt{out_R}(\sigma_1)}-2^{-n}$. Apply the above symmetrically to conclude that ${|\overline{\mathtt{out_R}(\sigma_0)}-\overline{\mathtt{out_R}(\sigma_1)}|\leq 2^{-n}}$.
\end{proof}

Unfortunately, keeping track of the state of $\hat{A}\left(\Dy(\eta*(-1))\right)$ and the position of its output is still not enough. Let's illustrate with a potentially problematic scenario. Suppose $A_d$ has seen $\eta$ on its input tape and written $\nu$ on the output tape and successfully ensured that $\sgnd{\nu}=\mathtt{out_R}(\Dy(\eta))$ holds so far. Suppose the next input bit read is $0$, and that $\hat{A}$ sees a transition of type T(ii) when reading this next input bit. This means that the right run is reset, i.e., $\mathtt{out_R}(\Dy(\eta*0))=\mathtt{out_L}(\Dy(\eta))*1$, and in order to maintain $\sgnd{\nu*b'}=\mathtt{out_R}(\Dy(\eta*0))$ we have to make $A_d$ write output bit $b'=-1$ (a transition of type T(ii) causes the value of $\outr$ to decrease, i.e., $\overline{\mathtt{out_R}(\Dy(\eta*0))}<\overline{\mathtt{out_R}(\Dy(\eta))}$). Now suppose that the next input bit read by $A_d$ after that is $-1$. Since $\Dy(\eta*0*(-1))\nsupset \Dy(\eta*0)$ it could be that along the input $\Dy(\eta*0*(-1))$ we do not make any transitions that resets the right run and therefore
%
%
%
%
it is possible that $\mathtt{out_R}(\Dy(\eta*0*(-1)))=\mathtt{out_R}(\Dy(\eta))*00$.
%
Unfortunately, in this case there is no way we can ensure
$\sgnd{\nu*(-1)*b'}=\mathtt{out_R}(\Dy(\eta*0*(-1)))$, as the next bit of our output is too insignificant.

The problem above was due to the fact that we reacted to the input bit $0$ too hastily. If $A_d$ had scanned two input bits instead of just one, it would have found that the next two input bits were $0$ followed by $-1$, and would have written $\nu*0$ instead of $\nu*(-1)$. However, scanning more input bits at once does not appear to help as well, because we might find an arbitrarily long sequence of $0$s before finally seeing a $-1$. Since $A_d$ has a fixed delay, the problem above appears to be still there. Fortunately, our solution to this problem will come from the next claim.

The next claim says that there cannot be three consecutive transitions of $\hat{A}$ where we kill the left run (or the right run). This is because if there was such a sequence of transitions where we kill, for instance, the left run three times consecutively, then the output will extend a string of the form $\tau*111$. However, the transducer $\hat{A}$ was designed so that if the output was of this form, then the input would remain $j$-splitting for some $j<|\tau|$, and we would not be able to make such transitions after all.

\begin{claim}\label{claim:kukunondeterministictwice} There cannot be three consecutive transitions of type T(ii) nor three consecutive transitions of type T(iii).
\end{claim}
\begin{proof} Without loss of generality, we assume that there are three consecutive transitions of type T(ii). Let $\sigma_0\subset \sigma_1$ be such that $|\sigma_1|=|\sigma_0|+3$ and $|\sigma_0|=D+n-1$ for some $n\geq 2$, and ${\mathtt{out_R}(\sigma_1)=\mathtt{out_L}(\sigma_0)*001}$. Note that $\left|\mathtt{out_L}(\sigma_0)\right|=n-1$. Let $\beta\in 2^\omega$ be a representation for $f(\overline{\sigma_1})$. By applying Claim \ref{claim:kukutransducer}, we obtain $\overline{\beta\upharpoonright n+2}\geq \overline{\mathtt{out_L}(\sigma_0)*011}$. However, we must also have $\beta\supset\mathtt{out_R}(\sigma_1)$ or $\beta\supset\mathtt{out_L}(\sigma_1)$, so ${\overline{\beta\upharpoonright n+2}\leq\overline{\mathtt{out_R}(\sigma_1)}}$. However, this is impossible since $\mathtt{out_R}(\sigma_1)=\mathtt{out_L}(\sigma_0)*001$.
\end{proof}

Now our solution to the problem is to have $A_d$ scan the next $4$ bits of the input  (in fact, $D+4$ bits, because $A_d$ has to simulate $\hat{A}$ which already has a delay of $D$) before deciding on the next output bit. Again, suppose that $A_d$ has read the input $\eta$ and written an output $\nu$ such that
$\sgnd{\nu}=\mathtt{out_R}(\Dy(\eta))$. If the next four bits are not $0000$ then the problem above does not occur, as any subsequent bit of $-1$ will not result in a carry that invalidates the current output. Therefore the most risky situation is that the next four bits scanned are $0000$, followed by a $-1$ bit. 
Now clearly the fact that $\Dy(\eta*0000*(-1))\upharpoonright |\eta|\neq\Dy(\eta)$ is going to potentially cause us some problems, so let's review the situation described above again more carefully.

Having scanned four more bits of the input, we saw the input sequence $\eta$ being extended to $\eta*0000$. However assume that $\hat{A}$ makes a transition of type T(ii) when reading the last bit of $\Dy(\eta*0)$. Since we have no evidence that  the input stream is going to decrease in value,  we have no choice but to make $A_d$ write $b'=-1$ (as explained above). This will allow us to keep our output $\nu*(-1)$ at the same value as $\outr{(\Dy(\eta*0))}\supset \outl{(\Dy(\eta))}$, which is fine for now. However it turned out that the \emph{next} bit read is $-1$, which would mean that the input stream, when converted to binary, got ``reset'', i.e., $\Dy(\eta*0000*(-1))\nsupset\Dy(\eta)$. This would mean that potentially, the right side output of this new binary string $\Dy(\eta*0000*(-1))$ could end up following $\outr{(\Dy(\eta))}$ instead of $\outl{(\Dy(\eta))}$. Since we have already written $-1$ on our output, we cannot extend $\nu*(-1)$ to match the value along $\outr{(\Dy(\eta))}$. To prevent this problem we will argue that the right sided output of $\Dy(\eta*0000*(-1))$ \emph{must} necessarily extend $\outl{(\Dy(\eta))}$.

By Claim \ref{claim:kukunondeterministictwice}, the transitions $\hat{A}$ makes when scanning the last three bits of the input $\Dy(\eta*0000)$ cannot all be of type T(iii). Therefore, $\mathtt{out_R}(\Dy(\eta*0000))$ must extend $\mathtt{out_L}(\Dy(\eta))*0$, $\mathtt{out_L}(\Dy(\eta))*10$, $\mathtt{out_L}(\Dy(\eta))*110$ or $\mathtt{out_L}(\Dy(\eta))*1110$. In any case, even if the next bit scanned is $-1$, by Claim \ref{claim:kukurelativeposition}, the $\outr$ values of $\Dy(\eta*0000*(-1)))$ and of $\Dy(\eta))$ can only differ by their least insignificant digit. This means that $\mathtt{out_R}(\Dy(\eta*0000*(-1)))\supset \mathtt{out_L}(\Dy(\eta))$, which means that we will be able to restore $\sgnd{\nu'}=\mathtt{out_R}(\Dy(\eta*0000*(-1)))$ for some $\nu'\supset\nu$. This avoids the problematic situation where $\mathtt{out_R}(\Dy(\eta*0000*(-1)))\supset  \mathtt{out_R}(\Dy(\eta))$ described above.
The following claim proves this formally.
\begin{claim}\label{claim:kukucollapse}
Let $\eta\in\{-1,0,1\}^{<\omega}$ and $n>0$ be such that $|\eta|=n+D+3$. Then $$\sgnd{\outr\left(\Dy(\eta)\upharpoonright n+D\right)*\gamma}=\sgnd{\outr\left(\Dy(\eta*(-1))\upharpoonright n+D+|\gamma|\right)}$$ for some $\gamma\in\{\langle\rangle,1,11,111,-1,(-1)*(-1),(-1)*(-1)*(-1)\}$.
\end{claim}
\begin{proof}
Suppose that $\outr\left(\Dy(\eta)\upharpoonright n+D\right)\neq\outr\left(\Dy(\eta*(-1))\upharpoonright n+D\right)$ (if they are equal then take $\gamma=\langle\rangle$). By Claim \ref{claim:kukurelativeposition}, we can assume that $\overline{\outr\left(\Dy(\eta)\upharpoonright n+D\right)}=\overline{\outr\left(\Dy(\eta*(-1))\upharpoonright n+D\right)}\pm 2^{-n}$. 

First, let's assume that $\overline{\outr\left(\Dy(\eta)\upharpoonright n+D\right)}=\overline{\outr\left(\Dy(\eta*(-1))\upharpoonright n+D\right)}- 2^{-n}$. Therefore $\outr\left(\Dy(\eta*(-1))\upharpoonright n+D\right)$ is of the form $\mu*100\cdots 0$ with length $n$ and $\outr\left(\Dy(\eta)\upharpoonright n+D\right)$ is of the form $\mu*011\cdots 1$ with the same length $n$.

The transition from $\Dy(\eta*(-1))\upharpoonright n+D$ to $\Dy(\eta*(-1))\upharpoonright n+D+1$ cannot be of type T(iii), because otherwise $\overline{\outr\left(\Dy(\eta*(-1))\upharpoonright n+D+1\right)}=\overline{\outr\left(\Dy(\eta*(-1))\upharpoonright n+D\right)}+2^{-n-1}$, and $\overline{\outr\left(\Dy(\eta)\upharpoonright n+D+1\right)}\leq \overline{\outr\left(\Dy(\eta)\upharpoonright n+D\right)}+2^{-n-1}$, contradicting Claim \ref{claim:kukurelativeposition} applied to length $n+D+1$. If the transition is of type T(ii) then $\outr\left(\Dy(\eta)\upharpoonright n+D+1\right)=\mu*011\cdots 1$ of length $n+1$, which means that we can take $\gamma=1$. Therefore we can assume that the transition is of type T(i), which means that $\outr\left(\Dy(\eta*(-1))\upharpoonright n+D+1\right)=\mu*100\cdots 0$ with length $n+1$.

Now we examine the transition from
$\Dy(\eta*(-1))\upharpoonright n+D+1$ to $\Dy(\eta*(-1))\upharpoonright n+D+2$. For the same reason as above, it cannot be of type T(iii). If it is of type T(ii) then again $\outr\left(\Dy(\eta)\upharpoonright n+D+2\right)=\mu*011\cdots 1$ of length $n+2$, which means that we can take $\gamma=11$. Therefore, we can assume that this next transition is of type T(i), which means that $\outr\left(\Dy(\eta*(-1))\upharpoonright n+D+2\right)=\mu*100\cdots 0$ with length $n+2$. Finally the transition 
from $\Dy(\eta*(-1))\upharpoonright n+D+2$ to $\Dy(\eta*(-1))\upharpoonright n+D+3$ cannot be of type T(iii) for the same reason, and we take $\gamma=111$ if it is of type T(ii). So we assume it is of type T(i), which means that $\outr\left(\Dy(\eta*(-1))\upharpoonright n+D+3\right)=\mu*100\cdots 0$ with length $n+3$.

Now by Claim \ref{claim:kukunondeterministictwice}, one of the three transitions of $\Dy(\eta)\upharpoonright n+D+i$ to $\Dy(\eta)\upharpoonright n+D+i+1$ for $i=0,1,2$ cannot be of type T(iii). This means that (by examining cases) $\overline{\outr\left(\Dy(\eta)\upharpoonright n+D+3\right)}<\overline{\mu*0111\cdots 1}$ with length $n+3$. Comparing this value with $\outr\left(\Dy(\eta*(-1))\upharpoonright n+D+3\right)=\mu*100\cdots 0$ would contradict Claim \ref{claim:kukurelativeposition}. 


The case where $\overline{\outr\left(\Dy(\eta)\upharpoonright n+D\right)}=\overline{\outr\left(\Dy(\eta*(-1))\upharpoonright n+D\right)}+2^{-n}$ is symmetric.
%
In this case, the string $\outr\left(\Dy(\eta*(-1))\upharpoonright n+D\right)$ is of the form $\mu*011\cdots 1$ of length $n$, and the string $\outr\left(\Dy(\eta)\upharpoonright n+D\right)$ is of the form $\mu*100\cdots 0$ of the same length $n$. We proceed as above, but this time arguing that if the transition from $\Dy(\eta*(-1))\upharpoonright n+D+i$ to $\Dy(\eta*(-1))\upharpoonright n+D+i+1$ for $i=0,1,2$ is of type T(iii) then we would take $\gamma\in \{-1, (-1)*(-1), (-1)*(-1)*(-1)\}$. 
\end{proof}


\subsubsection*{Describing $A_d$ formally}

$A_d$ will have delay $D+3$. After scanning $n+D+3$ many bits of an input $\eta$, $A_d$ would have written $n$ bits of the output and will need to remember the last $4$ bits of $\Dy\left(\eta\right)$ and the state that $\hat{A}$ is in after reading $\Dy\left(\eta\right)\upharpoonright n+D$ as well as the state that $\hat{A}$ is in after reading $\Dy\left(\eta*(-1)\right)\upharpoonright n+D$. Denote these as $\delta, \left\langle\Shat_L^+,\Shat_R^+\right\rangle$ and $\langle\Shat_L^-,\Shat_R^-\rangle$ respectively. By Claim \ref{claim:kukurelativeposition}, $\overline{ \outr\left( \Dy\left(\eta*(-1)\right)\upharpoonright n+D\right) }=\overline{\outr\left( \Dy\left(\eta\right)\upharpoonright n+D\right)}+\iota 2^{-n}$
where $\iota=-1,0$ or $1$. The states of $A_d$ will be of the form $\left(\delta,\iota,\langle\Shat_L^+,\Shat_R^+\rangle,\langle\Shat_L^-,\Shat_R^-\rangle\right)$ where $\delta$ is a binary string of length $4$, $\iota\in\{-1,0,1\}$ and $\Shat_L^+,\Shat_R^+,\Shat_L^-,\Shat_R^-$ are states of $\hat{A}$.

Since $A_d$ is not interested in inputs representing a negative number, we can, for instance, first prepare the input by replacing it with $0^\omega$ if the first nonzero digit is $-1$. Therefore, we can assume that the input to $A_d$ is an infinite string in $\{-1,0,1\}^\omega$ where the first nonzero digit is $1$. Let $\Shat_{init}$ be the initial state of $\hat{A}$. $A_d$ starts with the initial state $\left(0000,0,\langle \Shat_{init},\Shat_{init}\rangle,\langle \Shat_{init},\Shat_{init}\rangle\right)$ which  will never be visited again once $A_d$ has started reading the input. It then reads the first $D+4$ many bits of the input and transits to the state $\left(\delta,0,\langle\Shat_L^+,\Shat_R^+\rangle,\langle\Shat_L^-,\Shat_R^-\rangle\right)$, where $\delta$ is the last four bits of $\Dy(\eta)$ and $\eta$ is the first $D+4$ bits of the input read and $\langle\Shat_L^+,\Shat_R^+\rangle,\langle\Shat_L^-,\Shat_R^-\rangle$ are the states that $\hat{A}$ would be in after reading $\Dy\left(\eta\right)\upharpoonright D+1$ and $\Dy\left(\eta*(-1)\right)\upharpoonright D+1$ respectively.  $A_d$ then writes the first output bit $1$.

Now we describe the transition of $A_d$ from a noninitial state  $\left(\delta,\iota,\langle\Shat_L^+,\Shat_R^+\rangle,\langle\Shat_L^-,\Shat_R^-\rangle\right)$ to the state $\left(\delta^*,\iota^*,\langle\Shat_0^*,\Shat_1^*\rangle,\langle\Shat_2^*,\Shat_3^*\rangle\right)$. Assume the next input bit read is $b\in\{-1,0,1\}$. Let $\eta*b$ denote the input read so far; assume that $|\eta|=n+D+3$. $A_d$ of course remembers only $\delta$ and the last $4$ bits of $\Dy(\eta)$ and has no access to all of $\Dy(\eta)$. Nevertheless, it is easy to verify the following:
\begin{itemize}
\item If $b=1$ then $\Dy(\eta*b)=\Dy(\eta)*1$ and $\Dy\left(\eta*b*(-1)\right)=\Dy(\eta)*01$.
\item If $b=0$ then $\Dy(\eta*b)=\Dy(\eta)*0$ and $\Dy\left(\eta*b*(-1)\right)=\Dy\left(\eta*(-1)\right)*1$.
\item If $b=-1$ then $\Dy(\eta*b)=\Dy\left(\eta*(-1)\right)$ and $\Dy\left(\eta*b*(-1)\right)=\left(\Dy\left(\eta*(-1)\right)\upharpoonright n+D+3\right)*01$.
\end{itemize}
This gives us a way to determine $\left(\delta^*,\iota^*,\langle\Shat_0^*,\Shat_1^*\rangle,\langle\Shat_2^*,\Shat_3^*\rangle\right)$ from $\left(\delta,\iota,\langle\Shat_L^+,\Shat_R^+\rangle,\langle\Shat_L^-,\Shat_R^-\rangle\right)$. $A_d$ will simulate~$\hat{A}$ by feeding the next bit of a binary input $b_{sim}$ to $\hat{A}$. 
Additionally, $A_d$ also has to pick which state of~$\hat{A}$ to use when scanning the new binary bit $b_{sim}$. Denote this state by $S_{sim}$. The following table summarizes this.

\begin{center}
\begin{tabular}{|l|c|c|c|c|c|}
\hline
 & $\delta^*$ & $S^+_{sim}$ & $b^+_{sim}$ & $S^-_{sim}$ & $b^-_{sim}$ \\\hline

\multirow{2}{*}{$b=1$} & \multirow{2}{*}{$\delta(1)\delta(2)\delta(3)1$} & \multirow{2}{*}{$\langle\Shat_L^+,\Shat_R^+\rangle$} & \multirow{2}{*}{$\delta^*(0)$} & \multirow{2}{*}{$\langle\Shat_L^+,\Shat_R^+\rangle$} &  \multirow{2}{*}{$\delta^*(0)$}\\

&&&&& \\\hline

$b=0$ and & \multirow{2}{*}{$\delta(1)\delta(2)\delta(3)0$} & \multirow{2}{*}{$\langle\Shat_L^+,\Shat_R^+\rangle$} &  \multirow{2}{*}{$\delta^*(0)$} & \multirow{2}{*}{$\langle\Shat_L^-,\Shat_R^-\rangle$} & \multirow{2}{*}{$1$}\\

$\delta\in\{0000,1000\}$ &&&&&\\\hline

$b=0$ and  & \multirow{2}{*}{$\delta(1)\delta(2)\delta(3)0$} & \multirow{2}{*}{$\langle\Shat_L^+,\Shat_R^+\rangle$} &  \multirow{2}{*}{$\delta^*(0)$} & \multirow{2}{*}{$\langle\Shat_L^-,\Shat_R^-\rangle$} & \multirow{2}{*}{$\Dy(\delta^**(-1))(0) $}  \\

$\delta\not\in\{0000,1000\}$ &&&&&  \\\hline

$b=-1$ and & \multirow{2}{*}{$1111$} & \multirow{2}{*}{$\langle\Shat_L^-,\Shat_R^-\rangle$} & \multirow{2}{*}{$\delta^*(0)$} & \multirow{2}{*}{$\langle\Shat_L^-,\Shat_R^-\rangle$}  &  \multirow{2}{*}{$\delta^*(0)$}\\

$\delta\in\{0000,1000\}$ &&&&&\\\hline

$b=-1$ and & Last $4$ bits  & \multirow{2}{*}{$\langle\Shat_L^-,\Shat_R^-\rangle$} & \multirow{2}{*}{$\delta^*(0)$} & \multirow{2}{*}{$\langle\Shat_L^-,\Shat_R^-\rangle$}  & \multirow{2}{*}{$\delta^*(0)$} \\

 $\delta\not\in\{0000,1000\}$ & of $\Dy(\delta*(-1))$ &&&& \\\hline
\end{tabular}
\end{center}

%
%

Now take $\langle\Shat_0^*,\Shat_1^*\rangle$ to be the result of applying $\hat{A}$ to the state $S^+_{sim}$ with bit $b^+_{sim}$. Take $\langle\Shat_2^*,\Shat_3^*\rangle$ to be the result of applying $\hat{A}$ to the state $S^-_{sim}$ with bit $b^-_{sim}$. We want to take $\iota^*$ to encode the relative positions of $\outr\left( \Dy\left(\eta*b*(-1)\right)\upharpoonright n+D+1\right)$ and $\outr\left( \Dy\left(\eta*b\right)\upharpoonright n+D+1\right)$. In the first, fourth and fifth rows, $\iota^*=0$ since $\Dy\left(\eta*b\right)\upharpoonright n+D+1=\Dy\left(\eta*b*(-1)\right)\upharpoonright n+D+1$. In the second and third rows, we can tell $\iota^*$ by looking at $\iota$ and the type of transition $\hat{A}$ makes when applying  state $S^+_{sim}$ with bit $b^+_{sim}$ and when applying state $S^-_{sim}$ with bit $b^-_{sim}$.

Now we describe the output bit written by $A_d$ after reading $\eta*b$.
We first consider the case where~${b=0,1}$. Since $\Dy(\eta*b)=\Dy(\eta)*b$, we look at the transition type $\hat{A}$ makes when the state~$\langle\Shat^+_L,\Shat^+_R\rangle$ reads $\delta^*(0)$. If this is type T(i) we output $0$, for T(ii) we output $-1$ and for T(iii) we output $1$. Now suppose that $b=-1$. If $\iota=0$ then we look at the transition type $\hat{A}$ makes when the state~$\langle\Shat^-_L,\Shat^-_R\rangle$ reads~$\delta^*(0)$ and output $0,-1$ or $1$ respectively. However, if $\iota\neq 0$ then by Claim \ref{claim:kukucollapse},
\[\outr\left(\Dy(\eta)\upharpoonright n+D\right)*\gamma=\outr\left(\Dy(\eta*(-1))\upharpoonright n+D+|\gamma|\right)\] for some $\gamma\in\{1,11,111,-1,(-1)*(-1),(-1)*(-1)*(-1)\}$. Furthermore, we can tell which case holds by looking at $\iota$, $\langle\Shat^+_L,\Shat^+_R\rangle$, $\delta$, $\langle\Shat^-_L,\Shat^-_R\rangle$, $\delta^*$. We output $\gamma$. (Of course, if $|\gamma|>1$ we do not output anything else for the next few steps.)

\subsubsection*{Verifying $A_d$}

We now verify that $A_d$ correctly computes $f$. First, we show that the parameters of $A_d$ are represented correctly:

\begin{claim} Let $\eta$ be an input string of length $n+D+3$ scanned by $A_d$ for some $n>0$, and let $\left(\delta,\langle\Shat_L^+,\Shat_R^+\rangle,\langle\Shat_L^-,\Shat_R^-\rangle\right)$ be the resulting $A_d$-state. Then $\delta$ is the last four bits of $\Dy(\eta)$, $\langle\Shat_L^+,\Shat_R^+\rangle$ and~$\langle\Shat_L^-,\Shat_R^-\rangle$ are the $\hat{A}$-states after scanning $\Dy(\eta)\upharpoonright n+D$ and $\Dy(\eta*(-1))\upharpoonright n+D$ respectively, and \[\overline{ \outr\left( \Dy\left(\eta*(-1)\right)\upharpoonright n+D\right) }=\overline{\outr\left( \Dy\left(\eta\right)\upharpoonright n+D\right)}+\iota 2^{-n}.\]
\end{claim}
\begin{proof}
This follows by an easy induction on $n$ using the definition of $A_d$. The base case $n=1$ is clear. Assume it holds for $n$. The fact that $\delta^*$ is the last four bits of $\Dy(\eta*b)$ is easy; in the case $\delta=0000$ and $b=-1$ we have to use the fact that the first nonzero bit of $\eta$ is assumed to be $1$.
\end{proof}

We now verify that $A_d$ simulates $\hat{A}$ correctly:

\begin{claim}\label{claim:correctlysimulates} Let $\alpha\in\{-1,0,1\}^\omega$. There are infinitely many $n$ such that $\sgnd{\nu}=\overline{\mathtt{out_R}\left(\Dy(\eta)\upharpoonright n+D\right)}$, where $\eta=\alpha\upharpoonright n+D+3$ and $\nu$ be the output produced by $A_d$ after scanning $\eta$. 
%
%
\end{claim}
\begin{proof}We argue by induction on $n$, the length of the output. For the base case $n=1$, we have $\mathtt{out_R}\left(\Dy\left(\alpha\upharpoonright 4+D\right)\upharpoonright 1+D\right)=1$, since this is what $\hat{A}$ does after scanning the first $1+D$ bits. Furthermore, the first output bit of $A_d$ is always $1$, so the base case holds. Assume this holds for $\nu$ of length~$n$ and $\eta=\alpha\upharpoonright n+D+3$. Let $b=\alpha\left(n+D+3\right)$ be the next bit of the output, $\tau=\Dy(\eta)\upharpoonright n+D$ and $\tau'=\Dy(\eta*(-1))\upharpoonright n+D$. By the hypothesis we have $\sgnd{\nu}=\overline{\outr(\tau)}$.

If $b=0$ or $1$, then $\Dy(\eta*b)\upharpoonright n+D+1=\Dy(\eta)*b\upharpoonright n+D+1=\Dy(\eta)\upharpoonright n+D+1=\tau*\delta^*(0)$. Therefore, $\overline{\mathtt{out_R}\left(\tau\right)} - \overline{\mathtt{out_R}\left(\Dy(\eta*b)\upharpoonright n+D+1\right)}= \overline{\mathtt{out_R}\left(\tau\right)} - \overline{\mathtt{out_R}\left(\tau*\delta^*(0)\right)}$. This last quantity depends on the transition type $\hat{A}$ takes after scanning $\tau$ and next reads $\delta^*(0)$. However, the action of $A_d$ exactly mirrors this to ensure that $\sgnd{\nu}-\sgnd{\nu*d}$ is the same as this quantity, where $d$ is the next output bit written by $A_d$.

Now suppose that $b=-1$. Then we have $\Dy(\eta*b)\upharpoonright n+D+1=\tau'*\delta^*(0)$. If $\iota=0$ then ${\overline{\mathtt{out_R}\left(\tau\right)}=\overline{\mathtt{out_R}\left(\tau'\right)}}$, which means that
\[\overline{\mathtt{out_R}\left(\tau\right)} - \overline{\mathtt{out_R}\left(\Dy(\eta*b)\upharpoonright n+D+1\right)}=\overline{\mathtt{out_R}\left(\tau'\right)}-\overline{\mathtt{out_R}\left(\tau'*\delta^*(0)\right)}.\]
This last quantity again depends on the transition type $\hat{A}$ takes after scanning $\tau'$ and next reads $\delta^*(0)$. The action of $A_d$ exactly mirrors this to ensure that $\sgnd{\nu}-\sgnd{\nu*d}$ is the same as this quantity.

Now finally suppose that $b=-1$ and $\iota\neq 0$. Then $A_d$ will write $\gamma$, where
\[
\outr\left(\tau\right)*\gamma=\outr\left(\Dy(\eta*(-1))\upharpoonright n+D+|\gamma|\right).
\]
Therefore, $\sgnd{\nu*\gamma}=\overline{\outr\left(\Dy(\eta*(-1))\upharpoonright n+D+|\gamma|\right)}$. Since $b=-1$, this means that $\eta$ cannot be all zeroes. Furthermore, by our assumption, the first nonzero bit of $\eta$ is $1$. This means that the last bit of $\Dy(\eta*(-1))$ is~$1$, and thus $\Dy(\alpha\upharpoonright n+D+3+j)\supset \Dy(\eta*(-1))\upharpoonright n+D+3$ for any $j>0$. Since $|\gamma|\leq 3$, we have that $\Dy(\alpha\upharpoonright n+D+3+|\gamma|)\upharpoonright n+D+|\gamma|=\Dy(\eta*(-1)) \upharpoonright n+D+|\gamma|$. Thus, the claim holds for output of length~${n+|\gamma|}$.
\end{proof}

\begin{claim}
Given any $\alpha\in\{-1,0,1\}^\omega$ such that $\sgnd{\alpha}\in [0,1]$, if $\beta$ is the output produced by $A_d(\alpha)$, we have $\sgnd{\beta}=f\left(\sgnd{\alpha}\right)$.
\end{claim}
\begin{proof}
Since we assumed that the first nonzero bit of $\alpha$ is $1$, we have $\overline{\Dy\left(\alpha\upharpoonright n+D+3\right)}=\sgnd{\alpha\upharpoonright n+D+3}$ for all $n$. By Claim \ref{claim:correctlysimulates} we fix an $n$ so that $\overline{\mathtt{out_R}\left(\Dy(\alpha\upharpoonright n+D+3)\upharpoonright n+D\right)}=\sgnd{\beta\upharpoonright n}$. Applying Claim
\ref{claim:kukutransducer} to the input string $\Dy\left(\alpha\upharpoonright n+D+3\right)*0^\omega$ for $\hat{A}$ with a run length of $n+4$, we see that
\[
\left|f\left(\overline{\Dy\left(\alpha\upharpoonright n+D+3\right)}\right)  -   \overline{\mathtt{out_R}\left( \Dy\left(\alpha\upharpoonright n+D+3\right) \right)}\right|\leq 2^{-n-3}.
\]
Now $\left|\sgnd{\beta}-f\left(\sgnd{\alpha}\right)\right|$ is bounded above by the sum of the quantities
\begin{itemize}
\item $\left|\sgnd{\beta}- \sgnd{\beta\upharpoonright n}\right|$,
\item $\left|f\left(\sgnd{\alpha}\right) - f\left(\sgnd{\alpha\upharpoonright n+D+3}\right)\right|$,
\item $\left| f\left(\sgnd{\alpha\upharpoonright n+D+3}\right) -   \overline{\mathtt{out_R}\left( \Dy\left(\alpha\upharpoonright n+D+3\right) \right)}\right|$.
\item $\left| \overline{\mathtt{out_R}\left( \Dy\left(\alpha\upharpoonright n+D+3\right) \right)} -  \sgnd{\beta\upharpoonright n}\right|$.
\end{itemize}
The first term is bounded by $2^{-n}$ while the second term is bounded by $O(2^{-n-D-3})$ since $f$ is Lipschitz. As for the third term,
note that since $\overline{\Dy\left(\alpha\upharpoonright n+D+3\right)}=\sgnd{\alpha\upharpoonright n+D+3}$, it is bounded  by $2^{-n-3}$. As for the fourth term, since $\sgnd{\beta\upharpoonright n}=\allowbreak\overline{\mathtt{out_R}\left(\Dy(\alpha\upharpoonright n+D+3)\upharpoonright n+D\right)}$, it is bounded by $2^{-n-1}+2^{-n-2}+2^{-n-3}$. Since $n$ can be made arbitrarily large, we conclude that  $\left|\sgnd{\beta}-f\left(\sgnd{\alpha}\right)\right|=0$.
\end{proof}




This concludes the proof of Proposition \ref{lemma:kukudeterminsticnondeterministic}.
\end{proof}

\section{Proof of Theorem~\ref{thm:mainreg}}\label{sec:mainreg}

This section is devoted to the proof of Theorem~\ref{thm:mainreg}.
We first note that, given a total computable~$f$, checking that $f$ is regular is a $\Sigma^0_2$ problem.
By Theorem~\ref{thm:kuku1}, it is sufficient to check whether there exists a deterministic transducer $M$ working in signed binary such that, for every rational $q$, $f(q)$~is equal to~$M(\overline{\overline{q}})$ (as a real number).
The equality of two computable real numbers is $\Pi^0_1$, and this gives the desired  upper bound $\Sigma^0_2$.

We now prove $\Sigma^0_2$-completeness.
Fix a $\Sigma^0_2$-complete set $S$. Our goal is to produce a computable sequence of functions $(f_z)_{z \in \mathbb{N} }$ in $C[0,1]$ with the following properties:

\begin{itemize}
\item $z \notin S$  implies that $f_z$ is  Lipschitz, pointwise linear-time computable, linear almost everywhere, and such that $f[\mathbb{Q}\cap[0,1]]\subseteq\mathbb{Q}$, but is \emph{not} $b$-regular for any base $b$;

\item $z\in S$ implies $f_z$ is $2$-regular.
\end{itemize}
Here $f_z$ is represented as the index of a Turing functional. Fix a computable relation $R$ such that
 $$z \notin S \iff \exists^{\infty} y (R(z,y)=1). $$

In order to keep $f_z$ linear-time computable, we will need to approximate the membership of $S$ very slowly. By suitably modifying $R$, we can assume that $R(z,y)$ can be computed in $O(y)$ many steps. Since $z$ is fixed, we let $r(s)=R(z,s)$. Define the pairwise disjoint sequence of intervals $(J_n)_{n>2}$ by $J_n=\left(2^{-n}-2^{-2n},2^{-n}+2^{-2n}\right)$.

\begin{claim}\label{claim:kukudelta}
Define the function $\delta\colon 2^{<\omega}\rightarrow 2^{<\omega}$ such that for any $\sigma\in 2^{<\omega}$,
\[\delta(\sigma)=\begin{cases}
1^n&\text{if $\overline{\sigma}\in J_n$ for some $n>2$},\\
0 &\text{if $\displaystyle\overline{\sigma}\not\in\bigcup_{n=3}^\infty J_n$.}
\end{cases}\]
Then $\delta$ can be computed in $O(|\sigma|)$ many steps.
\end{claim}
\begin{proof}
Let $\delta^+$ and $\delta^-$ be defined as $\delta$, but with $J_n$ replaced by $\left[2^{-n},2^{-n}+2^{-2n}\right)$ and $\left(2^{-n}-2^{-2n},2^{-n}\right)$ respectively. We first show that we can compute $\delta^+(\sigma)$ in  $O(|\sigma|)$ many steps.
Given $\sigma$ on the input tape we scan $\sigma$
for the first digit $n$ such that $\sigma(n)=1$. If we reach the end of $\sigma$ without finding $n$, or if we find that $n\leq 2$, then we conclude that $\overline{\sigma}=0$ or $\overline{\sigma}\geq 2^{-2}$ and write $0$ on the output tape. Otherwise, test
if $\sigma(j)=0$ for every $n<j<\min\{|\sigma|,2n+1\}$, and if so, we halt with $1^n$ on the output tape. Otherwise, output $0$. This procedure clearly runs in $O(|\sigma|)$ steps and outputs $\delta^+(\sigma)$.

To compute $\delta^-(\sigma)$, we scan for the first $n$ such that $\sigma(n+1)=1$. We assume that $n\geq 3$, otherwise we conclude that $\overline{\sigma}=0$ or $\overline{\sigma}\geq 2^{-3}$ and we can output $0$ on the output tape. Next we test if $\sigma(j)=1$ for every $n+1\leq j \leq 2n$ and also for some $j>2n$. If one of these fails we output $0$, otherwise output $1^{n}$. This procedure clearly runs in $O(|\sigma|)$ steps and outputs $\delta^-(\sigma)$.
\end{proof}

Now define $f_z$ to be the following function:
\[f_z(x)=\begin{cases}
x+\left(2^{-2n}-2^{-n} \right) &\text{if $\displaystyle x\in \left(2^{-n}-2^{-2n},2^{-n}-2^{-2n-1}\right]$ and $r(n)=1$ for some $n\geq 3$,}\\
2^{-2n-1} &\text{if $\displaystyle x\in \left(2^{-n}-2^{-2n-1},2^{-n}+2^{-2n-1}\right)$ and $r(n)=1$ for some $n\geq 3$,}\\
-x+\left( 2^{-2n}+2^{-n}\right) &\text{if $\displaystyle x\in \left[2^{-n}+2^{-2n-1},2^{-n}+2^{-2n}\right)$ and $r(n)=1$ for some $n\geq 3$,}\\
0 &\text{if $\displaystyle x\in J_n$ and $r(n)=0$ for some $n\geq 3$,}\\
0 &\text{if $\displaystyle x\not\in\bigcup_{n=3}^\infty J_n$.}
\end{cases}
\]
$f_z$ is evidently effectively continuous, and we can compute uniformly in $z$ an index for a Turing functional representing $f_z$. Clearly we have $f_z[\mathbb{Q}\cap[0,1]]\subseteq\mathbb{Q}$ and $|f_z(x)-f_z(y)|\leq |x-y|$ for any $x,y\in [0,1]$. Now if $z\in S$ then $r(n)=0$ for almost all $n$, and so $f_z$ is a piecewise linear function with dyadic coordinates at the breakpoints. By Lemma~\ref{cor:piecewiselinearregular}, $f_z$ is $2$-regular.

Now suppose that $z\not\in S$. Given any $\sigma\in 2^{<\omega}$, it takes (by Claim \ref{claim:kukudelta}) $O(|\sigma|)$ many steps to compute~$\delta(\sigma)$ and in the case where $\overline{\sigma}\in J_n$, also linear-time many steps to evaluate $r(n)$ since $n=O(|\sigma|)$. Finally, in the case where $r(n)=1$, it takes another $O(|\sigma|)$ steps to figure out which of the three subcases in the definition of $f_z(\overline{\sigma})$ applies and to conduct the bitwise operations necessary to produce the output  $f_z(\overline{\sigma})$. So, $f_z$ is pointwise linear-time computable.

It remains to argue that $f_z$ is not $b$-regular for any base $b>1$. Let $b=2^\ell\cdot L$, where $L$ is odd and $\ell\geq 0$. Suppose that the graph of $f_z$ is accepted by some deterministic automaton $M$ in base $b$. The result of applying $M$ to the infinite string $0^\omega$ will be $\eta*\tau^\omega$ for some $\eta,\tau\in b^{<\omega}$, and its value will depend only on the state in which $M$ is in when starting the computation. Therefore, for any $d\in\mathbb{D}_b$, $f(d)=d'+b^{-|d'|}q$ where $d'\in \mathbb{D}_b$ and $q\in C$ where $C$ is a finite set of rationals, and $|d|=|d'|$. Let $C=\left\{\frac{q_i}{q_i'}\mid i=0,\ldots j\right\}$.

First assume that $\ell>0$. Fix $n_0,n_1$ such that $r(\ell n_0-n_1)=1$; we can always pick $n_1<\ell$, so $2^{\ell n_0-2n_1}$ can be made arbitrarily large and we may assume that $2^{\ell n_0-2n_1+1}>\max\{q_i'\mid i=0,\ldots,j\}$. Now since $2^{-\ell n_0+n_1}=L^{n_0}\cdot 2^{n_1}\cdot b^{-n_0}\in\mathbb{D}_b$, this means that $f\left(2^{-\ell n_0+n_1}\right)=C_2\cdot b^{-n_0}+q_{i_1}\cdot(q_{i_1}')^{-1}\cdot b^{-n_0}$ for some $C_2$ and~$i_1$. Since $f\left(2^{-\ell n_0+n_1}\right)=2^{-2\ell n_0+2n_1-1}$, it follows that $2^{\ell n_0-2n_1+1}$ divides $L^{n_0} q_{i_1}'$, which is a contradiction since $L$ is odd.

Now assume that $\ell=0$ and hence $b$ is odd. Consider an $n$ large enough so that ${2^{2n+1}>\max\{q_i'\mid i=0,\ldots,j\}}$ and $r(n)=1$. Unlike in the previous case, we cannot use the center of the interval $J_n$ to obtain a contradiction, since it isn't in $\mathbb{D}_b$. However since $\mathbb{D}_b$ is dense in $[0,1]$, let $C_0$ and~$D_0$ be such that $C_0\cdot b^{-D_0}\in\left(2^{-n}-2^{-2n-1},2^{-n}+2^{-2n-1}\right)$. Then $f\left(C_0\cdot b^{-D_0}\right)=C_1\cdot b^{-D_0}+q_{i_0}\cdot (q_{i_0}')^{-1}\cdot b^{-D_0}\cdot $ for some $C_1$ and~$i_0$. Since $f\left(C_0\cdot b^{-D_0}\right)=2^{-2n-1}$, it follows that $2^{2n+1}$ divides $b^{D_0}\cdot q_{i_0}'$, which is a contradiction since $b$~is odd.

This completes the proof of Theorem~\ref{thm:mainreg}.

\section{Proof of Theorem~\ref{thm:3}}\label{sec:thm:3}

Now we turn our attention to the global approach. 
 In order to make the following discussion clearer, we recall some definitions. Our result (Theorem \ref{thm:3}) is concerned with the complexity of identifying computable presentations of $C[0,1]$ as a Banach space. Recall that a Banach space is a complete normed vector space. A \emph{computable Banach space} is a tuple $(A, d, +, (r\cdot)_{r\in \mathbb{Q}})$ where $A$ is a computable set of special points (that are dense in the underlying metric), $d:A\times A\rightarrow \mathbb{R}$ is a computable metric, $+$ and $(r\cdot)_{r\in \mathbb{Q}}$ are computable functions defined on $A$ representing the vector space operations. Each (partial) computable Banach space can be encoded effectively by an index, and thus we can effectively list all partial computable Banach spaces $B_0,B_1,\cdots$. We say that $(A, d, +, (r\cdot)_{r\in \mathbb{Q}})$ is a computable presentation of a Banach space, say, $(C[0,1], d_{sup}, +, (r\cdot)_{r \in \mathbb{Q}})$, if there is an isometric isomorphism between $(C[0,1], d_{sup}, +, (r\cdot)_{r \in \mathbb{Q}})$ and the completion of $(A, d, +, (r\cdot)_{r\in \mathbb{Q}})$.

We assume that the reader is familiar with the elementary calculus of quantifiers in the arithmetical hierarchy, oracle computation, and with standard notation such as $\bf 0'$, $x \oplus \bf{0'}$ and $\Sigma^0_2(\bf{0'})$. For that, we cite the first few chapters of~\cite{Soa}. It should also be clear how to define an $X$-computable function $f:[0,1] \rightarrow \mathbb{R}$ (here $X$ is an oracle): replace the term ``computable'' with ``$X$-computable'' in the definition (see Subsection~\ref{subs:comp}).
We will also typically identify $f \in C[0,1]$ with an oracle capable of computing $f$ and write, e.g.,
$f \oplus \bf{0'}$, which should be understood as ``if $f$ is $X$-computable then the process under consideration is $X \oplus \bf{0'}$-computable''.
 It is also fairly easy to see that, for two $X$-computable functions $f$ and $g$ in $C[0,1]$, the functions $\max(f,g)$ and $\min(f,g)$ are also $X$-computable (folklore).
 The same can be said about applying the standard operations such as $+$ to $X$-computable functions. Finally, it is well known that, for an $X$-computable $f:[0,1] \rightarrow \mathbb{R}$, its maximum and minimum are $X$-computable reals (although $\{x: f(x) = \max_{y \in [0,1]} f(y)\}$ does not have to contain $X$-computable reals, we will not need this). Not much beyond this knowledge is needed to follow the proof.

\medskip

 We begin by confirming the computational strength required to compute the auxiliary functions we will need for our main result; as previously mentioned, Brown used a different approach for these auxiliary functions in \cite{BrownThesis}.
Recall that, in a metric space $(M,d)$, a Cauchy sequence $(x_i)_{i \in \mathbb{N}}$ is quick for a point $z \in M$ if  $d(x_i,z) < 2^{-i}$ for every $i$.

\begin{lemma}If $(A, d, +, (r\cdot)_{r\in \mathbb{Q}})$ is a computable presentation of $(C[0,1], d_{sup}, +, (r\cdot)_{r \in \mathbb{Q}})$, there is a $\mathbf{0'}$\nobreakdash-computable sequence of $A$-rational points which forms a quick Cauchy sequence for the constant function~$1$.
\end{lemma}
Note that as $f \mapsto -f$ is an auto-isometry of $C[0, 1]$, the decision that such a sequence converges to $1$ instead of $-1$ is arbitrary.

\begin{proof}
Fix an $A$-rational point $q_0$ such that $d(1, q_0) < \frac12$ and $q_0(x) < 1$ for all $x$.  We claim that for every~$\varepsilon$ with $0<\varepsilon < \frac12$, the functions $f$ such that $1 - \varepsilon \le f(x) < 1$ for all $x$ are characterized by:
\begin{enumerate}
\item $d(f, q_0) < \frac12$;
\item $d(f, 0) < 1$; and
\item For every $A$-rational point $p$ with $d(0, p) > \varepsilon$ and $d(f, p) < 1$, $d(0, f+p) \ge 1$.
\end{enumerate}

To verify this, we begin by supposing that $f$ is such that $1-\varepsilon \le f(x) < 1$ for all $x$.  Then clearly $(2)$ holds for $f$.  Since $f(x) > \frac12$ for all $x$, $d(f, q_0) < \frac12$, and $(1)$ holds.  Finally, suppose $p_0 \in C[0,1]$ is such that $d(0, p_0) > \varepsilon$ and $d(f, p_0) < 1$.  Then for all $x$, $p_0(x) > -\varepsilon$, and so for some $z$, $p_0(z) > \varepsilon$.  So $f(z) + p_0(z) > 1$, and $(3)$ holds.

Conversely, suppose $f$ is not such that $1-\varepsilon \le f(x) < 1$ for all $x$.  If there is some $z$ with $f(z) \ge 1$, then $(2)$ fails.  If there is some $z$ with $f(z) \le 0$, then $(1)$ fails.  If neither of these occur, then $0 < f(x) < 1$ for all $x$ and there is some $z$ with $f(z) < 1-\varepsilon$.  Fix $\delta > 0$ such that $f(x) < 1- \delta$ for all $x$ and $f(z) < 1-\varepsilon - \delta$.  Then there is some $A$-rational point $p$ with $1-f(x) - \delta < p(x) < 1-f(x)$ for all $x$.  Since $p(z) > \varepsilon$, we have that $d(0, p) > \varepsilon$.  As $p(x) > 1- f(x) - \delta > 0$, we see that $d(f, p) < 1$.  Therefore, $f(x) + p(x) < 1$ for all $x$, we have $d(0, f+p) < 1$, and $(3)$ fails.

Observe that $(1)$ and $(2)$ are $\Sigma^0_1$ conditions, while $(3)$ is a $\Pi^0_1$ condition.  Thus $\mathbf{0'}$ can uniformly compute the set of $A$-rational points $x$ with $1-\varepsilon \le f(x) < 1$ and so construct a quick Cauchy sequence converging to the constant function $1$.
\end{proof}

\begin{corollary}
If $(A, d, +, (r\cdot)_{r \in \mathbb{Q}})$ is a computable presentation of $(C[0,1], d_{sup},+, (r\cdot)_{r \in \mathbb{Q}})$, then the collection $f \in A$ such that $f(x) > 0$ for all $x$ is a $\Sigma^0_1(\mathbf{0'})$-class.
\end{corollary}

\begin{proof}
$f$ is everywhere positive if and only if for some (and indeed every) $q \in \mathbb{Q}_{>0}$ with $d(0, q\cdot f) < 1$, we can see that $d(0, 1-q\cdot f) < 1$.
\end{proof}

For $f \in C[0, 1]$, define $f^+(x) = \max\{ f(x), 0\}$ for all $x$, and define $f^- = (-f)^+$.  Thus $f = f^+ - f^-$. Now we can determine the difficulty of finding $f^+$ and $f^-$.

\begin{lemma}
If $(A, d, +, (r\cdot)_{r \in \mathbb{Q}})$ is a computable presentation of $(C[0,1], d_{sup},+, (r\cdot)_{r \in \mathbb{Q}})$, then $f \mapsto f^+$ and $f \mapsto f^-$ are $\Delta^0_{1}(\mathbf{0''})$ functions.
\end{lemma}

\begin{proof}
For a function $f$ and $\varepsilon > 0$, consider pairs $(g, h)$ satisfying the following three conditions:
\begin{enumerate}
\item For all $x$, $g(x) > 0$ and $h(x) > 0$;
\item $d(f, g-h) < \varepsilon$;
\item There does not exist an $A$-rational point $p$ with $d(0, p) > \varepsilon$, and $p$, $g-p$ and $h-p$ are all everywhere positive.
\end{enumerate}

First, we claim that there is a pair of $A$-rational points satisfying this.  Let $g$ and $h$ be any $A$-rational points with $f^+(x) < g(x) < f^+(x) + \varepsilon$ and $f^-(x) < h(x) < f^-(x) + \varepsilon$ for all $x$.  Then $g$ and $h$ are everywhere positive, so $(1)$ is satisfied.  Furthermore, $f(x) - \varepsilon < g(x) - h(x) < f(x) + \varepsilon$ for all $x$, and so $(2)$ is satisfied.  Finally, any $p$ such that $p$, $g-p$ and $h-p$ are everywhere positive would have to satisfy $0 < p(x) < \min\{ g(x), h(x)\} < \varepsilon$ for all $x$, and so $d(0, p) < \varepsilon$.

Next, we claim that for any pair $(g,h)$ satisfying these conditions, $d(f^+,g) \le 2\varepsilon$ and ${d(f^-, h) \le 2\varepsilon}$.  Suppose $(g, h)$ satisfies conditions $(1)$ and $(2)$ but there is some $z$ with $g(z) - f^+(z) > 2\varepsilon$.  Since ${d(f, g-h) < \varepsilon}$, it follows that $h(z) > \varepsilon$.  Then
\[
\{ p \st  (\forall x)\, 0 < p(x) < \min(g(x), h(x))\} \cap \{ p \st \varepsilon < p(z)\}
\]
is a nonempty open neighborhood and so contains some $A$-rational point $p$, contradicting condition $(3)$.

Suppose instead that $(g,h)$ satisfies conditions $(1)$ and $(2)$ but there is some $z$ with $f^+(z) - g(z) > 2\varepsilon$.  Since $d(f,g-h) < \varepsilon$, it follows that $h(z) < 0$, contrary to $(1)$.

The remaining cases ($h(z) - f^-(z) > 2\varepsilon$ and $f^-(z) - h(z) > 2\varepsilon$) are symmetric.  It follows that for any pair $(g,h)$ satisfying these conditions, $d(f^+, g), d(f^-, h) \le 2\varepsilon$.

Since $1$ is a $\Delta^0_2$ point, conditions $(1)$ and (3) are $\Sigma^0_1(\mathbf{0'})$ conditions on rational points (with no mention of $f$), while condition $(2)$ is a $\Sigma^0_1(f)$ condition.  Thus $f\oplus \emptyset''$ can enumerate these pairs and create a sequence quickly converging to $f^+$, and the map $f \mapsto f^+$ is computable with oracle $\mathbf{0''}$.  The same argument holds for $f\mapsto f^-$.
\end{proof}

\begin{corollary}
If $(A, d, +, (r\cdot)_{r \in \mathbb{Q}})$ is a computable presentation of $(C[0,1], d_{sup},+, (r\cdot)_{r \in \mathbb{Q}})$, then the modulus function, denoted $\text{mod}(f)$ and given by $\mathrm{mod}(f)(x) = |f(x)|$, and the max and min functions, given by $\max(f,g)(x) = \max\{f(x), g(x)\}$ and $\min(f,g)(x) = \min\{f(x), g(x)\}$, are $\Delta^0_1(\mathbf{0''})$.
\end{corollary}

\begin{proof}
By defining
\[\begin{array}{rcl}
\text{mod}(f) & = & f^+ + f^- \\
\max(f,g) & = & \frac12[f + g + \text{mod}(f-g)] \\
\min(f,g) & = & \frac12[f+g-\text{mod}(f-g)],
\end{array}\]
the claim is immediate.
\end{proof}

For rationals $0 \le p < q \le 1$ and $r > 0$, we define $t_{p,q,r}$ to be the tooth function which is 0 for $x \le p$ or~$x \ge q$ and which increases linearly to height $r$ midway between $p$ and $q$.
The key property of these functions is as follows.

\begin{claim}
The rational linear combinations of the tooth functions $t_{p,q,r}$ are dense in $C[0,1]$.
\end{claim}

\begin{proof}
This is a version of the Stone-Weierstrass Theorem; see, e.g., Theorem 7.29 of \cite{Hew}.
\end{proof}

If using ${\bf 0}^{(n)}$ (for a sufficiently large $n$) we could uniformly map every tooth function in the `natural' computable presentation of $C[0,1]$ into its isomorphic image in $A$, we would be done.
This is because  every function $f \in C[0,1]$ has an $f$-computable fast converging Cauchy name consisting of finite linear combinations of the tooth functions.
To calculate the image of $f$ in $A$, simply map these sums into the respective sums. 
If we succeed  in finding such an $n$, then saying that $A \cong C[0,1]$ would be equivalent to saying that there is a ${\bf 0}^{(n)}$-computable linear isometric isomorphism between $A$ and $C[0,1]$, which is easily seen to be an arithmetical statement. 

To this end, our task is to arithmetically find  suitable images for $t_{p,q,r}$ in $A$.
The lemma below is the importnat first step.

\begin{lemma}
Let $\text{id}$ denote the identity function in $A$. If $(A, d, +, (r\cdot)_{r \in \mathbb{Q}})$ is a computable presentation of $(C[0,1], d_{sup},+, (r\cdot)_{r \in \mathbb{Q}})$, the tooth functions $t_{p,q,r}$ are uniformly $\Delta^0_1(\mathbf{0''}\oplus \text{id})$.
\end{lemma}

\begin{proof}We simply express $t_{p,q,r}$ in terms of $\min$, $\max$, and $\text{id}$:
\[
t_{p,q,r} = \min\left( \max\left( \frac{2r}{q-p}[\text{id} - p\cdot 1], 0\right), \max\left( \frac{-2r}{q-p}[\text{id} - q\cdot 1], 0\right)\right).\qedhere
\]
\end{proof}

\begin{claim}
If $g$ is a homeomorphism of $[0,1]$ (i.e.\ a strictly increasing or strictly decreasing function in $C[0,1]$ with max 1 and min 0), then $f \mapsto f\circ g$ is an isometry of $C[0,1]$.
\end{claim}

The proof of this claim can be found in, e.g., Dunford-Schwartz \cite{DanSh1} vol. 1, Th. IV.6.26.
In fact, this gives a complete description of the automorphism group of $C[0,1]$ in terms of self-homeomorphisms of $[0,1]$. This fact can be generalised to arbitrary compact domains,
known as the Banach-Stone duality in the literature.

\medskip

For $g \in C[0,1]$, define
\[
t_{p,q,r}(g) = \min\left( \max\left( \frac{2r}{q-p}[g - p\cdot 1], 0\right), \max\left( \frac{-2r}{q-p}[g - q\cdot 1], 0\right)\right).
\]

Note that the maps $g \mapsto t_{p,q,r}(g)$ are uniformly computable with oracle $\bf{0^{(2)}}$.

\begin{corollary}
If $g$ is a homeomorphism of $[0,1]$, then the rational linear combinations of the $t_{p,q,r}(g)$ are dense in $C[0,1]$.  Further, $t_{p,q,r}(\text{id}) \mapsto t_{p,q,r}(g)$ induces an isometry.
\end{corollary}

By the corollary above, to finish the proof of the theorem it is sufficient to 
arithmetically locate, in~$A$, at least one strictly monotonic $g$ with minimum $0$ and maximum $1$.

\medskip

Now, if $g \in C[0,1]$ is any nonconstant function, define $\hat{g}$ via $\hat{g}(x) = \frac1{b - a} (g(x) - a)$, where
\[\begin{aligned}
a &= \min\{ g(y) \st  y \in [0,1]\} \rlap{\text{\ and}}\\
b &= \max\{ g(y) \st  y \in [0,1]\}.\\
\end{aligned}\]
Note that if $g$ is strictly increasing or strictly decreasing, then $\hat{g}$ is a homeomorphism of $[0,1]$, and if $g$~is a homeomorphism of $[0,1]$, then $\hat{g} = g$.

\begin{lemma}
If $(A, d, +, (r\cdot)_{r \in \mathbb{Q}})$ is a computable presentation of $(C[0,1], d_{sup},+, (r\cdot)_{r \in \mathbb{Q}})$, then $g \mapsto \hat{g}$ is a  $\Delta^0_1(\mathbf{0'})$ function with a $\Sigma^0_1(\mathbf{0'})$ domain. (Note the domain is a dense open set.)
\end{lemma}

\begin{proof}
A function $g$ is in the domain precisely if $b \neq a$, where
\[\begin{aligned}
a &= \min\{ g(y) \st  y \in [0,1]\}\rlap{\text{\ and}}\\
b &= \max\{ g(y) \st  y \in [0,1]\}
\end{aligned}\]
as above. Thus it suffices to show that $a$ and $b$ are $g\oplus\emptyset'$-computable, uniformly, since $b = d(0,g)$ and  $a = b - d(b\cdot 1 - g, 0)$.
\end{proof}

The only remaining challenge is to locate a \emph{strictly} monotinic function in $A$.

\begin{lemma}
Let $f \in C[0, 1]$ and $\varepsilon > 0$. If $\{ h \in C[0, 1] \st  \forall x\, f(x) - \varepsilon \le h(x) < f(x) + \varepsilon\}$ contains a nonstrict monotonic function $g$, then $B(f, \varepsilon)$ contains a strictly monotonic function.
\end{lemma}

\begin{proof}
Without loss of generality, assume $g$ is nondecreasing.  By compactness, fix a $\delta > 0$ such that $g(x) < f(x) + \varepsilon - \delta$ for all $x$, and let $h(x) = g(x) + \delta(1+x)/2$.  Then $g(x) < h(x) < f(x) + \varepsilon$ for all $x$, and $h$~is strictly monotonic.
\end{proof}

\begin{lemma}
If $(A, d, +, (r\cdot)_{r \in \mathbb{Q}})$ is a computable presentation of $(C[0,1], d_{sup},+, (r\cdot)_{r \in \mathbb{Q}})$, then the collection of open balls $B(f,\varepsilon)$ which contain a (strictly) monotonic function is $\Pi^0_4$.
\end{lemma}

\begin{proof}
We claim that such open balls $B(f, \varepsilon)$ can be characterized by the following:
\begin{quote}
For every finite collection of $A$-rational open balls $\{ B(g_j, \delta_j) \st  j < N\}$, there is a nonconstant $A$-rational point $h \in B(f,\varepsilon)$ such that for each $j < N$ there is a linear combination of the $t_{p, q, r}(\hat{h})$ which lies in $B(g_j, \delta_j)$.
\end{quote}
We observe that this is $\Pi^0_2(\mathbf{0''})$.

Suppose $B(f,\varepsilon)$ contains a (strictly) monotonic function $k$ and fix $\{ B(g_j, \delta_j) \st  j < N\}$.  Then $\hat{k}$ is a homeomorphism of $[0,1]$, and so the rational linear combinations of $t_{p,q,r}(\hat{k})$ are dense, and in particular there is some rational linear combination $\sum_{i < n_j} t_{p_{i,j}, q_{i,j}, r_{i,j}}(\hat{k})$ in $B(g_j, \delta_j)$ for each $j < N$.  Since each $k \mapsto \sum_{i < n_j} t_{p_{i,j}, q_{i,j}, r_{i,j}}(\hat{k})$ is continuous, each $B(g_j, \delta_j)$ is open, and $N$ is finite, there is some $A$-rational point $h \in B(f, \varepsilon)$ such that $\sum_{i < n_j} t_{p_{i,j}, q_{i,j}, r_{i,j}}(\hat{h}) \in B(g_j, \delta_j)$ for each $j < N$.

Conversely, suppose $B(f, \varepsilon)$ contains no (strictly) monotonic function.  Our argument is based on the unit interval being $T_4$.
\begin{claim}For such a $B(f,\varepsilon)$, there is an $N$ and a set of pairs of closed intervals $\{ (I_j^0, I_j^1) \st  j < N\}$ such that
\begin{itemize}
\item for $j < N$, $I_j^0 \cap I_j^1 = \emptyset$, and
\item for any $h \in B(f,\varepsilon)$, there is $j < N$ and $x_0 \in I_j^0, x_1 \in I_j^1$ with $h(x_0) = h(x_1)$.
\end{itemize}
We allow singletons as closed intervals.
\end{claim}
For the moment, assuming the claim is true, we note that for any nonconstant $h \in B(f, \varepsilon)$, if we fix the appropriate pair $(I_j^0, I_j^1)$, every linear combination $\ell$ of $t_{p, q, r}(\hat{h})$ must also have this property---there are $x_0 \in I_j^0$ and $x_1 \in I_j^1$ with $\ell(x_0) = \ell(x_1)$.  Now, for $j < N$, fix any $A$-rational $g_j$ such that $g(x) \le 0$ for all $x \in I_j^0$ and $g(x) \ge 2$ for all $x \in I_j^1$.  Our balls are $\{ B(g_j, 1) \st  j < N\}$.  For any $h \in B(f, \varepsilon)$, for the appropriate $j$, no rational linear combination of $t_{p, q,r}(\hat{h})$ can produce an element of $B(g_j, 1)$.

It remains only to prove the claim.  Let $B' = \{ h \in C[0, 1]\st  \forall x\, f(x) - \varepsilon \le h(x) < f(x) + \varepsilon\}$.  Then $B'$~contains no nonstrict monotonic function.  As $h(x) = \max\{ f(y) - \varepsilon \st  y \le x\}$ is nondecreasing, $h \not \in B'$, and so for any $x$ there is some $y < x$ such that $f(y) - \varepsilon \ge f(x)+\varepsilon$.

Similarly, $h(u) = \max\{ f(v) - \varepsilon \st  v \ge u\}$ is nonincreasing, and so not in $B'$, and thus  for any $u$ there is $v > u$ such that $f(v) - \varepsilon \ge f(u) + \varepsilon$.

Note that for any $h \in B(f, \varepsilon)$, $h(x) < h(y)$ and $f(u) < f(v)$.  We consider the various cases for the arrangements of $x, y, u$ and $v$.
\begin{enumerate}
\item $x < u$.  We divide into subcases:
\begin{enumerate}
\item $f(x) = f(u)$. Let $I_0^0 = [y, x]$ and $I_0^1 = [u, v]$.  Then for any $h \in B(f, \varepsilon)$, ${h(y) > f(y) - \varepsilon \ge f(x)+\varepsilon}$ and ${h(x) < f(x) + \varepsilon}$, and, similarly, ${h(v) > f(x) + \varepsilon > h(u)}$.  Therefore, $h$~must take the value $f(x) + \varepsilon$ somewhere on both $I_0^0$ and $I_0^1$ by the Intermediate Value Theorem.
\item $f(x) < f(u)$.  Fix $w$ with $y < w < x$.  Let $I_0^0 = [y, w]$, $I_0^1 = [x, v]$, $I_1^0 = [w, x]$ and $I_1^1 = [u,v]$.  For any $h \in B(f, \varepsilon)$, if $h(w) \le f(u) + \varepsilon$ and $h(y) \ge f(u) + \varepsilon$, then by the Intermediate Value Theorem $h$ takes the value $f(u)+ \varepsilon$ on $I_0^0$ and $I_0^1$.  If $h(y) < f(u) + \varepsilon \le f(v) - \varepsilon$, then since $h(x) < h(y)$, by the Intermediate Value Theorem $h$ takes the value $h(y)$ on $I_0^1$, and $y \in I_0^0$.  If $h(w) > f(u)+\varepsilon$, then by the Intermediate Value Theorem $h$ takes the value $f(u)+\varepsilon$ on $I_1^0$ and $I_1^1$.
\item $f(u) < f(x)$. As in case (1b), mutatis mutandis.
\end{enumerate}
\item $v < y$.  As in case (1), mutatis mutandis.
\item $x = u$.  Let $I_0^0 = \{y\}, I_0^1 = [u, v]$, $I_1^0 = [y, x]$, $I_1^1 = \{v\}$.  For any $h \in B(f, \varepsilon)$, if $h(y) \le h(v)$, then since $h(y) > f(u) + 2\varepsilon = f(x) + 2\varepsilon$, by the Intermediate Value Theorem $h$ must take the value $h(y)$ on $I_0^1$, and $y \in I_0^0$.  Similarly, if $h(v) \le h(y)$, then $h$ must take the value $h(v)$ on $I_1^0$, and~$v \in I_1^1$.
\item $v = y$.  As in case (3), mutatis mutandis.
\item $y < u < x < v$.  We divide into subcases:
\begin{enumerate}
\item $f(u) < f(x)$.  By continuity, choose an $x'$ with $y < x' < u$ and $f(x') < f(x)$.  Replace $x$ with~$x'$, reducing to case (1).
\item $f(x) < f(u)$.  As in case (5a), mutatis mutandis.
\item $f(x) = f(u)$.  Let $I_0^0 = [y, u]$ and $I_0^1 = [x, v]$. By the Intermediate Value Theorem, $h$ takes the value $f(u) + \varepsilon = f(x) + \varepsilon$ on both intervals.
\end{enumerate}
\item $u < y < v < x$.  As in case (5), mutatis mutandis.
\item $u < y < x < v$.  Let $I_0^0 = [y, x], I_0^1 = \{u\}, I_1^0 = [y, x], I_1^1 = \{v\}, I_2^0 = [u, y], I_2^1 = [x,v]$.  Fix~$h \in B(f, \varepsilon)$.  If ${h(x) \le h(u) \le h(y)}$, then by the Intermediate Value Theorem $j = 0$ suffices.  If ${h(x) \le h(v) \le h(y)}$, then by the Intermediate Value Theorem $j = 1$ suffices.  If $h(v) > h(y)$, which includes the case $h(u) > h(y)$, then by the Intermediate Value Theorem $h$ takes the value~$h(y)$ on $[x, v]$, and so $j = 2$ suffices.  If $h(u) < h(x)$, which includes the case $h(v) < h(x)$, then by the Intermediate Value Theorem $h$ takes the value $h(x)$ on $[u, y]$, so $j = 2$ suffices.
\item $u = y < x < v$.  Let $I_0^0 = \{u\}$ and $I_0^1 = [x, v]$.  By the Intermediate Value Theorem, $h$ takes the value $h(u)$ on $I_0^1$.
\item $u < y < x = v$.  As in case (8), mutatis mutandis.
\item $u = y < x = v$.  This is impossible since $f(u) < f(v)$ and $f(y) > f(x)$.
\item $y \le u < v \le x$.  As in cases (8), (9), or (10), mutatis mutandis.
\end{enumerate}
This completes the proof.
\end{proof}

Thus if $(A, d, +, (r\cdot)_{r \in \mathbb{Q}})$ is a computable presentation of $(C[0,1], d_{sup},+, (r\cdot)_{r \in \mathbb{Q}})$, there is a $\Delta^0_5$ monotonic function, but it may not be strict (because strictness is not generally preserved by limits). We need to be a bit more careful if we want our function to be strictly monotonic. This is done as follows.

\medskip

Using $\bf{0'}$, fix the dense open set $\mathcal{D}$ which is the domain of $g \mapsto \hat{g}$, and note that $g \mapsto t_{p,q,r}(\hat{g})$ is $\bf{0^{(2)}}$-computable upon this domain.  Let $(g_i)_{i \in \mathbb{N}}$ be the dense sequence of rational points in the given presentation. 
If $g$ were strictly monotonic, then any $g_i$ could be approximated by a finite linear combination of the teeth-functions $t_{p,q,r}(\hat{g})$ to an arbitrary precision.
If $L(g)$ is one such finite linear combination, then the map $h \mapsto L(h)$ is   $\bf{0^{(2)}}$-computable (upon  $\mathcal{D}$).
If $g$ were strictly monotonic, then for every $g_i$ and every $\epsilon$, we could find a finite linear combination $L$, a rational $g_j$ and $\delta>0$ such that 
$$ L(B(g_j, \delta)) \subseteq B(g_i, \epsilon), \mbox{ where } g \in B(g_j, \delta).$$
This is because $g \in \mathcal{D}$ and $L$ is $\bf{0^{(2)}}$-computable, and thus (in particular) continuous, and because $g_i$ can be approximated by finite linear combinations of $t_{p,q,r}(\hat{g})$ to arbitrary precision.

The (open name for) the strictly monotonic $g$ is now built in stages, as follows.  First, fix the sequence of open balls $(U_s)_{s \in \omega}$ defined as $U_s = B(g_i, 2^{-s})$ when $s = \langle i, k\rangle$, where $\langle \cdot, \cdot\rangle$ is the standard pairing function.

\smallskip

 At stage $0$ we fix an arbitrary basic open ball $B_0$ in $\mathcal{D}$ that contains a strictly monotonic function.  
 

 
 \smallskip

At stage $s$ we search for the first found finite formal linear combination $L$ of $t_{p,q,r}(\hat{\cdot})$ and a basic open  $B_s$ such that:

\begin{enumerate}
\item $B_s$ has radius $\leq 2^{-s}$;
\item $B_s$ is formally included in  $B_{s-1}$ (meaning that the inclusion is derived from the triangle inequality applied to their radii and centers);
\item $L(B_s) \subseteq U_s$;
\item $B_s$ contains a (strictly) monotonic function.
\end{enumerate}

\smallskip

\noindent (We remark that for $s_0 = \langle i, k_0\rangle$ and $s_1 = \langle i, k_1\rangle$, the witnessing $L$ may differ between stages $s_0$ and $s_1$.  We will see that this is not an issue.) 

\medskip

We argue that at no stage we are stuck, i.e., that the sequence $(B_s)_{s \in \mathbb{N}}$ is well defined. Indeed, we maintain the construction inside $\mathcal{D}$, and therefore 
every finite form $L$ that we consider is well defined in each $B_s$. Also, by induction,  $B_{s-1}$ contains a (strictly) monotonic function, and therefore for a small enough neighbourhood of this function and some $L$, the image of this neighbourhood will be inside $U_s$ as previously argued. Of course, $B_{s}$ does not have to be this particular neighbourhood, but as long as $B_{s}$
contains a (strictly) monotonic function and satisfies $L(B_s) \subseteq U_s$ we can carry on the construction to the next stage.
It follows from the lemmas preceding the construction that the construction of $(B_s)_{s \in \mathbb{N}}$ can be carried out effectively in ${\bf 0^{(4)}}$.

\medskip

Now we argue that  $(B_s)_{s \in \mathbb{N}}$ converges to a strictly monotonic function $g$. It should be clear that, since each $B_s$ contains a strictly monotonic function, their limit $g$ has to be monotonic. If $g$ is not strictly monotonic, suppose $g(x) = g(y)$, and thus $g(z) = g(x)$, for all $z \in [x,y]$.
Then infinite linear combinations of the tooth-functions $t_{p,q,r}$ defined with the help of $\hat{g}$ do not generate all of the space, but they generate a closed subspace.
Some rational special point $g_i$ will be left out of the span. (This could be a special point that looks roughly like a tooth-function with support strictly inside $[x,y]$.)
 For a small enough $\epsilon$, there is no $\delta$ and no finite form $L$ such that $L(B(g, \delta) )\subseteq B(g_i, \epsilon)$.
 But this contradicts condition $(3)$ in the construction for any $s$ of the form $s = \langle i,k\rangle$ with $2^{-s} < \epsilon$.

\medskip








Thus if $(A, d, +, (r\cdot)_{r \in \mathbb{Q}})$ is a computable presentation of $(C[0,1], d_{sup},+, (r\cdot)_{r \in \mathbb{Q}})$, there is a $\Delta^0_5$ strictly monotonic function and thus a $\Delta^0_5$ isomorphism between $(A, d, +, (r\cdot)_{r \in \mathbb{Q}})$ and $(C[0,1], d_{sup},+, (r\cdot)_{r \in \mathbb{Q}})$.  It follows that the collection of indices for computable presentations of $(C[0,1], d_{sup},+,(r\cdot)_{r\in\mathbb{Q}})$ is arithmetical.

\begin{corollary}\label{cor:multcor}
In any presentation of $(C[0,1], d_{sup},+,(r\cdot)_{r\in\mathbb{Q}})$, multiplication of functions is arithmetically computable relative to the presentation.
\end{corollary}

\section{Conclusion: Interpreting the index set results}\label{concl}
Now we can address questions (Q.1), (Q.2) and (Q.2) more directly: is there a useful classification of continuous regular/transducer functions? Is there a characterisation of 
$C[0,1]$ among all separable Banach spaces?

The first main result, Theorem~\ref{thm:kuku}, completely reduces (Q.2) to (Q.1), and 
even though its proof seemingly has nothing to do with index sets, the argument was discovered in our attempt to show that the index sets for these classes are different. Indeed, index set techniques often rely on fine-grained analysis of object being studied, and this analysis is not always recursion-theoretic in nature. Of course, Theorem~\ref{thm:kuku} is a  characterization-type result; however, it does not really answer (Q.1) or (Q2), it just shows these questions are equivalent.

The second main result Theorem~\ref{thm:mainreg} gives a rather strong and sharp estimate for the complexity of the index set of transducer (regular continuous) functions among the functions that exhibit all known properties expected from a transducer function.
What does it tell us about the classification problem for such functions? As we discussed in the introduction, this indicated that the known properties of transducer functions do not really help in their classification, and that very likely there is no nice classification of such functions that would be simpler than their definition.

 As  argued in \cite{compclassif}, a useful classification should provide a better algorithmic description of objects: the classifications of finite abelian groups, finite fields, and 2-dimensional compact surfaces all have these nice algorithmic properties.
In other words, even if a reasonably nice analytic description of transducer functions can be found,  it will have to  be a \emph{reformulation} rather than a useful structural \emph{classification} of such functions. For example, Theorem~\ref{thm:kuku} is one such reformulation.
Of course, there are many nice useful reformulations, often called dualities, in mathematics: regular and automatic languages, Boolean algebras and Stone spaces, compact and discrete abelian groups, etc. While these results are definitely helpful, none of these dualities help to reduce the complexity of natural decision procedures for the dual objects, and such results should not be confused with classifications.  Informally, Theorem~\ref{thm:mainreg} says that one should not expect a breakthrough classification, but perhaps an interesting analytic reformulation can still be found. Thus, Theorem~\ref{thm:mainreg} can be viewed as a negative result (an `anti-structure' result in terms of \cite{GK:02}).

In contrast, Theorem~\ref{thm:3} is more of a positive `structure' result. Indeed, it gives an unexpectedly low upper bound for the characterization of $C[0,1]$ among all separable Banach spaces. What does the characterisation tell us?
It is essentially formulated in terms of a certain notion of local independence that allows one to \emph{arithmetically} construct the standard basis consisting of the tooth functions with rational break-points.
This is done  using only the standard Banach space operations using  an arbitrary countable dense set. We strongly suspect that this implies that the tooth functions are uniformly definable in some natural adaptation of the infinitary logic $\mathcal{L}_{\omega_1\omega}$ and at some finite (computable) level, perhaps some version of the continuous logic, but  we leave this as an open problem.  We also conjecture that the uniformity of the proof of Corollary~\ref{cor:multcor} should entail that multiplication is definable from the other operations in some formal language.  These formal syntactic definitions (if they can be found) will not be first-order in the usual sense. Nonetheless,  the result implies the existence of a first order formula  $\psi$ in the language of arithmetic such that
$$\bar{B_i} \cong C[0,1] \iff \mathbb{N} \models \psi(i),$$
where $\cong$ stands for linear isometric isomorphism.   This is the best one can expect for general separable spaces, as essentially no problems for general spaces are decidable.
The formula says that there is a $\Delta^0_5$ isomorphism, but we conjecture that it can be simplified to say that a certain process of building a basis never terminates in the space.
Establishing a sharp upper bound on the complexity of the index set of $C[0,1]$ (thus, of the respective $\psi$) is left as an open problem.

\bibliographystyle{alpha}
\bibliography{ourbib}

\newcommand{\etalchar}[1]{$^{#1}$}
\def\cprime{$'$} \def\cprime{$'$} \def\cprime{$'$} \def\cprime{$'$}
\begin{thebibliography}{BGHK{\etalchar{+}}20}

\bibitem[BDKM19]{BDKM:19}
Nikolay Bazhenov, Rod Downey, Iskander Kalimullin, and Alexander Melnikov.
\newblock Foundations of online structure theory.
\newblock {\em Bull. Symb. Log.}, 25(2):141--181, 2019.

\bibitem[BG00]{blumensath2000}
Achim Blumensath and Erich Gr{\"a}del.
\newblock Automatic structures.
\newblock In {\em Fifteenth Annual IEEE Symposium on Logic in Computer Science
  (LICS 2000)}, pages 51--62, 2000.

\bibitem[BGHK{\etalchar{+}}20]{reeeg}
Alexi Block~Gorman, Philipp Hieronymi, Elliot Kaplan, Ruoyu Meng, Erik
  Walsberg, Zihe Wang, Ziqin Xiong, and Hongru Yang.
\newblock Continuous regular functions.
\newblock {\em Log. Methods Comput. Sci.}, 16(1):Paper No. 17, 24, 2020.

\bibitem[BHMN14]{BHMN}
Laurent Bienvenu, Rupert H\"{o}lzl, Joseph~S. Miller, and Andr\'{e} Nies.
\newblock Denjoy, {D}emuth and density.
\newblock {\em J. Math. Log.}, 14(1):1450004, 35, 2014.

\bibitem[BHS14]{norm1}
Ver\'{o}nica Becher, Pablo~Ariel Heiber, and Theodore~A. Slaman.
\newblock Normal numbers and the {B}orel hierarchy.
\newblock {\em Fund. Math.}, 226(1):63--78, 2014.

\bibitem[BHW08]{Brattka.Hertling.ea:08}
V.~Brattka, P.~Hertling, and K.~Weihrauch.
\newblock A tutorial on computable analysis.
\newblock In {\em New computational paradigms}, pages 425--491. Springer, New
  York, 2008.

\bibitem[BKN08]{BKN:08}
Michael Brough, Bakhadyr Khoussainov, and Peter Nelson.
\newblock Sequential automatic algebras.
\newblock In {\em Logic and theory of algorithms}, volume 5028 of {\em Lecture
  Notes in Comput. Sci.}, pages 84--93. Springer, Berlin, 2008.

\bibitem[BMMar]{leb}
T.~Brown, T.~McNicholl, and A.~Melnikov.
\newblock On the complexity of classifying {L}ebesgue spaces.
\newblock {\em The Journal of Symbolic Logic, in print}, To appear.

\bibitem[BMN16]{BMN}
Vasco Brattka, Joseph~S. Miller, and Andr\'{e} Nies.
\newblock Randomness and differentiability.
\newblock {\em Trans. Amer. Math. Soc.}, 368(1):581--605, 2016.

\bibitem[Bro19]{BrownThesis}
Tyler~Anthony Brown.
\newblock {\em Computable {S}tructure {T}heory on {B}anach {S}paces}.
\newblock PhD thesis, Iowa State University, 2019.

\bibitem[BRS06]{MR2194959}
L.~Bartholdi, I.~I. Reznykov, and V.~I. Sushchansky.
\newblock The smallest {M}ealy automaton of intermediate growth.
\newblock {\em J. Algebra}, 295(2):387--414, 2006.

\bibitem[BS14]{norm2}
Ver\'{o}nica Becher and Theodore~A. Slaman.
\newblock On the normality of numbers to different bases.
\newblock {\em J. Lond. Math. Soc.}, 90(2):472--494, 2014.

\bibitem[CR99]{CRindex}
Douglas Cenzer and Jeffrey~B. Remmel.
\newblock Index sets in computable analysis.
\newblock {\em Theoret. Comput. Sci.}, 219({1--2}):111--150, 1999.
\newblock Computability and complexity in analysis (Castle Dagstuhl, 1997).

\bibitem[CSV13]{vardi}
Swarat Chaudhuri, Sriram Sankaranarayanan, and Moshe~Y. Vardi.
\newblock Regular real analysis.
\newblock In {\em 28th Annual {ACM/IEEE} Symposium on Logic in Computer
  Science, {LICS} 2013, New Orleans, LA, USA, June 25--28, 2013}, pages
  509--518. {IEEE} Computer Society, 2013.

\bibitem[DM08]{DoMo}
R.~Downey and A.~Montalb{\'a}n.
\newblock The isomorphism problem for torsion-free abelian groups is analytic
  complete.
\newblock {\em J. Algebra}, 320(6):2291--2300, 2008.

\bibitem[DM14]{DoMel1}
Rodney Downey and Alexander~G. Melnikov.
\newblock Computable completely decomposable groups.
\newblock {\em Trans. Amer. Math. Soc.}, 366(8):4243--4266, 2014.

\bibitem[DM20]{compclassif}
Rodney~G. Downey and Alexander~G. Melnikov.
\newblock Computable analysis and classification problems.
\newblock In Marcella Anselmo, Gianluca~Della Vedova, Florin Manea, and Arno
  Pauly, editors, {\em Beyond the Horizon of Computability, 16th Conference on
  Computability in Europe, CiE 2020, Fisciano, Italy, June 29 -- July 3, 2020,
  Proceedings}, volume 12098 of {\em Lecture Notes in Computer Science}, pages
  100--111. Springer, 2020.

\bibitem[DS88]{DanSh1}
Nelson Dunford and Jacob~T. Schwartz.
\newblock {\em Linear operators. {P}art {I}}.
\newblock Wiley Classics Library. John Wiley {{\&}} Sons Inc., New York, 1988.
\newblock General theory, With the assistance of William G. Bade and Robert G.
  Bartle, Reprint of the 1958 original, A Wiley-Interscience Publication.

\bibitem[Gao09]{GaoBook}
Su~Gao.
\newblock {\em Invariant descriptive set theory}, volume 293 of {\em Pure and
  Applied Mathematics (Boca Raton)}.
\newblock CRC Press, Boca Raton, FL, 2009.

\bibitem[GN02]{GK:02}
S.~S. Goncharov and Dzh. Na\u{\i}t.
\newblock Computable structure and antistructure theorems.
\newblock {\em Algebra Logika}, 41(6):639--681, 757, 2002.

\bibitem[Gr{\"a}20]{gradel2020}
Erich Gr{\"a}del.
\newblock Automatic structures: Twenty years later.
\newblock In {\em 35th Annual ACM/IEEE Symposium on Logic in Computer Science
  (LICS 2020)}, pages 21--34, 2020.

\bibitem[Hod76]{hod1}
Bernard~R. Hodgson.
\newblock {\em Th{\'e}ories d{\'e}cidables par automate fini}.
\newblock PhD thesis, D{\'e}partement de math{\'e}matiques et de statistique,
  Universit{\'e} de Montr{\'e}al, 1976.

\bibitem[Hod83]{hod2}
Bernard~R. Hodgson.
\newblock D{\'e}cidabilit{\'e} par automate fini.
\newblock {\em Annales des sciences math{\'e}matiques du Qu{\'e}bec},
  7(1):39--57, 1983.

\bibitem[HS65]{Hew}
Edwin Hewitt and Karl Stromberg.
\newblock {\em Real and abstract analysis. {A} modern treatment of the theory
  of functions of a real variable}.
\newblock Springer-Verlag, New York, 1965.

\bibitem[KL94]{ki}
Haseo Ki and Tom Linton.
\newblock Normal numbers and subsets of {$\bold N$} with given densities.
\newblock {\em Fund. Math.}, 144(2):163--179, 1994.

\bibitem[KM14]{KnMcind}
J.~F. Knight and C.~McCoy.
\newblock Index sets and {S}cott sentences.
\newblock {\em Arch. Math. Logic}, 53(5--6):519--524, 2014.

\bibitem[KN94]{KhoussainovN94}
Bakhadyr Khoussainov and Anil Nerode.
\newblock Automatic presentations of structures.
\newblock In {\em Logical and Computational Complexity. Selected Papers. Logic
  and Computational Complexity, International Workshop {LCC} '94, Indianapolis,
  Indiana, USA, 13--16 October 1994}, pages 367--392, 1994.

\bibitem[KN08]{KhoussainovN08}
Bakhadyr Khoussainov and Anil Nerode.
\newblock Open questions in the theory of automatic structures.
\newblock {\em Bulletin of the {EATCS}}, 94:181--204, 2008.

\bibitem[Ko91]{Ko}
Ker-I Ko.
\newblock {\em Complexity theory of real functions}.
\newblock Progress in Theoretical Computer Science. Birkh\"{a}user Boston,
  Inc., Boston, MA, 1991.

\bibitem[Kon04]{Kon}
Michal Kone\v{c}n\'{y}.
\newblock Real functions incrementally computable by finite automata.
\newblock {\em Theoret. Comput. Sci.}, 315(1):109--133, 2004.

\bibitem[KST19]{kawamurafeeb}
Akitoshi Kawamura, Florian Steinberg, and Holger Thies.
\newblock Second-order linear-time computability with applications to
  computable analysis.
\newblock In {\em Theory and applications of models of computation}, volume
  11436 of {\em Lecture Notes in Comput. Sci.}, pages 337--358. Springer, Cham,
  2019.

\bibitem[Lis89]{Lisovik2}
L.~P. Lisovik.
\newblock Logical properties of partial continuous functions.
\newblock {\em Trudy Inst. Mat. (Novosibirsk)}, 12(Mat. Logika i Algoritm.
  Probl.):39--72, 189, 1989.

\bibitem[Lis98]{Lisovik3}
L.~P. Lisovik.
\newblock Applications of finite transducers for specification of mappings and
  fractal sets.
\newblock {\em Dokl. Akad. Nauk}, 358(1):19--21, 1998.

\bibitem[LS98]{Lisovik1}
L.~P. Lisovik and O.~Yu. Shkaravskaya.
\newblock Real functions defined by transducers.
\newblock {\em Cyber. Sys. Anal.}, 34(1):69--76, 1998.

\bibitem[Mea55]{Mealy}
George Mealy.
\newblock A method for synthesizing sequential circuits.
\newblock {\em Bell System Technical Journal}, 34(5):1045--1079, 1955.

\bibitem[Mel18]{Pontr}
Alexander Melnikov.
\newblock Computable topological groups and {P}ontryagin duality.
\newblock {\em Trans. Amer. Math. Soc.}, 370(12):8709--8737, 2018.

\bibitem[MN13]{CompComp}
Alexander~G. Melnikov and Andr{\'e} Nies.
\newblock The classification problem for compact computable metric spaces.
\newblock In {\em The nature of computation}, volume 7921 of {\em Lecture Notes
  in Comput. Sci.}, pages 320--328. Springer, Heidelberg, 2013.

\bibitem[MN16]{MelNg}
Alexander~G. Melnikov and Keng~Meng Ng.
\newblock Computable structures and operations on the space of continuous
  functions.
\newblock {\em Fund. Math.}, 233(2):101--141, 2016.

\bibitem[Mul94]{MullerTransducer}
Jean-Michel Muller.
\newblock Some characterizations of functions computable in on-line arithmetic.
\newblock {\em IEEE Trans. Comput.}, 43(6):752--755, 1994.

\bibitem[Myh71]{Myhill}
J.~Myhill.
\newblock A recursive function, defined on a compact interval and having a
  continuous derivative that is not recursive.
\newblock {\em Michigan Math. J.}, 18:97--98, 1971.

\bibitem[Nie10]{Nies:2010}
Andr{\'e} Nies.
\newblock Interactions of computability and randomness.
\newblock In {\em Proceedings of the {I}nternational {C}ongress of
  {M}athematicians. {V}olume {II}}, pages 30--57, New Delhi, 2010. Hindustan
  Book Agency.

\bibitem[NS15]{NSlocal}
Andr{\'{e}} Nies and Slawomir Solecki.
\newblock Local compactness for computable polish metric spaces is
  {$\Pi_1^1$}-complete.
\newblock In {\em Evolving Computability, 11th Conference on Computability in
  Europe, CiE 2015, Bucharest, Romania, June 29 -- July 3, 2015. Proceedings},
  pages 286--290, 2015.

\bibitem[PER83]{PRich}
Marian~Boykan Pour-El and Ian Richards.
\newblock Computability and noncomputability in classical analysis.
\newblock {\em Trans. Amer. Math. Soc.}, 275(2):539--560, 1983.

\bibitem[PER89]{PourElRich}
Marian~B. Pour-El and J.~Ian Richards.
\newblock {\em Computability in analysis and physics}.
\newblock Perspectives in Mathematical Logic. Springer-Verlag, Berlin, 1989.

\bibitem[Rou19]{Roughgarden:19}
Tim Roughgarden.
\newblock Beyond worst-case analysis.
\newblock {\em Communications of the ACM}, 62(3):88--96, 2019.

\bibitem[RS06]{MR2264276}
Illya~I. Reznykov and Vitaliy~I. Sushchansky.
\newblock On the 3-state {M}ealy automata over an {$m$}-symbol alphabet of
  growth order {$[n^{\log n/2\log m}]$}.
\newblock {\em J. Algebra}, 304(2):712--754, 2006.

\bibitem[SN07]{srimani}
P.~K. Srimani and S.~F.~B. Nasir.
\newblock {\em Transducers}, pages 270--303.
\newblock Foundation Books, 2007.

\bibitem[Soa87]{Soa}
R.~Soare.
\newblock {\em Recursively enumerable sets and degrees}.
\newblock Perspectives in Mathematical Logic. Springer-Verlag, Berlin, 1987.
\newblock A study of computable functions and computably generated sets.

\bibitem[TE77]{triverdi}
Kishor~S. Trivedi and Milo\v{s}~D. Ercegovac.
\newblock On-line algorithms for division and multiplication.
\newblock {\em IEEE Trans. Comput.}, C-26(7):681--687, 1977.

\bibitem[Tura]{Turing:36}
Alan~M. Turing.
\newblock On computable numbers, with an application to the
  entscheidungsproblem.
\newblock {\em Proceedings of the London Mathematical Society}, 42:230--265.

\bibitem[Turb]{Turing:37}
Alan~M. Turing.
\newblock On {C}omputable {N}umbers, with an {A}pplication to the
  {E}ntscheidungsproblem. {A} {C}orrection.
\newblock {\em Proceedings of the London Mathematical Society}, 43:544--546.

\bibitem[Wei00]{Wei00}
Klaus Weihrauch.
\newblock {\em Computable analysis}.
\newblock Texts in Theoretical Computer Science. An EATCS Series.
  Springer-Verlag, Berlin, 2000.
\newblock An introduction.

\bibitem[Wes14]{Westr}
Linda~Brown Westrick.
\newblock A lightface analysis of the differentiability rank.
\newblock {\em J. Symb. Log.}, 79(1):240--265, 2014.

\end{thebibliography}

\end{document}